\documentclass[12pt]{article}

\pdfoutput=1

\usepackage{latexsym,amsfonts,bm,epsfig,amsmath,natbib,authblk,amsthm,thmtools,amssymb}

\usepackage{float}
\usepackage{booktabs}
\usepackage{graphicx}
\usepackage[margin=.5cm]{caption}
\usepackage{xr-hyper}
\usepackage{color}
\usepackage[colorlinks,linkcolor=black,citecolor=black,filecolor=black,bookmarks=false,pagebackref]{hyperref}

\usepackage{pgfplots}
\usepackage{tikz}
\newlength\figureheight
\newlength\figurewidth
\usepackage{pdflscape}
\usepackage{pdfpages}

\floatstyle{plain}
\newfloat{Algorithm}{thp}{lop}
\floatname{Algorithm}{Algorithm}

\usepackage[plain,noend]{algorithm2e}

\makeatletter
\renewcommand{\algocf@captiontext}[2]{#1\algocf@typo. \AlCapFnt{}#2} 
\def\@algocf@capt@plain{top}
\renewcommand{\algocf@makecaption}[2]{%
  \addtolength{\hsize}{\algomargin}%
  \sbox\@tempboxa{\algocf@captiontext{#1}{#2}}%
  \ifdim\wd\@tempboxa >\hsize
   \hskip .5\algomargin%
    \parbox[t]{\hsize}{\algocf@captiontext{#1}{#2}}
 \else%
   \global\@minipagefalse%
   \hbox to\hsize{\box\@tempboxa}
 \fi%
  \addtolength{\hsize}{-\algomargin}%
}
\makeatother

\def\tr{\text{\rm tr}}

\newcommand{\ignore}[1]{}

\newcommand{\y}{{y}}
\newcommand{\z}{z}
\newcommand{\dt}{\mathrm{d}}
\newcommand{\plim}{\operatornamewithlimits{plim}} 
\newcommand{\E}{\mathbb{E}}
\def \d {\mathrm{d}}

\newtheorem{assumption}{Assumption}
\newtheorem{theorem}{Theorem}

\newtheorem{corollary}{Corollary}
\newtheorem{lemma}{Lemma}
\theoremstyle{remark}
\newtheorem{remark}{Remark}

\begin{document}
\def\spacingset#1{\renewcommand{\baselinestretch}%
	{#1}\small\normalsize} \spacingset{1}

\title{Synthetic likelihood in misspecified models}
\date{\empty}
\author[1]{David T. Frazier\thanks{Corresponding author:  david.frazier@monash.edu}}
\author[2,3]{David J. Nott}
\author[4]{Christopher Drovandi}
\affil[1]{Department of Econometrics and Business Statistics, Monash University, Clayton VIC 3800, Australia}
\affil[2]{Department of Statistics and Applied Probability, National University of Singapore, Singapore 117546}
\affil[3]{Operations Research and Analytics Cluster, National University of Singapore, Singapore 119077}
\affil[4]{School of Mathematical Sciences, Queensland University of Technology, Brisbane 4000 Australia}

\maketitle

\begin{abstract}
	Bayesian synthetic likelihood is a widely used approach for conducting Bayesian analysis in complex models where evaluation of the likelihood is infeasible but simulation from the assumed model is tractable. We analyze the behaviour of the Bayesian synthetic likelihood posterior when the assumed model differs from the actual data generating process. We demonstrate that the Bayesian synthetic likelihood posterior can display a wide range of non-standard behaviours depending on the level of model misspecification, including multimodality and asymptotic non-Gaussianity. Our  results suggest that likelihood tempering, a common approach for robust Bayesian inference, fails for synthetic likelihood whilst recently proposed robust synthetic likelihood approaches can ameliorate this behavior and deliver reliable posterior inference under model misspecification. All results are illustrated using a simple running example. 
\vspace{1cm}

\noindent \textbf{Keywords.}  {Synthetic likelihood}. Approximate Bayesian computation. Model misspecification. Likelihood tempering. 
\end{abstract}
\spacingset{1.9} 

\section{Introduction \label{sec:introduction}}
Approximate Bayesian methods, sometimes called likelihood-free methods, have become a common approach to conduct Bayesian inference in situations where the likelihood function is intractable. Two of the most prominent statistical methods in this paradigm are approximate Bayesian computation, see  \cite{marinea2012} for a review and \cite{sisson2018handbook} for a handbook treatment, and the method of synthetic likelihood (\citealp{wood2010statistical}). Synthetic likelihood-based inference is often conducted by placing a prior over the unknown model parameters and using Markov chain Monte Carlo methods to sample the resulting posterior. Throughout the remainder we refer to such methods as Bayesian synthetic likelihood, and refer to \cite{price2018bayesian} for an introduction. 

The goal of both approximate Bayesian computation (ABC) and Bayesian synthetic likelihood (BSL) is to conduct inference on the unknown model parameters by simulating summary statistics under the assumed model and matching them against  observed summaries calculated from the data. {\color{black}The simulated summary statistics are used to construct an estimate of the likelihood, which is then used to conduct posterior inference on the model unknowns.} While ABC implicitly constructs a nonparametric estimate of the likelihood for the summaries, synthetic likelihood uses a Gaussian approximation with an estimated mean and variance.

As these methods have grown in prominence, much research has been conducted to understand the benefits and disadvantages of different approximate Bayesian approaches. In terms of statistical regularity, i.e., large sample behavior, the method of approximate Bayesian computation behaves quite regularly in the context of correct model specification (see, e.g., \citealp[]{LF2016}, and \citealp{FMRR2016}), with \cite{frazier2019bayesian} demonstrating that BSL displays similar large sample behavior to ABC, while scaling to higher-dimensional summaries more easily than simple implementations of ABC. For a more in-depth comparison of ABC and BSL, we refer to Section 3.1.3 of \cite{martin2023approximating}.

{The goal of both methods is to conduct inference in models that are so complicated that the resulting likelihood is intractable. However, models are only ever an approximation of reality and, thus, correct specification is unlikely.  Hence, for a diverse collection of summary statistics, it is unlikely that the assumed model can exactly match all the summaries calculated from the observed data; with this problem likely exacerbated in early phases of model exploration, design, and formulation. When the summaries simulated under the assumed model cannot match the observed summaries for any value of the unknown model parameters, we say that the model is misspecified. This notion of misspecification is consistent with previous analyses of misspecification in ABC (see, e.g., \citealp{marin2014relevant}, and \citealp{frazier2020model}). }

Several authors have now discussed the impacts of model misspecification in likelihood-based Bayesian inference (see, e.g., \citealp{kleijn2012}, \citealp{miller2015}, \citealp{bhattacharya2019bayesian}), and it is known that ABC posteriors display non-standard asymptotic behaviors in misspecified models (\citealp{frazier2020model}). Given the links between ABC and BSL, it is critical for us to understand the behavior of the BSL posterior in misspecified models. However, we are unaware of any research that rigorously characterises the behavior of BSL under model misspecification. 

The need to theoretically examine the behavior of the BSL posterior in misspecified models is also motivated by the empirical analysis carried out in \cite{frazier2019robust}, where the authors present an empirical example showing that the BSL posterior displays non-standard behavior at a small sample size. Critically, \cite{frazier2019robust} did not explore whether this behavior abated as the sample size increased, or whether it was an artefact of their Monte Carlo approximation; nor did the authors explore the mechanism causing this non-standard posterior behavior, or present any theoretical results on the behavior of the BSL posterior \textit{in misspecified models}.

{In this manuscript we make four contributions to the literature on approximate Bayesian methods, and BSL methods more particularly. Our first contribution is to formally demonstrate that the BSL posterior is sensitive to model misspecification: depending on the nature of the model misspecification,  we prove that the BSL posterior can display non-standard asymptotic concentration or standard (i.e., Gaussian) concentration. These results deviate from those in correctly specified models, where \cite{frazier2019bayesian} demonstrate that the BSL posterior displays standard concentration.}\footnote{We later clarify that even the proof techniques used in \cite{frazier2019bayesian} are not applicable in \textit{misspecified models}, and that the posterior concentration results in \cite{frazier2019robust} do not extend to misspecified models.}

Our second contribution is to categorise the wide behaviours that the BSL posterior can present in misspecified models, which we illustrate empirically using a running example, while also verifying the technical conditions needed for our theoretical results in this example. As part of this analysis, we significantly extend the initial findings of \cite{frazier2019robust} by empirically demonstrating novel behaviours that the BSL posterior exhibits in misspecified models, including concentration onto a boundary point, multi-modality, and regions of posterior flatness. Critically, our theoretical analysis demonstrates that the non-standard behavior of the BSL posterior is not caused by Monte Carlo approximations, or a small sample issue, but is driven by the asymptotic behavior of the synthetic likelihood. 

The third contribution of this work is to highlight the differing behaviour of the BSL and ABC posteriors in misspecified models. In contrast to the case of correct model specification, the ABC and BSL posteriors concentrate onto different points in the parameter space when the model is misspecified. Our final contribution is an in-depth comparison of three possible approaches for dealing with model misspecification when conducting likelihood-free Bayesian inference. In this comparison, we theoretically demonstrate that a popular approach to robust Bayesian inference, likelihood tempering, does not ameliorate the non-standard behavior of the BSL posterior. However, we show empirically and theoretically that certain ``robust'' BSL approaches deliver reliable approximate Bayesian inference even when the model is misspecified. In particular, we propose a novel modification of the BSL adjustment procedure in \cite{frazier2019bayesian} that can adequately handle model misspecification, and we formally demonstrate that this procedure delivers valid uncertainty quantification in misspecified models.

{To motivate our analysis, we first illustrate the sensitivity of the BSL posterior to model misspecification in a simple example.}

\subsection*{Running example: moving average model of order one}\label{sec:maexamintro}
The researcher believes the observed data $y=(y_1,\dots,y_n)^\top$ is generated according to a moving average model of order one
\begin{equation}
y_{t}=e_{t}+\theta _{}e_{t-1} \quad( t=1,\dots,n), \label{MA2}
\end{equation}with $e_t$ independent and identically distributed standard normal, and our prior beliefs are uniform over $\theta\in[-1,1]$. We take as summary statistics the
sample autocovariances $\gamma _{j}(y_{1:n})=\frac{1}{n}\sum_{t=1+j}^{n}y_{t}y_{t-j}$, for $j\in\{0,1\}$, and let $S_n\left( y\right) =(\gamma_0(y_{1:n}), \gamma_1(y_{1:n}))^\top$ denote the observed summaries we will use to conduct inference on $\theta$. 

While the assumed model is \eqref{MA2}, the data actually evolves according to a stochastic volatility model: for $0<\rho<1$, $0<\sigma_{v}<1$, and $u_{t}$,  $v_{t}$ independent standard normal errors 
\begin{flalign}\label{trueDGP}
y_{t}=\exp(h_{t}/2)u_{t},\;\;h_{t}=\omega+\rho h_{t-1}+v_{t}\sigma_{v}\quad( t=1,\dots,n).
\end{flalign}Under the process in \eqref{trueDGP}, the assumed model in \eqref{MA2} is misspecified, however, for any value of $\omega,\rho,\sigma_v$ above, the population autocovariances are zero. Therefore, \textit{a priori} we expect the Bayesian synthetic likelihood posterior for $\theta$ to have significant mass near $\theta=0$, as this yields simulated data with no autocorrelation, and would most closely ``match'' the observed summaries. Throughout the remainder, unless otherwise stated, we use the term posterior to refer to the Bayesian synthetic likelihood posterior.

{\color{black}We generate data  from the model in \eqref{trueDGP} with parameter values $\omega = -0.10$, $\rho = 0.90$ and $\sigma_v = 0.40$}, which produces a series that displays many of the same features as monthly asset returns, and we consider three different sample sizes: $n=100,500,1000$. For each sample size and dataset, we plot the exact posterior in Figure \ref{fig:ma2_theta}; we refer to Supplementary Appendix E.1 
for further details regarding construction of the exact posterior in this example. 

{\color{black}At a small sample size ($n=100$), the posterior appears to be well-behaved with a single mode around $\theta=0$. However, as the sample size increases the posterior becomes bi-modal with well-separated modes of nearly equal height that both lie in the interior of the parameter space ($n=1000$).\footnote{We note that the behavior observed in Figure \ref{fig:ma2_theta} is in stark contrast to the findings in \cite{frazier2019robust}, where the authors found that, under a different parameterization of the same model, at small sample sizes ($n=100$) the BSL posterior placed most of its mass near the boundary of the parameter space. } 
The emergence of this bi-modality as the sample size increases signals the presence of a non-standard asymptotic phenomena, and suggests that the posterior \textit{will not concentrate onto a single point}. Moreover, in this example the normality of the summary statistics is reasonable: both summaries can be verified to satisfy a central limit theorem under the process in \eqref{trueDGP}. This behavior is surprising, and worrisome, given that the value of $\theta$ that (asymptotically) minimizes the Euclidean distance between the observed and simulated summaries is $\theta=0$. While $\theta=0$ ensures that the simulated summaries are as close as possible to the observed summaries, the posterior has little mass near this point at large samples sizes. Instead, at larger sample sizes the posterior gives the impression that we require meaningful autocorrelation to match the observed summaries, when in fact the observed data has no autocorrelation. 
}

In the remainder of the paper, we elaborate on the above behavior and characterize the asymptotic  behavior of the posterior in misspecified models. The remainder of the paper is organized as follows. In Section \ref{sec:two}, we discuss the relevant concept of model misspecification in synthetic likelihood. In Section \ref{sec:asymp}, we characterize the asymptotic behavior of the BSL posterior in misspecified models and demonstrate that the posterior may be asymptotically non-Gaussian. We also compare this theoretical behavior to what is obtained in the case of ABC. In Section \ref{sec:adjust}, we obtain new insights into approaches aimed at dealing with model misspecification. Section \ref{sec:discuss} concludes. Proofs of all results are contained in the Supplementary Appendix. 

\begin{figure}[h]
	\centering
	\setlength\figureheight{3.5cm} 
	\setlength\figurewidth{3.5cm} 
\include{fig1_MA1_R1} 
\vspace{-2cm}
	\caption{{Bayesian synthetic likelihood posterior for $\theta$ in the misspecified moving average model.}}
	\label{fig:ma2_theta}
\end{figure}

\section{Synthetic likelihood and model misspecification}\label{sec:two}

Let $\y=(y_1,\dots,y_n)^{\top}$ denote the observed data and define $P^{(n)}_{_\bullet}$ as the true distribution of $\y$. 
The observed data is assumed to be generated from a class of parametric models  $\{P^{(n)}_\theta:\theta\in\Theta\subseteq\mathbb{R}^{d_\theta}\}$ for which the likelihood function is intractable, but from which
we can easily simulate pseudo-data $\z$ for any $\theta\in\Theta$. Let $\Pi $ denote the prior measure for $\theta$ and $\pi({\theta })$ its density. 

Since the likelihood function is intractable, we conduct inference using approximate Bayesian methods. The main idea is to search for values of $\theta$ that produce pseudo-data $\z$ which is ``close enough'' to $\y$, and then retain these values to build an approximation to the posterior. To make the problem computationally practical, the comparison is generally carried out using summaries of the data. Let $S_n:\mathbb{R}^n\rightarrow\mathbb{R}^d$, $d\ge d_\theta$,  denote  the vector summary statistic mapping used in the analysis. Where there is no confusion, we write $S_n$ for the mapping or its value when evaluated at the observed data $\y$.

The method of BSL approximates the intractable distribution of $S_n(\y)\mid \theta$ using a Gaussian distribution with mean $b(\theta)=\mathbb{E}\{S_n(z)\mid \theta\}$ and variance $\Sigma_n(\theta)=\text{var}\{S_n(z)\mid \theta\}$,  both of which are calculated under $P^{(n)}_\theta$. {The map $\theta\mapsto \mathbb{E}\{S_n(z)\mid \theta\}$ may technically depend on $n$, however, if the data are iid or weakly dependent, and $S_n$ takes the form of an average, then $\mathbb{E}\{S_n(z)\mid \theta\}$ will not depend on $n$ in any meaningful way. As the majority of summaries used in BSL take the form of averages, it is reasonable to neglect the potential dependence on $n$. We also note that the notation $b(\theta)$ has also been used in several papers on ABC and BSL; see, e.g., \cite{FMRR2016}, and \cite{frazier2019bayesian}. }  

The synthetic likelihood is denoted as $N\{S_n;b(\theta),\Sigma_n(\theta)\}$, where $N(x;\mu,\Sigma)$ is the normal density function evaluated at $x$ with mean $\mu$ and covariance matrix $\Sigma$.  In typical applications
$b(\theta)$ and $\Sigma_n(\theta)$ are unknown, and are
estimated using the sample mean $\widehat{b}_n(\theta)$ and sample variance $\widehat{\Sigma}_n(\theta)$
calculated using $m$ independent simulated datasets. {These sample quantities are depicted as $n$-dependent, rather than $m$-dependent, as we later take $m$ to diverge as $n$ diverges. }

\cite{wood2010statistical} and \cite{price2018bayesian} suggest exploring the estimated synthetic likelihood  $N\{S_n;\widehat{b}_n(\theta),\widehat{\Sigma}_n(\theta)\}$ and obtaining point estimates of $\theta$ using Markov chain Monte Carlo. Following \cite{andrieu2009pseudo}, the use of $N\{S_n;\widehat{b}_n(\theta),\widehat{\Sigma}_n(\theta)\}$ within Markov chain Monte Carlo results in draws from the target posterior
\begin{flalign*}
\widehat{\pi}(\theta\mid S_n)&\propto {\pi(\theta)\widehat{g}_n(S_n\mid\theta)}
,
\\	\widehat{g}_n(S_n\mid\theta) &= \int N\{S_n;\widehat{b}_n(\theta),\widehat{\Sigma}_n(\theta)\} \prod_{i=1}^m \dt P^{(n)}_\theta\{S_n(\z^i)\}\,\dt S_n(\z^1)\,\dots\, \dt S_n(\z^m),
\end{flalign*}where $\widehat{g}_n(S_n\mid\theta)$ is the expectation of the estimated synthetic likelihood. {See \cite{price2018bayesian} and \cite{frazier2019bayesian} for a  discussion on the connection between pseudo-marginal methods and BSL.} In contrast, if $b(\theta)$ and $\Sigma_n(\theta)$ were known, we could target the exact posterior
\begin{equation}
{\pi}(\theta\mid S_n)\propto{\pi(\theta)N\{S_n;{b}_{}(\theta),{\Sigma}_n(\theta)\}}
.\label{eq:BSL_post_exact}
\end{equation}

\subsection{Model misspecification}\label{sec:misspec}
While Bayesian synthetic likelihood is based on a likelihood, it is not a likelihood for $\y$ but for $S_n(\y)$, and this ``likelihood'' is itself a normal approximation of the sampling distribution for the summaries. As such, interpreting the impact of model misspecification requires us to consider the loss of information that results from replacing the data by the summaries, as well as the use of  an approximation for the likelihood of the summaries. To cultivate intuition regarding the impact of these approximations when the model is misspecified, we explore these approximations when the mean and variance of the summaries are known. {{\color{black}The same general conclusions will follow in the case where the synthetic likelihood is estimated, but does not seem to add additional insights.}

Let $f_{_\bullet}^{(n)}$ denote the density function for the summary statistics $S_n(y)$ under $P^{(n)}_{_\bullet}$. The Kullback-Leibler divergence between the synthetic likelihood, $N\{\cdot;b(\theta),\Sigma_n(\theta)\}$, and the density function of the summaries, $f_{_\bullet}^{(n)}$, is 
\begin{flalign*}
\textsc{kl}[f^{(n)}_{_\bullet}(\cdot) \| N\{\cdot;b(\theta),\Sigma_n(\theta)\}]=&\int \log  \left[\frac{f_{_\bullet}^{(n)}(s)}{N\{s;b(\theta),\Sigma_n(\theta)\}}\right] f_{_\bullet}^{(n)}(s)\dt s\\= &\frac{1}{2}\log\left\{|{\Sigma}_n(\theta)|\right\}+\frac{1}{2}\int \left\{s-b(\theta)\right\}^{\top}\Sigma_n(\theta)^{-1}\left\{s-b(\theta)\right\}f_{_\bullet}^{(n)}(s)\dt s+C,
\end{flalign*}where $C$ is a constant that does not depend on $\theta$. 
For $b_{\bullet}=\int s f_{_\bullet}^{(n)}(s)\dt s$, and $V=\int \left(s-b_{\bullet}\right)\left(s-b_{\bullet}\right)^{\top}f_{_\bullet}^{(n)}(s)\dt s$, using properties of quadratic forms,
\begin{flalign*}
\textsc{kl}[f^{(n)}_{_\bullet}(\cdot) \| N\{\cdot;b(\theta),\Sigma_n(\theta)\}]&=\frac{1}{2}\log\left\{|{\Sigma}_n(\theta)|\right\}+\frac{1}{2}\text{tr}\left\{\Sigma_n(\theta)^{-1}V\right\}\\&+\frac{1}{2}\{b(\theta)-b_{\bullet}\}^{\top}\Sigma_n(\theta)^{-1}\{b(\theta)-b_{\bullet}\}+C .
\end{flalign*}

Thus, outside of cases where $f_{_\bullet}^{(n)}$ \textit{is a Gaussian density, the synthetic likelihood is always misspecified.} In addition, the asymptotic behavior of the Kullback-Leibler divergence is governed by the term $\{b(\theta)-b_{\bullet}\}^{\top}\Sigma_n(\theta)^{-1}\{b(\theta)-b_{\bullet}\}.$ Since the summaries are generally an average, $\Sigma_n^{}(\theta)$ is  generally of order $n^{-1}$,  and when $\Sigma_n^{}(\theta)$ is positive-definite it will be the case that 
$$
c_1\|{\sqrt{n}}\{b(\theta)-b_{\bullet}\}\|^2\le \|\Sigma_n(\theta)^{-1/2}\{b(\theta)-b_{\bullet}\}\|^2\le  c_2\|{\sqrt{n}}\{b(\theta)-b_{\bullet}\}\|^2,
$$for some $0<c_1\le c_2<\infty$. Therefore, if there exists no $\theta\in\Theta$ such that $b(\theta)=b_{\bullet}$ as $n\rightarrow\infty$, then $$\inf_{\theta\in\Theta}\textsc{kl}[f^{(n)}_{_\bullet}(\cdot) \| N\{\cdot;b(\theta),\Sigma_n(\theta)\}]\rightarrow\infty.$$

The above shows that the meaningful concept of model misspecification in synthetic likelihood is that there does not exist any $\theta\in\Theta$ such that $b_{}(\theta)=b_{\bullet}$. This condition is called model incompatibility by \cite{marin2014relevant}, and features in the literature on model misspecification in approximate Bayesian computation (\citealp{frazier2020model}). We then say that the model is {misspecified, or incompatible,} if
\begin{flalign}\label{eq:bslmiss}
\lim_{n \rightarrow \infty}\inf_{\theta\in\Theta}\{b_{}(\theta)-b_{\bullet}\}^{\top}\left\{n\Sigma_n(\theta)\right\}^{-1}\{b_{}(\theta)-b_{\bullet}\}>0.
\end{flalign}
Throughout the remainder, model misspecification  is interpreted in terms of equation \eqref{eq:bslmiss}.

\subsection{Consequences of model misspecification}\label{sec:maexam}

We briefly demonstrate the consequences of model misspecification in BSL by returning to the simple running example (see Section \ref{sec:maexamintro} for details). Depending on the level of model misspecification, the posterior can display Gaussian-like posterior concentration,  bi-modality, and/or concentration onto the boundary of the parameter space. 

\subsubsection*{Example: Moving Average model}\label{sec:maexam}
The researcher believes $y_1,\dots,y_n$ is generated according to an MA(1) model, see equation \eqref{MA2}, and our prior beliefs are uniform over $[-1,1]$. The summary statistics are
$\gamma_j(y_{1:n})=\frac{1}{n}%
\sum_{t=1+j}^{n}y_{t}y_{t-j}$, for $j\in\{0,1\}$, and $S_n\left( y\right) =(\gamma_0(y_{1:n}), \gamma_1(y_{1:n}))^\top$. In this example, the mean and variance of the summaries can be calculated exactly, with these quantities then used to construct the exact Bayesian synthetic likelihood posterior. The mean of the summaries is
$
b(\theta)=\mathbb{E}\{S_n(z_{1:n})|\theta\}=\left(1+\theta^2,\theta\right)^{\top}. 
$ 

The variance of the summaries also has a closed-form, and can be derived using the results of \cite{de1981investigation} on the variance and covariance of sample autocorrelations in autoregressive integrated moving average models.
Partitioning $\Sigma_n(\theta)$ as $$
\Sigma_n(\theta)=\begin{pmatrix}
\Sigma_{11,n}(\theta)&\Sigma_{12,n}(\theta)\\\Sigma_{12,n}(\theta)&\Sigma_{22,n}(\theta)
\end{pmatrix},
$$ the leading terms in the components of $\Sigma_n(\theta)$ are as follows:\footnote{The precise formulas are too long to state analytically. The interested reader is referred to the supplementary material where it is given in full detail.}  
\begin{flalign*}
\Sigma_{11,n}(\theta)&=(2/n^4)\left[n^3 \cdot (1+\theta^2)^2+2 \cdot n^2 \cdot (n-1) \cdot \theta^2\right]+O(n^{-2})\\
\Sigma_{22,n}(\theta)&=(1/n^2)\left[(n-1) \cdot ((1+\theta^2)^2+\theta^2)+2 \cdot (n-2) \cdot \theta^2\right]+O(n^{-2})\\
\Sigma_{12,n}(\theta)&=(2/n^4)\left[n^2 \cdot ( (n-1) \cdot (2 \cdot (1+\theta^2) \cdot \theta))\right]+O(n^{-2}).
\end{flalign*}From this representation, it is clear that each term has a dominant $O(n^{-1})$ term, and that $n\Sigma_n(\theta)$ is positive-definite for all $\theta\in[-1,1]$.

Recall that the actual data generating process (DGP) for $y_{1:n}$ evolves according to the stochastic volatility (SV) model  in equation \eqref{trueDGP}. Under this DGP the summaries $S_n(\y)$ converge in probability towards
\begin{equation*}
b_{\bullet}:=(b_{\bullet,0},b_{\bullet,1})^\top=\begin{pmatrix}\exp\left( \frac{\omega}{1-\rho}+\frac{1}{2}\frac{\sigma_v^2}{1-\rho^2}\right)
,&0
\end{pmatrix}^{\top}.
\end{equation*}Therefore, if for given values of $\omega,\sigma_v$ and $\rho$ there does not exist a value of $\theta$ such that $$\exp\left\{ {\omega}/{(1-\rho)}+\frac{1}{2}{\sigma_v^2}/{(1-\rho^2)}\right\}=1+\theta^2,$$ we cannot match the first summary, and the assumed model is misspecified. When $0<b_{\bullet,0}\le1$, the unique minimum of $\|b(\theta)-b_{\bullet}\|$ is achieved at $\theta=0$, and it is this value onto which we would hope the BSL posterior would concentrate asymptotically. However, as we have already seen from Figure \ref{fig:ma2_theta}, for certain values of $\omega,\sigma_v,\rho$, {\color{black}and depending on the sample size, the posterior can be bi-modal.} 

{\color{black}To help explain this phenomena, we now analyze the posterior across various levels of model misspecification by fixing the value of the observed summaries $S_n(\y)$ at its limit $b_{\bullet}=(b_{\bullet,0},b_{\bullet,1})^\top$, and by changing the value of $b_{\bullet,0}$. To this end, we plot the posteriors for three values of $n=100,500,1000$, and across six different values of the first summary statistic $b_{\bullet,0}\in\{0.01,0.10,0.25,0.50,0.75,0.99\}  $. These values represent a situation of significant misspecification, at $b_{\bullet,0}=0.01$, tending towards no misspecification, $b_{\bullet,0}=0.99$. We plot the resulting posteriors graphically in Figure \ref{fig2}. The results demonstrate that the behavior of the posterior varies markedly as the level of model misspecification changes. 

At small samples sizes $(n=100)$ the posteriors are (nearly all) uni-modal with a mode whose location depends on the level of misspecification. However, for larger levels of misspecification, as the sample size increases the posterior concentrates mass on two distinct modes, with the heights of the two modes varying with the level of misspecification. 

The results in Figure \ref{fig2} are surprising, and show that depending on the level of model misspecification, the posterior can display concentration onto the boundary of the parameter space ($b_{\bullet,0}=0.01$); bi-modality with concentration occurring on the interior of the parameters space ($b_{\bullet,0}\in\{0.10,0.25\}$); a region of ``flatness'' ($b_{\bullet,0}=0.50$); and approximate Gaussianity ($b_{\bullet,0}\in\{0.75,0.99\}$). We speculate that the region of posterior flatness may be the result of two modes on either size of $\theta=0$, so that in neighbourhoods around $\theta=0$ the posterior appears flat.\footnote{We thank an anonymous referee for bringing this possible interpretation to our attention.} From a practical standpoint, however, whether the posterior is genuinely flat near $\theta=0$, or has two close modes on either side of $\theta=0$ is largely irrelevant as it would require a very large sample size to reliably distinguish between the two cases. {In Supplementary Appendix E.3 
we expand on the mechanisms causing this posterior  behavior, and give additional discussion on the cause of the posterior ``flatness'' observed in Figure \ref{fig2}.} }

{\color{black}At larger levels of model misspecification, the values onto which the exact posterior in \eqref{eq:BSL_post_exact} is concentrating are not at all related to the values of $\theta$ under which $\|b(\theta)-b_{\bullet}\|$ is small. In comparison, if one were to apply ABC based on $\|\cdot\|$ in the same example, the resulting posterior would be uni-modal and have the majority of its mass near the origin ($\theta=0$). This is due to the fact that when $0<b_{_\bullet,0}\le1$, the distance $\|b(\theta)-b_{\bullet}\|$ is uniquely minimized at $\theta=0$; hence, following Theorem 1 in \cite{frazier2020model}, the ABC posterior would concentrate mass onto $\theta=0$. }

{\color{black}The behavior observed in the left-most panel of Figure \ref{fig2} is similar to, but distinct from, the behavior documented by \cite{frazier2019robust} in the MA(1) model. In that work, for a sample size of $n$=100 the authors empirically showed that an importance sampling estimate of the posterior $\widehat{\pi}(\theta\mid S_n)$ placed nearly all of its mass near the boundaries of the parameter space, i.e., $\theta\pm1$. 
}	
{\color{black}In contrast to \cite{frazier2019robust}, all of the above results pertain to \textit{the exact posterior} $\pi(\theta\mid S_n)$ that results from using the exact mean and variance of the summaries. As such, Figure \ref{fig2} demonstrates the root cause of the irregular posterior behavior: the behavior is not due to small sample sizes, or Monte Carlo approximations, but is caused by the asymptotic (non-standard) behavior of the exact synthetic likelihood $N\{S_n;b(\theta),\Sigma_n(\theta)\}$. Furthermore, we remark that the behavior observed in this current example is vastly more diverse than the behavior observed by \cite{frazier2019robust}. It is the diversity of this behavior which we theoretically investigate in the following section. }

\begin{figure}[h!]
	\centering 
	\setlength\figureheight{3.5cm} 
	\setlength\figurewidth{3.5cm} 
	\include{sl_postma1_R2} 	
\vspace{-2cm}
	\caption{Comparison of the exact (synthetic likelihood) posterior under different levels of model misspecification. The solid line corresponds to $n=100$, the dashed line to $n=500$ and the dotted line to $n=1000$. } 
	\label{fig2} 
\end{figure}


\section{Asymptotic behavior of BSL}\label{sec:asymp}
We now characterize the behavior of $\widehat{\pi}(\theta\mid S_n)$ when the assumed model is misspecified. The following notations are used to make the results easier to state and follow.
For $x\in\mathbb{R}$, $| x | $ denotes the absolute value of $x$, and for $x\in\mathbb{R}^p$, $\|x\|$ denotes the Euclidean norm of $x$. For $A$ denoting a square matrix, we abuse notation and let $|A|$ denote the determinant of $A$ and $\|A\|$ any convenient matrix norm. The terms $\lambda_{\text{max}}(A)$ and $\lambda_{\text{min}}(A)$ denote the maximal and minimal eigenvalues of $A$. Throughout, $C$
denotes a generic positive constant that can change with each usage. For real-valued sequences $\{a_{n}\}_{n\geq 1}$ and
$\{d_{n}\}_{n\geq 1}$: $a_{n}\lesssim d_{n}$ implies $a_{n}\leq Cd_{n}$ for
some finite $C>0$ and all $n$ large; $a_{n}\asymp d_{n}$ implies $a_{n}\lesssim d_{n}$ and $d_n \lesssim a_{n}$. For $x_{n}$ a random variable, $x_{n}=o_{p}(a_{n})$ 
and $x_{n}=O_{p}(a_{n})$ have their usual definitions. Likewise, $\plim_{n\rightarrow\infty} x_n$ denotes the probability limit of $x_n$. All limits are taken as $n\rightarrow\infty$ so that when no confusion will result, we use $\lim_{}$ and $\plim$ to denote $\lim_{n\rightarrow\infty}$ and $\plim_{n\rightarrow\infty}$, respectively. The notation $\Rightarrow$ denotes weak convergence under $P^{(n)}_{_\bullet}$.  Proofs of all stated results are given in the Supplementary Appendix.

\subsection{Asymptotic behavior: multiple modes}\label{sec:theory}
Let $g_n(\cdot\mid\theta):=N\{\cdot;b(\theta),\Sigma_n(\theta)\}$ denote the synthetic likelihood with known mean and variance, and define its score and limit counterpart as
\begin{flalign*}
{M}_n(\theta)&=n^{-1}{\partial \log g_n(S_n\mid\theta)}/{\partial\theta}, \quad
M(\theta)=\plim_{}{M}_n(\theta).
\end{flalign*}Likewise, define the Hessian of $n^{-1}\log g_n(\cdot\mid\theta)$ and its limit counterpart as 
\begin{flalign*}
{H}_n(\theta)&=n^{-1}{\partial^2 \log g_n(S_n\mid\theta)}/{\partial\theta\partial\theta^\top}, \quad
H(\theta)=\plim_{}{H}_n(\theta).
\end{flalign*}
The posterior multi-modality in the running example can be traced back to the existence of multiple roots to the score equation $0=M_n(\theta)$, e.g., $\theta_{1,n}$ and $\theta_{2,n}$. In such cases, if $-H_n(\theta_{j,n})$ is positive-definite, for $j=1,2$, then the posterior will exhibit multiple modes (around $\theta_{1,n}$ and $\theta_{2,n}$). {\color{black}Let $\text{Int}\left(\Theta\right)$ denote the interior of $\Theta$, and let
$\Theta_\star:=\left\{\theta\in\text{Int}\left(\Theta\right):0=M(\theta)\right\}$ denote the collection of asymptotic roots. }

We maintain the following regularity conditions. 

\begin{assumption}\label{ass:three} (i) $\Theta\subset\mathbb{R}^{d_\theta}$ is compact; (ii) The map $\theta\mapsto \log g_n(\cdot\mid\theta)$ is twice continuously differentiable on $\mathrm{Int}(\Theta)$.
\end{assumption}

\begin{assumption}\label{ass:one}
	There exist $b_{\bullet}\in\mathbb{R}^d$, $d\ge d_\theta$, and a positive-definite matrix $V$ such that ${\sqrt{n}}\left(S_ { n }-b_{\bullet}\right)\Rightarrow N(0,V)$. 
\end{assumption}

\begin{assumption}\label{ass:limits} The set $\Theta_\star$ is non-empty and finite. For some $\delta>0$, at least one $\theta\in\Theta_\star$ satisfies $0<\delta\le\lambda_{\text{min}}\left\{-H(\theta)\right\}\le1/\delta$.
\end{assumption}
{\color{black}
\begin{remark}
		The compactness in Assumption \ref{ass:three}(i) and the smoothness condition in Assumption \ref{ass:three}(ii) ensure the existence of a solution to $0=M_n(\theta)$, and are standard regularity conditions employed in the analysis of frequentist point estimators (see, e.g., \citealp{jennrich1969asymptotic} for a classical reference, and Chapter 5 of \citealp{van2000asymptotic} for a textbook treatment). We note, however, that both conditions can be relaxed by requiring stronger conditions on $M_n(\theta)$ and $H_n(\theta)$. Assumption \ref{ass:one} is standard in the literature on approximate Bayesian methods, and requires that the observed summary statistics are asymptotically Gaussian with a well-behaved covariance matrix. Assumption \ref{ass:limits} restricts $\log g_n(\cdot\mid\theta)$ to have a finite collection of local maxima, all of which lie in $\mathrm{Int}(\Theta)$.\footnote{{The behavior of the posterior when a root is on the boundary of the parameter space can be complicated, and so we leave a detailed study of this situation for future research.}} Importantly, as illustrated in the simple running example, there is no reason to suspect that values in $\Theta_\star$ deliver small values of $\|b(\theta)-b_{\bullet}\|$. {\color{black}Recall that $b(\theta)$ denotes the (asymptotic) mean of the simulated summaries under $P_\theta^{(n)}$, while $b_{\bullet}$ denotes the (asymptotic) mean of the observed summaries under $P_{_\bullet}^{(n)}$.}
\end{remark}
}

\begin{lemma}\label{lem:cons}
	If Assumptions \ref{ass:three}-\ref{ass:limits}  are satisfied, then there exists $\theta_n$ in $\Theta$ that solves $0=M_n(\theta)$, and  $\|\theta_n-\theta_\star\|=o_p(1)$ for some $\theta_\star\in\Theta_\star$.
\end{lemma}

Consider that $M(\theta)$ has zeros $\theta_{1,\star}$ and $\theta_{2,\star}$, with $-H(\theta_{1,\star})$ and $-H(\theta_{2,\star})$ both positive-definite. Since both values satisfy the sufficient conditions in Lemma \ref{lem:cons}, it follows that $\theta_{1,n}=\theta_{1,\star}+o_p(1)$ and $\theta_{2,n}=\theta_{2,\star}+o_p(1)$. Consequently, we should expect that the  posterior assigns non-vanishing probability mass to both points. Therefore, if we wish to analyze $\widehat{\pi}(\theta\mid S_n)$ in misspecified models, we must restrict our attention to regions around roots of the synthetic likelihood score equation.

\begin{assumption}\label{ass:hess}For any $\theta_\star\in\Theta_\star$, and some $\delta>0$, for all  $\|\theta-\theta_\star\|\le\delta$, there exists a $K>0$ such that  $\|\partial^2 \log g_n(S_n\mid\theta)/\partial\theta_j\partial\theta_i-\partial^2 \log g_n(S_n\mid\theta_\star)/\partial\theta_j\partial\theta_i\|\le K\|\theta-\theta_\star\|$ for $i,j\in\{1,\dots,d_\theta\}.$
\end{assumption}

{\color{black}
	\begin{remark}
	The smoothness condition in Assumption \ref{ass:three} is a commonly used sufficient condition to obtain second-order approximations of frequentist criterion functions (see, e.g., Chapter 5 in \citealp{van2000asymptotic}), and ensures that $\log g_n(S_n\mid\theta)$ admits a valid expansion around each $\theta_\star\in\Theta_\star$. 
	Assumption \ref{ass:hess} gives sufficient regularity to ensure that the remainder term in such an expansion can be appropriately controlled. In this way, Assumptions \ref{ass:three}, and \ref{ass:hess} resemble commonly employed assumptions in frequentist point-estimation theory, and will be satisfied so long as the mappings $\theta\mapsto b(\theta)$ and $\theta\mapsto \Sigma_n(\theta)$ are smooth enough.\footnote{Given Assumption \ref{ass:three}, a sufficient condition for Assumption \ref{ass:hess} is that $\theta\mapsto b(\theta)$ and $\theta\mapsto \Sigma_n(\theta)$ are twice continuously differentiable.} However, the possible multi-modality of the limit synthetic likelihood (Assumption \ref{ass:limits}) requires that these assumptions hold at each $\theta_\star\in\Theta_\star$. The smoothness conditions in Assumptions \ref{ass:three} and \ref{ass:hess} are stronger than those employed by \cite{frazier2019bayesian} to analyze the posterior in correctly specified models. In Supplementary Appendix D.1,  
	 we show that the smoothness assumptions in \cite{frazier2019bayesian} are insufficient to deduce the behavior of the posterior in misspecified models, and we discuss why additional smoothness conditions are necessary in misspecified models.
\end{remark}
}

\begin{assumption}\label{ass:two}
	Let $A_n(\theta)$ denote either $\widehat{\Sigma}_n(\theta)$ or $\Sigma_{n}(\theta)$. For some $\delta>0$, any $\theta_\star\in\Theta_\star$, and all $\|\theta-\theta_\star\|\leq\delta$, the sequence of matrices $A_n(\theta)$ satisfy: (i) for all $n$ large enough, there exist positive constants $c_1\leq c_2$, such that $0<{c}_1\leq |nA_{n}({\theta })|\leq {c}_2<\infty$; (ii) there exists a matrix function $\Sigma(\theta)$ that is continuous over $\Theta$, is positive-definite for all $\|\theta-\theta_\star\|\le\delta$, and any $\theta_\star\in\Theta$, and is such that $\sup_{\theta\in\Theta}\left\|nA_n(\theta)-\Sigma(\theta)\right\|=o_{p}(1)$.
\end{assumption}

\begin{assumption}\label{ass:four}
	For any $\delta>0$, and any $\theta \in\Theta_\star$, $\pi(\theta_\star)>0$, and $\pi(\cdot)$ is continuous on $\{\theta\in\Theta:\|\theta-\theta_\star\|\leq\delta\}$.
\end{assumption}
{\color{black}
\begin{remark}
	Assumption \ref{ass:two} places sufficient regularity on the covariance matrix used in BSL to ensure the posterior asymptotically concentrates on $\Theta_\star$, and is nearly identical to the regularity conditions for $\Sigma_n(\theta)$ used by \cite{frazier2019bayesian} in \textit{correctly specified models} (see their {Assumption 3}). For a detailed discussion on Assumption \ref{ass:two}, we refer the interested reader to \cite{frazier2019bayesian}. 
	Assumption \ref{ass:four} is a standard regularity condition in the theoretical analysis of Bayesian posterior distributions for Euclidean valued parameters. 
\end{remark}
}
\begin{assumption}\label{ass:five}For all $\theta\in\Theta$, $\E\left\{\|S_{n}(\z)\|^4\mid \theta\right\}<\infty.$
\end{assumption}
{\color{black}
\begin{remark}
	Assumption \ref{ass:five} requires that the simulated summary statistics admit at least a finite fourth moment, which is required to ensure that the posterior $\widehat{\pi}(\theta\mid S_n)$ exists, and concentrates toward the exact posterior $\pi(\theta\mid S_n)$ as $m\rightarrow+\infty$. This condition is substantially weaker than the corresponding condition used in \cite{frazier2019bayesian}, which required that the summary statistics have a sub-Gaussian tail. Since the distribution of $S_n(z)$ is intractable, analytically verifying Assumption \ref{ass:five} is likely infeasible in complex models. However, it is relatively straightforward to verify Assumptions \ref{ass:three}, \ref{ass:hess}, \ref{ass:two} and \ref{ass:five} using simulation from the assumed model. We refer the interested reader to Supplementary Appendix E.2 
	for additional discussion of this point, and for verification of Assumptions \ref{ass:three}-\ref{ass:five} in the running example. 	
\end{remark}
}

To state our first result, let $\Delta=\{-H(\theta_\star)\}^{-1}$, let $t={\sqrt{n}}(\theta-\theta_n)$ be a local parameter, and $\widehat{\pi}(t\mid S_n)=\widehat{\pi}(\theta_n+t/\sqrt{n}\mid S_n)/\sqrt{n}^{d_\theta}$ the posterior for $t$. 

\begin{theorem}[Asymptotic shape]\label{thm:one} If Assumptions \ref{ass:three}-\ref{ass:five} are satisfied, then for each $\theta_\star\in\Theta_\star$ the following is satisfied: for any finite $\gamma>0$, and $m\rightarrow\infty$ as $n\rightarrow\infty$,
	$$
	\left|\int_{\|t\|\le\gamma}\widehat{\pi}(t\mid S_n)\dt t-\int_{\|t\|\le\gamma}\pi_\star(t)\dt t\right|=O_p(1/m),
	$$ for some density function $\pi_\star(t)\propto N\{t;0,\Delta\}$.
\end{theorem}

Consider again that $\Theta_\star$ contains only $\theta_{1,\star}$ and $\theta_{2,\star}$; Theorem \ref{thm:one} then demonstrates that in a shrinking neighborhood of $\theta_{1,\star}$ (respectively, $\theta_{2,\star}$) the posterior $\widehat{\pi}(t\mid S_n)$ is proportional to $N\{t;0,\Delta_1\}$ (respectively, $N\{t;0,\Delta_2\}$), where $\Delta_j=\{-H(\theta_{j,\star})\}^{-1}$. Thus, $\widehat{\pi}(t\mid S_n)$ is not asymptotically Gaussian, but is mixed Gaussian with modes near $\theta_{1,\star}$ and $\theta_{2,\star}$.

\begin{remark}	
	While multi-modality of the posterior $\widehat{\pi}(\theta\mid S_n)$ can result from model misspecification, nothing confines Theorem \ref{thm:one} exclusively to misspecified models. {\color{black}Consider that there are two values of $\theta$, say $\theta_{1,\star}$ and $\theta_{2,\star}$,  such that $b(\theta_{j,\star})=b_{\bullet}$.} In this case, the model is ``correctly specified'', and both $\theta_{1,\star}$ and $\theta_{2,\star}$ solve $0=M(\theta)$. Theorem \ref{thm:one} then implies that the posterior will concentrate around $\theta_{1,\star}$ and $\theta_{2,\star}$. Therefore, if $\theta\mapsto b(\theta)$ does not identity a unique value of $\theta$ such that $b(\theta)=b_{\bullet}$, then the result of Theorem \ref{thm:one} applies.\footnote{In the case of correctly specified models, $\Delta$ in Theorem \ref{thm:one} has the form $\Delta=[\{\partial b(\theta)/\partial\theta^\top\}^\top \Sigma(\theta)^{-1}\{\partial b(\theta)/\partial\theta^\top\}]^{-1}$. However, when the model is misspecified this is not the case, and the general definition of $\Delta$ is given in Supplementary Appendix C. 
	}
	An example demonstrating the conclusions of Theorem \ref{thm:one} in correctly specified models is given in Supplementary Appendix E.5.
\end{remark}

\begin{remark}
	One interpretation of the posterior $\widehat{\pi}(\theta\mid S_n)$ is that we are conducting generalized Bayesian inference using a scoring rule that assesses ``goodness of fit'' relative to the first two moments of the distribution for $S_n(\y)$. Following \cite{gneiting2007strictly}, a score which only depends on first and second moments ``is strictly proper relative to any convex class of probability measures characterized by the first two moments.'' However, even if the true distribution of $S_n(\y)$ is asymptotically Gaussian with mean $b_{\bullet}$ and variance $V$, there may be no value in $\Theta$ such that $b(\theta)=b_{\bullet}$ and $\Sigma(\theta)=V$; since $b(\theta)$ does not span $\mathbb{R}^{d}$, nor does $\Sigma(\theta)$ span the space of possible variance matrices. Therefore,  there is no guarantee in general that the posterior will concentrate onto a unique element even though our inferences are based on a strictly proper scoring rule.
\end{remark}

{\color{black}
\begin{remark}\label{rem:strat}
	The results and proof strategies used in \cite{frazier2019bayesian} to deduce the  asymptotic behavior of the posterior in correctly specified models are entirely dependent on the existence of a unique $\theta\in\Theta$ such that $b(\theta)=b_{\bullet}$. In particular, the proof strategy used by \cite{frazier2019bayesian} hinges on the validity of a particular global approximation to $\log g_n(S_n\mid\theta)$ which explicitly requires that $b(\theta)=b_{\bullet}$ for some $\theta\in\Theta$. As such, when the model is misspecified, the results and arguments used in \cite{frazier2019bayesian} are invalid, and we require novel arguments to
	derive the behavior of $\widehat{\pi}(\theta\mid S_n)$. 
	For additional discussion of this point, we refer the interested reader to Supplementary Appendix D.1.
\end{remark}
}

\subsection{Asymptotic behavior: single mode}\label{sec:single}

Even if the model is misspecified, if there exists a unique solution to $0=M(\theta)$, then $\widehat{\pi}(\theta\mid S_n)$ will be asymptotically Gaussian. To deduce such a result, we must restrict the set $\Theta_\star$ in Assumption \ref{ass:limits} to be a singleton.

\begin{assumption}\label{ass:strong:three}
	$\Theta_\star=\{\theta_\star\}$, and for some $\delta>0$, $0<\delta\le \lambda_{\text{min}}\{-H(\theta_\star)\}\le 1/\delta$.
\end{assumption}

Define the set $\mathcal{T}_n=\{t:t={\sqrt{n}}(\theta-\theta_n),\theta\in\Theta\}$. 
\begin{theorem}[Asymptotic Shape]\label{thm:two} If Assumptions \ref{ass:three}, \ref{ass:one}, and Assumptions \ref{ass:hess}-\ref{ass:strong:three} are satisfied, then, for $m\rightarrow\infty $ as $n\rightarrow\infty$,
	$
	\int_{\mathcal{T}_n}\|t\|\left|\widehat{\pi}(t\mid S_n)-N\{t;0,\Delta\}\right|\dt t=O_p(1/m).
	$
\end{theorem}

{{\color{black}Even when $\Theta_\star$ is a singleton, the model is still misspecified;  there does not exist a $\theta\in\Theta$ such that $b(\theta)=b_{\bullet}$. Hence, as discussed in Remark \ref{rem:strat}, the results and arguments presented in \cite{frazier2019bayesian} cannot be used to establish the behavior of $\widehat{\pi}(\theta\mid S_n)$ in this case; please see Supplementary Appendix D.1 
for further details.}} 

When $\Theta_\star$ is a singleton, if the score equations $M_n(\theta_\star)$ satisfy a central limit theorem, then the BSL posterior mean will be asymptotically Gaussian.
\begin{assumption}\label{ass:norm}For some matrix $W_\star$,  ${\sqrt{n}}M_n(\theta_\star)\Rightarrow N(0,W_\star)$.
\end{assumption}
\begin{corollary}\label{corr:bslm}
	For $\overline{\theta}_n=\int_{\Theta}\theta\widehat{\pi}(\theta\mid S_n)\dt\theta$,  if the conditions in Theorem \ref{thm:two} are satisfied and $\sqrt{n}/m=o(1)$, then
	$
	\sqrt{n}\left(\overline{\theta}_n-\theta_{\star}\right) \Rightarrow N\left(0, \Delta^{}W_\star\Delta^{\top} \right)$.
\end{corollary}

Theorem \ref{thm:two} demonstrates that if the model is misspecified, but the synthetic likelihood has a single mode asymptotically, then the BSL posterior resembles a shrinking Gaussian density in large samples. 
The posterior shape in Theorem \ref{thm:two} is in stark contrast to that exhibited by the approximate Bayesian computation posterior under model misspecification. Let $\theta_{_\bullet}$ denote the minimizer of $\|b(\theta)-b_{\bullet}\|$ and let $\epsilon_{_\bullet}=\|b(\theta_{_\bullet})-b_{\bullet}\|$. If the tolerance $\epsilon_n$ satisfies ${\sqrt{n}}(\epsilon_n-\epsilon_{_\bullet})\rightarrow0$, from above, then Theorem 2 in \cite{frazier2020model} demonstrates that asymptotically the approximate Bayesian computation posterior is proportional to the following density:$${N}\left\{\left(\left\|A \{b(\theta_{_\bullet})-b_{\bullet}\}\right\| \epsilon_{_\bullet}\right)^{-1}\left[{-\left\{A{\sqrt{n}}(S_n-b_{\bullet})\right\}^\top\left[A \{b(\theta_{_\bullet})-b_{\bullet}\}\right] \epsilon_{_\bullet}}{}-{x^{\top} L(\theta_{_\bullet}) x/4}\right];0,I_{d_\theta}\right\},$$  for $x={n}^{1/4}(\theta-\theta_{_\bullet})$, $A=\Sigma(\theta_{_\bullet})^{-1/2}$, and $L(\theta_{_\bullet})=(\partial^2/\partial\theta\partial\theta^\top)\|b(\theta)-b_{\bullet}\||_{\theta=\theta_{_\bullet}}$. 

{\color{black}It is clear that the ABC and BSL posteriors produce significantly different inferences in misspecified models. Even in the case where the BSL posterior concentrates onto a single mode, the two posteriors will generally concentrate onto different points in $\Theta$. This is an interesting contrast to the case of correctly specified models, where the two posteriors have the same shape asymptotically (see \citealp{LF2016}, \citealp{FMRR2016} and \citealp{frazier2019bayesian} for details). 
}

Furthermore, the behavior of the approximate Bayesian computation posterior mean in misspecified models is currently unknown. However, the form of the above limiting posterior suggests that its behavior may be non-standard. If this is indeed the case, it would again be a contrast to the case of Bayesian synthetic likelihood, which, from Corollary \ref{corr:bslm}, has a posterior mean that exhibits standard asymptotic behavior when the synthetic likelihood has a single mode asymptotically. 

\begin{remark}\label{rem:cover}
 	Theorem \ref{thm:two} implies that the width of posterior credible sets is determined by $\Delta$. However, Corollary \ref{corr:bslm} implies that the asymptotic variance of the BSL posterior mean is $\Delta W_\star \Delta^{\top}$. If $b(\theta_\star)\neq b_{\bullet}$ and $\Delta\ne W_\star$, we can immediately conclude that
$$
\int_{\mathcal{T}_n} tt^{\top}\widehat{\pi}(t\mid S_n)\dt t=\Delta^{}+o_p(1)\neq \text{Var}\{\sqrt{n}(\overline{\theta}_n-\theta_\star)\}=\Delta W_\star \Delta^{\top}.
$$Consequently, the BSL posterior does not deliver asymptotically valid uncertainty quantification for $\theta_\star$ in misspecified models. 	{Moreover, in Supplementary Appendix C 
	we show that $\Delta$ directly depends on the level of model misspecification, via $\{b(\theta_\star)-b_{\bullet}\}$, and in general satisfies $\Delta\ne W_\star$ when $b(\theta_\star)\neq b_{\bullet}$.}
\end{remark}
 
\begin{remark}
{\color{black}The asymptotic covariance matrix of the BSL posterior mean, $\Delta W_\star \Delta^\top$, has the same structure as the sandwich covariance matrix for the exact Bayesian posterior mean: the ``bread'' of the covariance is the inverse Hessian of the synthetic likelihood, $\Delta$, and the ``meat'' is the variance of the score equations, denoted by $W_\star$. Since the synthetic likelihood $g_n(S_n\mid\theta)$ depends on the $\theta$-dependent mean $b(\theta)$ and inverse covariance $\Sigma^{-1}_n(\theta)$, the inverse Hessian in the BSL case, $\Delta$, depends on both the first and second derivatives of $b(\theta)$, and $\Sigma^{-1}_n(\theta)$, which ensures that this covariance matrix has a complicated form. We refer the interested reader to Supplementary Appendix C  for further discussion.} 
\end{remark} 
 

\section{Robustifying the synthetic likelihood}\label{sec:adjust}
The  asymptotic mixed Gaussianity that can result from applying Bayesian synthetic likelihood methods in misspecified models can produce inaccurate inferences, and irrelevant conclusions. Hence, there is a strong sense in which we should attempt to guard against this behavior when applying these methods. In this section, we compare different approaches for ameliorating the performance of Bayesian synthetic likelihood in misspecified models. 
\subsection{Tempering the synthetic likelihood}
To obtain robustness in misspecified models several authors, including, \cite{bhattacharya2019bayesian}, \cite{grunwald2017inconsistency}, \cite{bissiri2016general}, and \cite{miller2015}, have proposed to temper the likelihood used in Bayesian inference. {\color{black}It is therefore tempting to apply this strategy to correct the behavior of synthetic likelihood under model misspecification.

{\color{black}For $\alpha\ge0$ a positive constant, the standard approach to tempering would be to use a powered version of the assumed model density within an MCMC scheme in order to generate draws from the tempered posterior. Using likelihood tempering in the case of synthetic likelihood would then lead us to use the density $N\{S_n;\widehat{b}_n(\theta),\widehat{\Sigma}_n(\theta)\}^\alpha$ within the corresponding MCMC scheme. The results of \cite{bhattacharya2019bayesian} suggest that, in the case of a genuine likelihood, so long as $\alpha\in(0,1)$, the tempered likelihood can still display posterior concentration. }
	
 Following arguments in \cite{andrieu2009pseudo}, as well as those given in \cite{price2018bayesian}, using $N\{S_n;\widehat{b}_n(\theta),\widehat{\Sigma}_n(\theta)\}^\alpha$ within the MCMC scheme results in draws from the target posterior 
\begin{align*}
		\widehat{\pi}_\alpha(\theta\mid S_n)&\propto{{}{\widehat{g}^{\alpha}_n(S_n\mid\theta) \pi(\theta)}}\\
\widehat{g}^{\alpha}_n(S_n\mid\theta) & = \int N\{S_n;\widehat{b}_n(\theta),\widehat{\Sigma}_n(\theta)\}^{\alpha} \prod_{i=1}^m \dt P^{(n)}_\theta\{S_n(\z^i)\}\,\dt S_n(\z^1)\,\dots\, \dt S_n(\z^m).
\end{align*}
The posterior $\widehat{\pi}_\alpha(\theta\mid S_n)$ does not resemble a tempered posterior, but instead resembles a posterior based on an integrated likelihood.}


We now return to the running example and examine the behavior of the tempered posterior $\widehat{\pi}_\alpha(\theta\mid S_n)$. Similar to the posterior $\widehat{\pi}(\theta\mid S_n)$, since we can compute the mean and variance of the summaries exactly, it is possible to compute an exact version of the tempered posterior. We apply the tempered version of the exact posterior using a fixed tempering schedule with $\alpha=1/2$ for each value of $n$. Following the introductory example, we plot the tempered posterior for $n=100,500,1000$ and compare the results to those obtained in Figure \ref{fig:ma2_theta}.\footnote{The results displayed in Figure \ref{fig:ma1_temp} are not overly sensitive to the choice of $\alpha$.}

Figure \ref{fig:ma1_temp} demonstrates that the tempered posterior displays similar behavior to the exact posterior in Figure \ref{fig:ma2_theta}, and does not lead to any meaningful increase in posterior mass near $\theta=0$, the point under which $\|b(\theta)-b_{\bullet}\|$ is smallest. This result is perhaps unsurprising considering that the synthetic likelihood is Gaussian, and so tempering only changes the scaling of the posterior. 

\begin{figure}[h!]
	\centering
	\setlength\figureheight{3.5cm} 
\setlength\figurewidth{3.5cm} 
	\include{fig1_MA1_R1_Tempered} 
	\vspace{-2cm}
	
	\caption{{Tempered Bayesian synthetic likelihood posteriors for $\theta$ in the running example.}}
	\label{fig:ma1_temp}
\end{figure}

{\color{black}We now formally demonstrate that the tempered posterior $\widehat{\pi}_\alpha(\theta\mid S_n)$ produces qualitatively similar behavior to $\widehat{\pi}(\theta\mid S_n)$ when the results of Theorem \ref{thm:one} are valid. To state this result recall the local parameter $t={\sqrt{n}}(\theta-\theta_n)$, let $\widehat{\pi}_\alpha(t\mid S_n)=\widehat{\pi}_\alpha(\theta_n+t/\sqrt{n}\mid S_n)/\sqrt{n}^{d_\theta}$ be the fractional posterior for $t$, and let $\pi_\alpha(t)$ be a density function that satisfies $\pi_\alpha(t)\propto N\{t;0,\Delta\}^\alpha$.
\begin{corollary}[Fractional Posteriors]\label{corr:frac} If Assumptions \ref{ass:three}-\ref{ass:five} are satisfied, then for any finite $\gamma>0$, and any fixed  $\alpha\in[0,1]$, 
	$
	\left|\int_{\|t\|\le\gamma}\widehat{\pi}(t\mid S_n)\dt t-\int_{\|t\|\le\gamma}\pi_\alpha(t)\dt t\right|=O_p(1/m)
	$ for $m\rightarrow\infty$ as $n\rightarrow\infty$.
\end{corollary}
}

\subsection{Robust synthetic likelihood}\label{sec:rbsl}

The multi-modality observed in the running example exists because  $\widehat{\pi}(\theta\mid S_n)$ assigns high probability mass to values of $\theta$ that ensure $\|\widehat{\Sigma}_n(\theta)^{-1/2}\{\widehat{b}_n(\theta)-S_n\}\|$ is small. Measuring differences between summary statistics using this relative distance, rather than an absolute distance like $\|\widehat{b}_n(\theta)-S_n\|$, means that there can exist values of $\theta$ such that $\|\widehat{b}_n(\theta)-S_n\|$ is large, while $\|\widehat{\Sigma}_n(\theta)^{-1/2}\{\widehat{b}_n(\theta)-S_n\}\|$ is small.
With the above realization, there are several approaches for correcting this behavior. For brevity, we focus on two, leaving a detailed comparison and discussion on alternative approaches for future research. 
\subsubsection{Robust Bayesian Synthetic Likelihood}

The first approach we consider is the robust Bayesian synthetic likelihood  approach presented in \cite{frazier2019robust}.\footnote{For simplicity, we only focus on the variance adjustment approach detailed in \cite{frazier2019robust}.} This approach accounts for model misspecification by altering the covariance matrix used in the synthetic likelihood to ensure that the magnitude of $\|\widehat{b}_n(\theta)-S_n\|$ is properly taken into account.  For $\Gamma=(\gamma_1,\dots,\gamma_{d})'$ denoting a $d$-dimensional random vector with support $\mathcal{G}$, define the regularized covariance matrix
$
\widehat{\Sigma}_n(\theta,\Gamma)=\widehat{\Sigma}_n(\theta)+\widehat{\Sigma}^{1/2}_n(\theta)\text{diag}\{\gamma_1,\dots,\gamma_{d}\}\widehat{\Sigma}^{1/2}_n(\theta). 
$ Let 
$$
\widehat{g}_n(S_n\mid\theta,\Gamma)=\int N\{S_n;\widehat{b}_n(\theta),\widehat{\Sigma}_n(\theta,\Gamma)\}\prod_{i=1}^{m}\dt P^{(n)}_\theta\left\{ S_n(z^{i}_{}) \right\}\dt  S_n(z_{}^{1}) \dots\dt S_n(z_{}^{m})
$$
denote the synthetic likelihood based on $\widehat{\Sigma}_n(\theta,\Gamma)$.  

The parameters $\Gamma$ allow the variance of the synthetic likelihood to increase so that the weighted norm $\|\widehat{\Sigma}_n(\theta,\Gamma)^{-1/2}\{\widehat{b}_n(\theta)-S_n\}\|$ can be made small even if there is no value in $\Theta$ under which $\|\widehat{b}_n(\theta)-S_n\|$ is small. For $\pi(\Gamma)$ denoting the prior density of $\Gamma$, \cite{frazier2019robust} use independent exponential priors, the joint posterior is
$
\widehat{\pi}(\theta,\Gamma\mid S_n)\propto{\pi(\theta)\pi(\Gamma)\widehat{g}_n(S_n\mid\theta,\Gamma)}
,
$
and Markov chain Monte Carlo methods can be used to sample from $\widehat{\pi}(\theta,\Gamma\mid S_n)$.

We now illustrate the behavior of the posterior $	\widehat{\pi}(\theta,\Gamma\mid S_n)$ under different levels of model misspecification in the running example. Following the analysis in Section \ref{sec:maexam}, we consider three sample sizes $n=100,500,1000$. The posterior is sampled using the robust option in the \texttt{BSL} package (\citealp{an2019bsl}) under the default prior choice for $\Gamma$.\footnote{We start the sampler at $\theta=0$ and retain all resulting draws. In addition, we run the sampler for 50,000 iterations and use 10 synthetic datasets for each replication. These choices are fixed across the different sample size and misspecification combinations. The acceptance rates for the resulting procedure are reasonable, and between 20\% and 60\% across all combinations.}  

We plot the robust posteriors in Figure \ref{fig:ma1_reg}. The results demonstrate that the robust posteriors for $\theta$ are Gaussian and concentrating around $\theta=0$, with the robust posterior seemingly being insensitive to the level of model misspecification. This behavior is due to the regularization of the covariance matrix, which ensures the criterion is globally concave and achieves its maximum at $\theta=0$. {\color{black} The results in Figure \ref{fig:ma1_reg} provide convincing evidence that R-BSL yields reliable posteriors in a much broader set of circumstances than originally investigated in \cite{frazier2019robust}, which considered the analysis of the R-BSL posterior in the MA(1) model but only considered a single misspecified DGP with sample size $n=100$.}

{\color{black}
	While \cite{frazier2019robust} prove posterior concentration of $\widehat{\pi}(\theta,\Gamma\mid S_n)$ in the case of \textit{correctly specified models}, as with the posterior $\widehat{\pi}(\theta\mid S_n)$, these arguments do not extend to the case of misspecified models.  Determining the theoretical behavior of $\widehat{\pi}(\theta,\Gamma\mid S_n)$ is significantly complicated by the introduction of $\Gamma$ and the behavior of these components when the model is misspecified. While the authors have observed reliable behavior for $\widehat{\pi}(\theta,\Gamma\mid S_n)$ across a multitude of examples, formal results on the asymptotic behavior of $\widehat{\pi}(\theta,\Gamma\mid S_n)$ would require specific conditions on the prior  $\pi(\Gamma)$, and the construction of novel arguments to deduce the asymptotic results. Given these complications, we leave a comprehensive study on the behavior of  $\widehat{\pi}(\theta,\Gamma\mid S_n)$ for future work. 
}

\begin{figure}[h!]
	\centering
	\setlength\figureheight{3.5cm} 
	\setlength\figurewidth{3.5cm} 
	\include{fig1_rbsl_MA1} 
\vspace{-2cm}
	\caption{{r-BSL Posteriors for $\theta$ in the misspecified MA(1) model across six different levels of model misspecification. The solid line corresponds to $n=100$, the dashed line to $n=500$ and the dotted line to $n=1000$.}}%
	
	\label{fig:ma1_reg}
\end{figure}

\subsubsection{A Robust Adjustment Approach}\label{sec:rob_adjust}
While the robust synthetic likelihood approach delivers reliable inference even in highly-misspecified models, it requires conducting posterior inference over $d_\theta+d$ (where $d\ge d_\theta)$ elements, which can become cumbersome in cases where either $\theta$ or $S_n$ is high-dimensional. However, the key insight of \cite{frazier2019robust} in regards to misspecification is that it can be handled by sufficiently altering the structure of the synthetic likelihood.

{An alternative approach to deal with model misspecification in the case of high-dimensional summaries, or parameters, is to replace the  covariance matrix $\widehat{\Sigma}_n(\theta)$ in $g_n(S_n\mid\theta)$ with a naive but fixed version $A_n$. Replacing $\widehat{\Sigma}_n(\theta)$ by the fixed matrix $A_n$ means that $\log g_n(S_n\mid\theta)$ is roughly a weighted quadratic form based on a fixed covariance matrix, and thus should be well-behaved. As the result of Theorem \ref{thm:two} suggests, if this naive posterior is indeed uni-modal, then it will be approximately Gaussian in large samples, but with a covariance matrix that depends on the choice of $A_n$. }

{\color{black}An unintended consequence of replacing $\widehat{\Sigma}_n(\theta)$ in $g_n(S_n\mid\theta)$ by the naive covariance matrix $A_n$ is that the resulting posterior will concentrate onto values of $\theta$ under which $\|A_n^{-1/2}\{b(\theta)-S_n\}\|_{}$ is small. Therefore, such a procedure ensures that the specific choice of $A_n$ influences the resulting pseudo-true value onto which the posterior will concentrate. To ensure that the pseudo-true value onto which the posterior concentrates remains meaningful, we suggest setting $A_n= I_{d}/n$, so that the naive posterior asymptotically concentrates onto the minimizer of $\|b(\theta)-b_{\bullet}\|$. This choice ensures that the resulting pseudo-true value remains a meaningful quantity in approximate Bayesian inference.\footnote{Under the choice of $A_n=I_d/n$, the pseudo-true value onto which the posterior would concentrate would be $\theta_{_\bullet}:=\arg\min_{\theta\in\Theta}\|b(\theta)-b_{\bullet}\|$, which can be interpreted as the value of $\theta\in\Theta$ that yields the closest match to the observed summaries in the Euclidean norm, and is the same value onto which the ABC posterior will concentrate if $\|\cdot\|$ were used in ABC.}}

{\color{black} While the naive posterior will concentrate onto a meaningful pseudo-true value,  the coverage of the resulting posterior will not have the correct level (see Remark \ref{rem:cover}). However, we can adjust the posterior variance to ensure it attains the correct level of frequentist coverage. In particular, we can follow a similar idea to the adjustment procedure of \cite{frazier2019bayesian} and adjust the posterior draws using the following algorithm.}

\noindent1.  Take $\widehat{\Sigma}_n(\theta)=A_n$, for all $\theta\in\Theta$, as the covariance matrix in the synthetic likelihood and obtain the corresponding naive posterior mean, $\bar{\theta}_n$, and posterior covariance, $\widehat{\Delta}_n$. 

\noindent2. For $\theta^j$, $j=1,\dots,N$, a sample from the  naive posterior, adjust $\theta^j$ according to 
$
\widetilde{\theta}^{j}=\bar{\theta}_n+\widehat{\Delta}_n\widehat{W}_n^{1/2}\widehat{\Delta}_n^{-1/2}(\theta^j-\bar{\theta}_n),
$ where $\widehat{W}_n$ is any consistent estimator of the asymptotic variance $W_\star:= \lim_n\text{var}\{{\sqrt{n}}M_n(\theta_\star)\}$.

{\color{black}While the naive posterior in the first step ensures concentration onto values of $\theta$ under which $\|b(\theta)-b_{\bullet}\|$ is small, this posterior has unreliable uncertainty quantification, and so the second step adjusts the posterior variance. When the model is misspecified, the most reliable estimator of $ W_\star:=\lim_n\text{var}\{{\sqrt{n}}M_n(\theta_\star)\}$ is obtained using bootstrapping where we re-sample the summary statistics and recalculate the synthetic likelihood equations $M_n(\theta)$ at each bootstrapped summary statistic. In this way, we can interpret the above adjusted posterior as being similar to the ``BayesBag'' posterior (see  \citealp{huggins2019using}), but in the synthetic likelihood context. However, unlike BayesBag our approach \textit{does not} require re-running any posterior sampling mechanism, and only requires bootstrapping summary statistics and recalculating the synthetic likelihood score equations. Moreover, unlike BayesBag, the following result demonstrates that the adjusted posterior delivers asymptotically valid uncertainty quantification. 

Let $\theta_n:=\arg\min_{\theta\in\Theta}\|b(\theta)-S_n\|$, $\vartheta:=\sqrt{n}(\widetilde{\theta}-\theta_n)$, and  $\mathcal{T}_n=\{t:\vartheta={\sqrt{n}}(\widetilde{\theta}-\theta_n),\;\widetilde{\theta}\in\Theta\}$, where $\widetilde{\theta}$ denotes the transformed version of $\theta$ given in the second step  of the algorithm.
\begin{corollary}\label{corr:absl}
 Assume Assumptions \ref{ass:three}, \ref{ass:one}, and Assumptions \ref{ass:hess}-\ref{ass:norm} are satisfied, and that $\widehat\Delta_n$ and $\widehat W_n$ are consistent estimators for $\Delta$ and $W_\star$, respectively. For $m\rightarrow\infty $ as $n\rightarrow\infty$,
	$
	\int_{\mathcal{T}_n}\left|\widehat{\pi}(\vartheta\mid S_n)-N\{\vartheta;0,\Delta W_\star \Delta^\top\}\right|\dt \vartheta=o_p(1)
	$, and for $\widetilde{\theta}_n^{}:=\int_{\Theta}\widetilde{\theta} \cdot \widehat{\pi}(\widetilde{\theta}\mid S_n)\dt\widetilde{\theta}$, 
	$\sqrt{n}(\widetilde{\theta}_n-\theta_{_\bullet}) \Rightarrow N(0, \Delta^{}W_\star\Delta^{\top} )$.	
\end{corollary}

	\begin{remark}
	The adjusted BSL posterior concentrates onto the same limiting value as the ABC posterior under model misspecification. However, unlike the ABC posterior the adjusted BSL posterior displays Gaussian posterior concentration, and asymptotically correctly quantifies uncertainty about the pseudo-true value $\theta_{_\bullet}:=\arg\min_{\theta\in\Theta}\|b(\theta)-b_{\bullet}\|$. Consequently, if the user is faced with a misspecified model in approximate Bayesian inference, and correct uncertainty quantification is desirable, then we recommend the use of robust BSL methods over ABC methods. 
\end{remark}

\begin{remark}
 The above adjustment procedure is related to, but distinct from, the procedure discussed in Section 4 of \cite{frazier2019bayesian}. In correctly specified models, \cite{frazier2019bayesian} use a similar approach to the second step of the above algorithm to correct the posterior covariance in cases where a computational convenient covariance matrix is initially used in BSL. We advise against using the approach outlined in \cite{frazier2019bayesian} when the model is misspecified, since the resulting posterior will concentrate onto a point that is determined by the choice of computationally convenient covariance matrix. Due to space limitations, we forgo a detailed comparison with the adjustment approach of \cite{frazier2019bayesian} to Appendix D.2.
\end{remark}
}

We now examine the behavior of the adjusted posterior in the running example through a repeated sampling experiment. 
In particular, we generate five hundred datasets from the model in \eqref{trueDGP}, where the parameter values are $\omega = -0.15$, $\rho = 0.90$ and $\sigma_v = 0.40$, and with $n=100,500,1000$ observations. The adjusted posterior is obtained by calculating the exact naive  posterior, setting $A_n$ to be the identity covariance matrix, and then adjusting 10,000 samples from the naive posterior. The variance term $\widehat{W}_n$ used in the adjustment is estimated via the block bootstrap with a block size of 5 and using 1,000 bootstrap samples. 

In Table \ref{tab:one} we record the posterior mean (multiplied by $10^3$), variance and Monte Carlo coverage for both the adjusted and naive approaches, and across each of the three samples sizes. The results demonstrate that both approaches give precise estimators for the location of the pseudo-true value, $\theta=0$, with the adjusted approach having a smaller posterior variance across all sample sizes. In terms of Monte Carlo coverage, both procedures display over-coverage for the unknown pseudo-true value. Therefore, given the tighter posteriors for the adjusted approach, and similar posterior means, we conclude that the adjusted approach is more accurate than the naive approach. 

{	\begin{table}[h!]
		\centering
		\centering
		{	\begin{tabular}{lrrrrrr}
				& \multicolumn{2}{c}{\underline{$n$=100}} &   \multicolumn{2}{c}{\underline{$n$=500}} &  \multicolumn{2}{c}{\underline{$n$=1,000}}   \\
				& \multicolumn{1}{l}{a-BSL} & \multicolumn{1}{l}{n-BSL} & \multicolumn{1}{l}{a-BSL} & \multicolumn{1}{l}{n-BSL} &\multicolumn{1}{l}{a-BSL} & \multicolumn{1}{l}{n-BSL}        \\\hline
				Mean  & -0.5691  & -0.1955 & -0.9916   & 0.0863 & -0.2654   & 0.0191       \\
				Variance & 0.0966  &  0.3322 & 0.0197  &  0.0668 & 0.0107  &  0.0334      \\
				Coverage   & 100\%  & 100\% & 100\% & 100\% & 100\% & 100\%        \\
			\end{tabular}%
		
		}\caption{{Summary measures of posterior accuracy, calculated as averages across the replications. Mean - posterior mean multiplied by $10^{3}$, Variance - posterior variance, Coverage - Monte Carlo coverage. n-BSL refers to naive BSL, and a-BSL refers to adjusted BSL.}}	\label{tab:one}%
	
	\end{table}
}		

In Supplementary Appendix F 
we apply the robust procedures discussed in Section \ref{sec:rbsl} to analyze a simple model of financial returns. Both methods behave similarly and suggest the presence of significant model misspecification.

\section{Discussion}\label{sec:discuss}
Over the last decade, approximate Bayesian methods like synthetic likelihood have gained acceptance in the statistical community for their ability to produce meaningful inferences in complex models. The ease with which these methods can be applied has also led to their use in diverse fields of research; see \cite{sisson2018handbook} for examples. 

While the initial impetus for these methods was one of practicality, recent research has begun to focus on the theoretical behavior of these methods. In the context of BSL, \cite{frazier2019bayesian} demonstrate that BSL posteriors are well-behaved in large samples, and can deliver inferences that are just as reliable as those obtained by ABC, assuming the model is correctly specified.

However, the important message delivered in this paper is that if the assumed model is misspecified, then synthetic likelihood-based inference can be unreliable, and the resulting posterior can be significantly different to those obtained from other approximate Bayesian methods. Critically, the type of behavior exhibited by the Bayesian synthetic likelihood posterior is intimately related to the form and degree of model misspecification, which cannot be reliably measured without first conducting some form of inference.

While our results have only focused on the most commonly applied variant of the synthetic likelihood posterior, we conjecture that recently proposed variations, such as the semiparametric approach of \cite{an2020robust} or the whitening approach of \cite{priddle2019efficient}, will exhibit similar behavior. {The semiparametric synthetic likelihood approach of \cite{an2020robust} allows the user to remove the implicit Gaussian assumption for the summary statistic likelihood and estimate the likelihood of the summaries using kernel smoothing methods. However, this more general approach will still suffer from the same issues as the original BSL posterior under model misspecification. Estimating the distribution of the summary statistics  does not change the fact that when the model is misspecified the observed summary statistic cannot be matched by the simulated statistics. Consequently, changing the class of approximations used for the synthetic likelihood from the Gaussian to a more flexible class \textit{does not address} the underlying issue of model misspecification, and so the resulting posterior will behave similarly to those analyzed herein. 
}


\noindent\textbf{Acknowledgments.}  David Frazier was supported by the Australian Research Council's Discovery Early Career Researcher Award funding scheme (DE200101070). David Nott was supported by a Singapore Ministry of Education Academic Research Fund Tier 1 grant and is affiliated with the Operations Research and Analytics Research cluster at the National University of Singapore.    Christopher Drovandi was supported by the Australian Research Council Future Fellowships Scheme (FT210100260).  We are grateful to the Associate Editor and two referees for their very helpful comments and suggestions that have significantly improved the paper. All remaining errors are our own.

{
\bibliographystyle{chicago}
\bibliography{refs_mispecCDF}
}

\appendix

\section{Proofs of main results}\label{appendix:one}

\begin{proof}[Proof of Theorem \ref{thm:one}]
	Define  $C_\Delta= 1/\{{(2\pi)}^{d_\theta}|\Delta|^{}\}^{1/2}$ and let $\Pi_\star(\gamma)=C_\Delta\int_{\|t\|\le\gamma}\exp\left(-t^\intercal \Delta^{-1}t/2\right)\d t$. By the triangle inequality, 
	\begin{flalign*}
		&\left|\int_{\|t\|\le \gamma} \widehat{\pi}(t\mid S_n)\d t-C_\Delta\int_{\|t\|\le\gamma}\exp\left(-t^\intercal \Delta^{-1}t/2\right)\d t\right|=\left|\int_{\|t\|\le \gamma} \widehat{\pi}(t\mid S_n)\d t-\Pi_\star(\gamma)\right|\nonumber\\&\leq\left|\int_{\|t\|\le\gamma}\left\{\widehat{\pi}(t\mid S_n)-\pi(t\mid S_n)\right\}\d t\right|+\left|\int_{\|t\|\le \gamma}{\pi}(t\mid S_n)\d t-\Pi_\star(\gamma)\right|.
	\end{flalign*}
	By Lemma \ref{lem:orders} (see Section \ref{app:lemmas}),  $\int_{\|t\|\le\gamma}\left|\widehat{\pi}(t\mid S_n)-\pi(t\mid S_n)\right|\d t=O_p(1/m)$ for $m\rightarrow\infty$ as $n\rightarrow\infty$; and so the result follows  if $J_n=|\int_{\|t\|\le \gamma}{\pi}(t\mid S_n)\d t-\Pi_\star(\gamma)|=o_p(1).$
	
	Let $Q_n(\theta)=-\|\Sigma_{n}^{-1/2}(\theta)\{b_n(\theta)-S_n\}\|^2/2$, and rewrite the exact posterior $\pi(\theta\mid S_n)$ as
	\begin{flalign}
		\pi(\theta\mid S_n)&= \pi(\theta_n\mid S_n)\left[{|n\Sigma_n^{}(\theta)|}/{|n\Sigma_n^{}(\theta_n)|}\right]^{-1/2}\exp\left\{Q_n(\theta)-Q_n(\theta_n)\right\}\pi(\theta)/\pi(\theta_n),\label{eq:Clev_post}
	\end{flalign}for $\theta_n$ as in Lemma \ref{lem:cons}. 
	For any $\gamma>0$, let $\gamma_n=\gamma/{{\sqrt{n}}}$, with $\gamma_n=o(1)$, and define $\mathcal{N}_\gamma=\{\theta\in\Theta:\|\theta-\theta_n\|\le\gamma_n\}$. Using \eqref{eq:Clev_post},
	decompose the posterior probability over $\mathcal{N}_\gamma$ as 
	\begin{flalign*}
		\int_{\mathcal{N}_\gamma}\pi(\theta\mid S_n)\d\theta&=\pi(\theta_n\mid S_n)\int_{\mathcal{N}_\gamma}\exp\left\{Q_n(\theta)-Q_n(\theta_n)\right\}\d\theta\\&+\pi(\theta_n\mid S_n)\int_{\mathcal{N}_\gamma}\exp\left\{Q_n(\theta)-Q_n(\theta_n)\right\}\left[\frac{\pi(\theta)}{\pi(\theta_n)}-1\right]\d\theta\\&+\pi(\theta_n\mid S_n)\int_{\mathcal{N}_\gamma}\exp\left\{Q_n(\theta)-Q_n(\theta_n)\right\}\left[\frac{|n\Sigma_n(\theta)|^{-1/2}}{|n\Sigma_n(\theta_n)|^{-1/2}}-1\right]\d\theta\\&+\pi(\theta_n\mid S_n)\int_{\mathcal{N}_\gamma}\exp\left\{Q_n(\theta)-Q_n(\theta_n)\right\}\left[\frac{|n\Sigma_n(\theta)|^{-1/2}}{|n\Sigma_n(\theta_n)|^{-1/2}}-1\right]^{}\left[\frac{\pi(\theta)}{\pi(\theta_n)}-1\right]\d\theta\\&=\pi(\theta_n\mid S_n)\int_{\mathcal{N}_\gamma}\exp\left\{Q_n(\theta)-Q_n(\theta_n)\right\}\d\theta+\pi(\theta_n\mid S_n)(C_{1n}+C_{2n}+C_{3n})
	\end{flalign*}
	
	From Assumption \ref{ass:four} and Lemma \ref{lem:cons}, $|\pi(\theta_n)-\pi(\theta_\star)|=o_p(1)$, with $\pi(\theta_\star)>0$. For notational brevity, let $\ell_n(\theta)=|n\Sigma_n(\theta)|^{-1/2}$ and $\ell(\theta)=|\Sigma(\theta)|^{-1/2}$. Upper bound $C_{jn}$, $j=1,2,3$, as 
	\begin{flalign*}
		C_{1n}&\le C\left\{\sup_{\|\theta-\theta_n\|\leq\gamma_n}|\pi(\theta)-\pi(\theta_n)|\right\} \int_{\mathcal{N}_\gamma}\exp\left\{Q_n(\theta)-Q_n(\theta_n)\right\}\d\theta \\
		C_{2n}&\le C\left\{\sup_{\|\theta-\theta_n\|\leq\gamma_n}|\ell_n(\theta)-\ell(\theta)|\right\} \int_{\mathcal{N}_\gamma}\exp\left\{Q_n(\theta)-Q_n(\theta_n)\right\}\d\theta\\&+C\left\{\sup_{\|\theta-\theta_n\|\leq\gamma_n}|\ell(\theta)-\ell(\theta_n)|\right\} \int_{\mathcal{N}_\gamma}\exp\left\{Q_n(\theta)-Q_n(\theta_n)\right\}\d\theta\\
		C_{3n}&\le C\left[\sup_{\|\theta-\theta_n\|\leq\gamma_n}|\pi(\theta)-\pi(\theta_n)|\left\{|\ell_n(\theta)-\ell(\theta)|+|\ell(\theta)-\ell(\theta_n)|\right\}\right] \int_{\mathcal{N}_\gamma}\exp\left\{Q_n(\theta)-Q_n(\theta_n)\right\}\d\theta 
	\end{flalign*}for a given constant $C$ that changes line-by-line. 
	
	Since $\pi(\theta),\;\ell(\theta)$ are continuous, they are bounded over $\mathcal{N}_\gamma$ and, by Assumption \ref{ass:two}, $\ell_n(\theta)$ is bounded over $\mathcal{N}_\gamma$ for $n$ large enough. Thus, 
	\begin{equation}\label{eq:1}
		\sup_{\theta\in\mathcal{N}_\gamma}|\ell_n(\theta)-\ell(\theta)|=o_p(1),\quad 	\sup_{\theta\in\mathcal{{N}}_\gamma}|\pi(\theta)-\pi(\theta_n)|=o_p(1).
	\end{equation}
	By Assumption \ref{ass:three}, for each $j=1,\dots,d_\theta$, the matrix $\partial\Sigma(\theta)/\partial\theta_j$  is continuous so that, over $\mathcal{N}_\gamma$, 
	\begin{equation}\label{eq:2}
		\sup_{\|\theta-\theta_n\|\le \gamma_n}|\ell(\theta)-\ell(\theta_n)|\leq \sup_{\|\theta-\theta_n\|\le \gamma_n}\|\partial \ell(\theta)/\partial\theta\|\gamma_n=o(1).
	\end{equation}
	Therefore, equations \eqref{eq:1}-\eqref{eq:2} imply that $C_{jn}=o_{p}\left[\int_{\mathcal{N}_\gamma}\exp\left\{Q_n(\theta)-Q_n(\theta_n)\right\}\dt\theta\right]$, for $j=1,2,3$,  and we have
	\begin{equation}
		\int_{\mathcal{N}_\gamma}\pi(\theta\mid S_n)\d\theta=\pi(\theta_n\mid S_n)\{1+o_p(1)\}\int_{\mathcal{N}_\gamma}\exp\left\{Q_n(\theta)-Q_n(\theta_n)\right\}\d\theta\label{eq:approx1}.
	\end{equation}
	
	Now, consider the change of variables $\theta\mapsto t={{\sqrt{n}}}(\theta-\theta_n)$, and the posterior $\pi(t\mid S_ n)=\pi(\theta_n+t/{{\sqrt{n}}}\mid S_n)/{n^{{d_\theta}/2}}$. Noting that $\mathcal{N}_\gamma$ can be written as $\{t:\|t\|\le\gamma\}$, we have
	\begin{flalign}
		\int_{\mathcal{N}_\gamma}\pi(\theta\mid S_n)\d \theta&=\int_{\|t\|\le\gamma}\pi(t\mid S_n)\dt t\nonumber\\&=\frac{\pi(\theta_n\mid S_n)}{{n^{{d_\theta}/2}}}\{1+o_p(1)\}\int_{\|t\|\le\gamma}\exp\left\{Q_n(\theta_n+t/{{\sqrt{n}}})-Q_n(\theta_n)\right\}\dt t\nonumber\\&\le C_\Delta\{1+o_p(1)\}\int_{\|t\|\le\gamma}\exp\left\{Q_n(\theta_n+t/{{\sqrt{n}}})-Q_n(\theta_n)\right\}\dt t\label{eq:xxx};
	\end{flalign} the second equality follows from equation \eqref{eq:approx1} and the change of variables, and the last equation follows by Lemma \ref{bound} (with probability converging to one)  
	Applying equation \eqref{eq:xxx}, we see that, up to an $o_p(1)$ term,
	\begin{flalign*}
		J_{n}&\lesssim \left|\int_{\|t\|\le\gamma}\left[\exp\left\{Q_n(\theta_n+{t}/{{{\sqrt{n}}}})-Q_n(\theta_n)\right\}-\exp\left(-t^{\top}\Delta^{-1}t/2\right)\right]\d t\right|;
	\end{flalign*} and the result follows if the right hand side of the above converges to zero in probability.

	For $Z_n={{\sqrt{n}}}(S_n-b_{\bullet})$, let $\Omega_n=\{Z_n:\|Z_n\|\leq M_n/2\}$, for $M_n\rightarrow\infty$, with $M_n=o({{\sqrt{n}}})$,  and  $P^{(n)}_{\bullet}(\Omega_n)\rightarrow1$ by Assumption \ref{ass:one}. For $\|t\|\le\gamma$, on the set $\Omega_n$, from the expansion in Lemma \ref{lemma:freq},
	\begin{flalign*}
		Q_n(\theta_n+t/{{\sqrt{n}}})-Q_n(\theta_n)&=\frac{1}{2}t^{\top}H(\theta_n)t+O(\|t\|^2\{1+\|Z_n\|\}/{{\sqrt{n}}})
		=\frac{1}{2}t^{\top}H(\theta_n)t+o(1),
	\end{flalign*} where the second equality follows since $\|Z_n/{{\sqrt{n}}}\|\le M_n/{{\sqrt{n}}}=o(1)$. Over $\{t:\|t\|\le\gamma\}\cap\Omega_n$, rewrite the above as
	\begin{flalign}
		Q_n(\theta_n+t/{{\sqrt{n}}})-Q_n(\theta_n)&=-\frac{1}{2}t^{\top}\left(I+V_n\right)\Delta^{-1}t+o(1),\label{eq1:rep1}
	\end{flalign}where $V_n=
	\left[-H(\theta_n)-\Delta^{-1}\right]\Delta$. 
	
	By Assumption \ref{ass:hess},
	$
	\|V_n\|\le K\|\Delta\|\|\theta_n-\theta_{\star}\|,
	$ for some $K>0$.  
	Define $A_n=K \Delta \|\theta_n-\theta_{\star}\|$, and conclude that $A_n$ is positive semi-definite with maximal eigenvalue
	$\lambda_{\text{max}}(A_n)=K\|\theta_n-\theta_\star\|\lambda_{\text{max}}(\Delta)\ge0$. Therefore, over $\|t\|\le\gamma$, 
	\begin{equation*}
		-t^{\top}A_nt\le t^{\top}V_nt\le t^{\top}A_nt,
	\end{equation*}and applying the above into \eqref{eq1:rep1} yields 
	\begin{flalign*}
		-t^{\top}\Delta^{-1}t/2-t^{\top}A_n\Delta^{-1}t/2
		\leq Q_n(\theta_n+t/{{\sqrt{n}}})-Q_n(\theta_n)&\leq -t^{\top}\Delta^{-1}t/2+t^{\top}A_n\Delta^{-1}t/2.
	\end{flalign*}
	However, using the definition of $A_n$, over $\|t\|\le\gamma$ the above equation can be simplified as
	\begin{flalign}\label{eq:bound}
		-t^{\top}\Delta^{-1}t/2-C\|\theta_n-\theta_\star\|
		\leq Q_n(\theta_n+t/{{\sqrt{n}}})-Q_n(\theta_n)&\leq -t^{\top}\Delta^{-1}t/2+C\|\theta_n-\theta_\star\|.
	\end{flalign}
	
	Using equations \eqref{eq:xxx}-\eqref{eq:bound}, for $n$ large enough, we can bound the posterior over $\mathcal{N}_{\gamma}$ above and below as
	\begin{flalign*}
		&\frac{\pi(\theta_n\mid S_n)}{{n^{{d_\theta}/2}}}\int_{\mathcal{N}_{\gamma}}\exp\left\{Q_n(\theta_n+t/{{\sqrt{n}}})-Q_n(\theta_n)\right\}\d t \\&\leq\frac{\pi(\theta_n\mid S_n)}{{n^{{d_\theta}/2}}} \exp(C\|\theta_n-\theta_\star\|)\int_{\mathcal{N}_\gamma}\exp\left(-t^{\top}\Delta^{-1}t/2\right)\d t\\&\geq \frac{\pi(\theta_n\mid S_n)}{{n^{{d_\theta}/2}}}\exp(-C\|\theta_n-\theta_\star\|)\int_{\mathcal{N}_\gamma}\exp\left(-t^{\top}\Delta^{-1}t/2\right)\d t.
	\end{flalign*}
	As $n\rightarrow\infty$, Lemma \ref{bound}, and the dominated convergence theorem allow us to conclude that
	$$
	\Pi_\star(\gamma)\{1+o_p(1)\}\le \frac{\pi(\theta_n\mid S_n)}{{n^{{d_\theta}/2}}}\int_{\mathcal{N}_{\gamma}}\exp\left\{Q_n(\theta_n+t/{{\sqrt{n}}})-Q_n(\theta_n)\right\}\d t\le \Pi_\star(\gamma)\{1+o_p(1)\}.
	$$It follows that $J_n\rightarrow0$ in probability.
\end{proof}

\begin{proof}[Proof of Theorem \ref{thm:two}] 
	By the triangle inequality
	\begin{flalign}
		\int_{\mathcal{T}_n} \|t\|^{}|\widehat{\pi}(t\mid S_n)-N\{t;0,\Delta\}|\dt t&\leq \int_{\mathcal{T}_n} \|t\|^{}|\pi(t\mid S_n)-N\{t;0,\Delta\}|\dt t\\&+\int_{\mathcal{T}_n} \|t\|^{}|\widehat{\pi}(t\mid S_n)-\pi(t\mid S_n)|\dt t\nonumber\\&= \mathcal{V}_{1n}+\mathcal{V}_{2n}\label{defns}.
	\end{flalign}
	We separately demonstrate that $\mathcal{V}_{1n}$ and $\mathcal{V}_{2n}$ are $o_p(1)$. 
	
	\medskip 
	
	\noindent\textbf{Part 1: $\mathcal{V}_{1n}=o_p(1)$.} For some $\delta_n$ such that $\delta_n=o(1)$ and ${{\sqrt{n}}}\delta_n\rightarrow\infty$,  split the region of integration $\mathcal{T}_n$ as follows: (i) $\mathcal{T}_{1n}=\{\|t\|\le \delta_n{{\sqrt{n}}}\}$, and (ii) $\mathcal{T}_{2n}=\{\|t\|\ge\delta_n{{\sqrt{n}}}\}$. We consider the integral over each region separately. 
	
	\medskip 
	
	\noindent\textbf{Region (i): $\mathcal{T}_{1n}=\{\|t\|\le \delta_n{{\sqrt{n}}}\}$.}  Rewrite the exact  posterior as	
	\begin{flalign*}
		\pi(\theta\mid S_n)=\frac{\pi(\theta_n\mid S_n)|n\Sigma_n(\theta_n)|^{1/2}}{\pi(\theta_n)}|n\Sigma_n(\theta)|^{-1/2}\exp\{Q_n(\theta)-Q_n(\theta_n)\}\pi(\theta).
	\end{flalign*}However,  $\|\theta_n-\theta_\star\|=o_p(1)$ (by Lemma \ref{lem:cons}), and under uniqueness of $\theta_\star$ ({Assumption \ref{ass:strong:three}})  Lemma \ref{bound} implies
	\begin{flalign}\label{eq:star1}
		\pi(\theta_n\mid S_n)/{n^{{d_\theta}/2}}=C_\Delta+o_p(1),
	\end{flalign}and with $C_\Delta=1/\{(2\pi)^{d_\theta}|\Delta|\}^{1/2}$ as defined in Theorem \ref{thm:one}. Over $\{\|t\|\le \delta_n{{\sqrt{n}}}\}$, equation \eqref{eq:star1} and similar arguments to those used in Theorem \ref{thm:one} to obtain equation \eqref{eq:xxx}
	yield
	\begin{flalign*}
		\int_{\|\theta-\theta_n\|\le\delta_n}\pi(\theta\mid S_n)\dt\theta&=\frac{\pi(\theta_n\mid S_n)}{{n^{{d_\theta}/2}}}\{1+o_p(1)\}\int_{\|t\|\le\delta_n{{\sqrt{n}}}}\pi(t\mid S_n)\dt t\nonumber\\&=C_\Delta\{1+o_p(1)\}\int_{\mathcal{T}_{1n}}\exp\{Q_n(\theta_n+t/{{\sqrt{n}}})-Q_n(\theta_n)\}\dt t,
	\end{flalign*} and, up to an $o_p(1)$ term,
	\begin{flalign}\label{eq:star4}
		&\int_{\mathcal{T}_{1n}} \|t\||\pi(t\mid S_n)-N\{t;0,\Delta\}|\dt t\noindent\\&=C_\Delta\int_{\mathcal{T}_{1n}}\|t\||\exp\{Q_n(\theta_n+{t}/{{{\sqrt{n}}}})-Q_n(\theta_n)\}-\exp(-t^\intercal\Delta^{-1}t/2)|\dt t.
	\end{flalign}Over $\mathcal{T}_{1n}$, arguments similar to those used in Theorem \ref{thm:one} to obtain equation \eqref{eq1:rep1} yield
	\begin{equation}
		Q_n(\theta_n+t/{{\sqrt{n}}})-Q_n(\theta_n)\leq -t^{\top}\left(I-A_n\right)\Delta^{-1} t/2+o_p(1)\label{eq:star3},
	\end{equation}
	for $A_n$ as defined in that proof.
	
	For some $h>0$, further split $\mathcal{T}_{1n}=\{\|t\|\le h\}\cup\{h\le \|t\|\le\delta_n{{\sqrt{n}}}\}$. Over the first set, $\|t\|$ can be dropped from the computation. Recalling that $A_n\rightarrow0$ as $n\rightarrow\infty$, over $\|t\|\le h$, equation \eqref{eq:star3} implies that
	$$
	Q_n(\theta_n+t/{{\sqrt{n}}})-Q_n(\theta_n)=-t^\intercal\Delta^{-1}t/2+o_p(1),
	$$and it follows that the integral in \eqref{eq:star4} is $o_p(1)$ over $\|t\|\le h$. 
	
	Over $\{h\le \|t\|\le \delta_n{{\sqrt{n}}}\}$, $
	\int_{h<\|t\|\le {{\sqrt{n}}} \delta_n}\|t\| N\{t;0,\Delta\}\dt t
	$ can be made arbitrarily small by taking $h$ large enough and $\delta_n$ small enough, so that, applying \eqref{eq:star3}, it suffices to show that, for any $\varepsilon>0$ there exists an $h$ and $\delta_n$ such that, for some $n$ large enough,  
	$$
	P^{(n)}_{\bullet}\left[\int_{h<\|t\|\le {{\sqrt{n}}} \delta_n}\|t\| \exp\{-t^{\top}\left(I-A_n\right)\Delta^{-1} t/2\}\dt t<\varepsilon\right]\ge 1-\varepsilon.
	$$
	
	However, for $h'$ large enough, and all $h>h'$, on the set $h<\|t\|\le \delta_n{{\sqrt{n}}}$, 
	$$
	\|t\| \exp(-t^\intercal \Delta^{-1}t/2)=O(1/h). 
	$$Therefore, for any $\varepsilon>0$, there is an $h$ large enough and a $\delta_n$ small enough such that
	$$
	\int_{h<\|t\|\le {{\sqrt{n}}} \delta_n}\|t\|\exp(-t^{\top}\Delta^{-1}t/2)\dt t<\varepsilon.
	$$Since $A_n\rightarrow 0$, we can conclude that for some $n$ large enough, with $P^{(n)}_{\bullet}$-probability at least $1-\varepsilon$,
	$$
	\int_{M<\|t\|\le {{\sqrt{n}}} \delta}\|t\| \exp\{-t^{\top}\left(I-A_n\right)\Delta^{-1}t/2\}\dt t<\varepsilon.
	$$
	
	\medskip 
	
	\noindent\textbf{Region (ii): $\mathcal{T}_{2n}=\{\|t\|\ge\delta_n{{\sqrt{n}}}\}$.} Again, 
	$
	\int_{\mathcal{T}_{2n}}\|t\|N\{t;0,\Delta\}\dt t 
	$ can be made arbitrarily small by taking $\delta_n {{\sqrt{n}}}$ large enough, and it remains to show that $\int_{\mathcal{T}_{2n}}\|t\|\pi(t\mid S_n)=o_p(1)$. 
	
	Applying, in-turn, the expression for the exact posterior in \eqref{eq:Clev_post}, Assumptions \ref{ass:two} and \ref{ass:four}, and Lemma \ref{bound}, 
	\begin{flalign*}
		\int_{\mathcal{T}_{2n}} \|t\|\pi(t\mid S_n)\dt t\le C \int_{\mathcal{T}_{2n}}\frac{\|t\|\pi(\theta_n+t/{{\sqrt{n}}})}{|n\Sigma_n(\theta_n+t/{{\sqrt{n}}})|^{1/2}}\exp\{Q_n(\theta_n+t/{{\sqrt{n}}})-Q_n(\theta_n)\}\dt t,
	\end{flalign*}for some $C>0$. Using the change of variables $t\mapsto\theta$, the integral on the right hand side becomes 
	\begin{flalign}\label{eq:terms1}
		C\left\{1+o_p(1)\right\}{n^{d_{\theta}/2}}\int_{\|\theta-\theta_\star\|>\delta}\|\theta-\theta_\star\|\pi(\theta)|n\Sigma_n(\theta)|^{-1/2}\exp\{Q_n(\theta)-Q_n(\theta_n)\}\d\theta,
	\end{flalign}where the $o_p(1)$ term follows from the triangle inequality and consistency of $\theta_n$ for $\theta_\star$ (Lemma \ref{lem:cons}). 
	For any $\delta>0$, and $Q(\theta)=-\{b(\theta)-b_{\bullet}\}^\intercal\Sigma(\theta)^{-1}\{b(\theta)-b_{\bullet}\}/2$, where we recall that $b(\theta)=\lim_n b_n(\theta)$, under Assumptions \ref{ass:three} and \ref{ass:one},
	\begin{flalign*}
		\sup_{\|\theta-\theta_\star\|>\delta}n^{-1}\left\{Q_n(\theta)-Q_n(\theta_n)\right\}&
		\leq \sup_{\|\theta-\theta_\star\|>\delta}\{Q(\theta)-Q(\theta_\star)\}+o_p(1).
	\end{flalign*}
	Further, continuity of $Q(\theta)$, and uniqueness of $\theta_\star$ imply that for any $\delta>0$ there exists some $\epsilon>0$ such that 
	$
	\sup_{\|\theta-\theta_\star\|>\delta}\{Q(\theta)-Q(\theta_\star)\}\le -\epsilon. 
	$ Therefore, for any $\delta>0$, 
	$$
	\lim_{n\rightarrow\infty} P^{(n)}_{\bullet}\left[\sup_{\|\theta-\theta_\star\|>\delta}\exp\{Q_n(\theta)-Q_n(\theta_n)\}\le \exp(-\epsilon n)\right]=1. 
	$$Applying the above  into equation \eqref{eq:terms1}, and dropping the $o_p(1)$ term, we can conclude
	\begin{flalign*}
		\int_{\|t\|>{{\sqrt{n}}}\delta} \|t\| \pi(t\mid S_n)\dt t&\leq C {n^{d_{\theta}/2}}\exp(-\epsilon n)\int_{\|\theta-\theta_\star\|>\delta}\|\theta-\theta_\star\||n\Sigma_n(\theta)|^{-1/2}\pi(\theta)\d\theta 
	\end{flalign*}with probability converging to one. Using Cauchy-Schwartz, upper bound the last term as 
	\begin{flalign*}
		&C {n^{d_{\theta}/2}}\exp(-\epsilon n)\int_{\|\theta-\theta_\star\|>\delta} {\|\theta-\theta_\star\||n\Sigma_n(\theta)|^{-1/2}\pi(\theta)}\d\theta \\&\leq C {n^{d_{\theta}/2}}\exp(-\epsilon n)\left[\int_{\Theta}\|\theta\|^{2}\pi(\theta)\d\theta\right]^{1/2}\left[\int_{\Theta}{|n\Sigma_n(\theta)|^{-1}\pi(\theta)}{}\d\theta\right]^{1/2}\lesssim {n^{d_{\theta}/2}}\exp(-\epsilon n),
	\end{flalign*}where the last inequality follows from compactness of $\Theta$, and Assumption \ref{ass:two}.
	
	Since $\mathcal{V}_{1n}$ is $o_p(1)$ across both $\mathcal{T}_{1n}$ and $\mathcal{T}_{2n}$, $\mathcal{V}_{1n}$ is $o_p(1)$ over $\mathcal{T}_n$.
	
	\medskip 
	
	\noindent\textbf{Part 2: $\mathcal{V}_{2n}=o_p(1)$.} From equation \eqref{eq:new1} in the proof of Lemma \ref{lem:orders}, and the fact that $\int_\Theta{g}_n(S_n\mid\theta)\pi(\theta)\dt\theta<\infty$ (shown in \textbf{Part 1} of this result), we have 
	\begin{flalign*}
		\int_{\Theta}\|\theta\||\widehat{g}_n(S_n\mid\theta)\pi(\theta)- {g}_n(S_n\mid\theta)\pi(\theta)|\dt \theta
		\lesssim m^{-1}\int_\Theta{g}_n(S_n\mid\theta)\pi(\theta)\dt\theta\int_\Theta \|\theta\|\pi(\theta\mid S_n)\dt\theta,
	\end{flalign*} Applying the change of variables $\theta\mapsto t$, and the triangle inequality, we have that 
	\begin{flalign*}
		\int_\Theta \|\theta\|\pi(\theta\mid S_n)\dt\theta&\le n^{-d_\theta/2}\int_{\mathcal{T}_n}\|t\||\pi(t\mid S_n)-N\{t;0,\Delta\}|\dt t+n^{-d_\theta/2}\int_{\mathcal{T}_n}\|t\|N\{t;0,\Delta\}\dt t+\|\theta_n\|\\&=o_p(1)+\|\theta_n\|, 
	\end{flalign*}where the  equality follows from $\mathcal{V}_{1n}=o_p(1)$ (\textbf{Part 1} of the result). We can therefore conclude that 
	\begin{eqnarray}\label{eq:ppp}
		m^{-1} \int_\Theta \|\theta\|\pi(\theta\mid S_n)\dt\theta\le o_p(1/m)+\|\theta_n\|/m=O_p(1/m) 
	\end{eqnarray}
	since, by Lemma \ref{lem:cons}, $\|\theta_n-\theta_\star\|=o_p(1)$. 
	Applying similar arguments to the proof of Lemma \ref{lem:orders}, and equation \eqref{eq:ppp}, to conclude that  
	$$
	\mathcal{V}_{2n}=\int_\Theta\|\theta\||\widehat{\pi}(\theta\mid S_n)-{\pi}(\theta\mid S_n)|\dt\theta\lesssim \frac{1}{m}\int_\Theta\|\theta\|\pi(\theta\mid S_n)\dt\theta+O_p(1/m)=o_p(1).
	$$
\end{proof}	

\begin{proof}[Proof of Corollary \ref{corr:bslm}]
	The change of variables $\theta\mapsto t={{\sqrt{n}}}(\theta-\theta_n)$ yields
	\begin{flalign*}
		\overline{\theta}_n&=\int_{\Theta} \theta\widehat{\pi}(\theta\mid S_n)\dt \theta=\int_{\mathcal{T}_n} ({t}/{{{\sqrt{n}}}}+\theta_n)\widehat{\pi}(t\mid S_n)\dt t=\frac{1}{{{\sqrt{n}}}}\int_{\mathcal{T}_n} t \widehat{\pi}(t\mid S_n)\dt t +\theta_n,	
	\end{flalign*}so that
	\begin{flalign}
		{{\sqrt{n}}}(\overline{\theta}_n-\theta_n)&=\int_{\mathcal{T}_n} t \widehat{\pi}(t\mid S_n)\dt t= \int_{\mathcal{T}_n} t \left[\widehat{\pi}(t\mid S_n)-N\{t;0,\Delta\}\right]\dt t+\int_{\mathcal{T}_n} t N\{t;0,\Delta\}\dt t.\label{eq:post_mean1}
	\end{flalign}The second term on the right-hand side of \eqref{eq:post_mean1} is zero by definition. Therefore,	by Theorem \ref{thm:two},
	\begin{flalign*}
		\left\|{{\sqrt{n}}}(\overline{\theta}_n-\theta_n)\right\|&=\left\| \int t \left[\widehat{\pi}(t\mid S_n)-N\{t;0,\Delta\}\right]\dt t\right\|\leq \int \|t\|\left|\widehat{\pi}(t\mid S_n)-N\{t;0,\Delta\}\right|\dt t=o_p(1).
	\end{flalign*} Furthermore, under Assumption \ref{ass:norm} and Lemma \ref{lemma:freq}, standard arguments can be used to show that 
	\begin{flalign*}
		{{\sqrt{n}}}(\theta_n-\theta_\star)\Rightarrow N(0,\Delta W_\star \Delta^\intercal).
	\end{flalign*}
	The preceding two display equations together yield the stated result. 
\end{proof}

\begin{proof}[Proof of Corollary \ref{corr:frac}]The proof follows a similar approach to the proof of Theorem \ref{thm:one}. In particular, define  $C_\Delta^\alpha= 1/\{{(2\pi)}^{d_\theta}|\Delta|^{}\}^{\alpha/2}$ and let $\Pi_\alpha(\gamma)=C_\Delta^\alpha\int_{\|t\|\le\gamma}\exp\left(-\frac{\alpha}{2}t^\intercal \Delta^{-1}t\right)\d t$. By the triangle inequality, 
	\begin{flalign*}
		&\left|\int_{\|t\|\le \gamma} \widehat{\pi}_\alpha(t\mid S_n)\d t-C_\Delta^\alpha\int_{\|t\|\le\gamma}\exp\left(-\frac{\alpha}{2}t^\intercal \Delta^{-1}t\right)\d t\right|\\&=\left|\int_{\|t\|\le \gamma} \widehat{\pi}_\alpha(t\mid S_n)\d t-\Pi_\alpha(\gamma)\right|\\&\leq\left|\int_{\|t\|\le\gamma}\left\{\widehat{\pi}_\alpha(t\mid S_n)-\pi_\alpha(t\mid S_n)\right\}\d t\right|+\left|\int_{\|t\|\le \gamma}{\pi}_\alpha(t\mid S_n)\d t-\Pi_\alpha(\gamma)\right|.
	\end{flalign*}
	Let $g_n^\alpha(S_n\mid\theta)=N\{S_n;b_n(\theta),\Sigma_n(\theta)\}^\alpha$, and recall that $\widehat{g}_n^\alpha(S_n\mid\theta):=\E_z[N\{S_n;\widehat{b}_n(\theta),\widehat{\Sigma}_n(\theta)\}^\alpha]$, so that we have 
	$$  
	\pi_\alpha(\theta\mid S_n)=\frac{g_n^\alpha(S_n\mid\theta)\pi(\theta)}{\int_\Theta g_n^\alpha(S_n\mid\theta)\pi(\theta)\dt\theta}, \text{ and }\widehat{\pi}_\alpha(\theta\mid S_n)=\frac{\widehat{g}_n^\alpha(S_n\mid\theta)\pi(\theta)}{\int_\Theta \widehat{g}_n^\alpha(S_n\mid\theta)\pi(\theta)\dt\theta}.
	$$
	Given these forms, the same argument used to prove Lemma \ref{lem:orders} (see Section \ref{app:lemmas}) allow us to conclude that 
	$$
	\int_{\|t\|\le\gamma}\left|\widehat{\pi}_\alpha(t\mid S_n)-\pi_\alpha(t\mid S_n)\right|\d t=O_p(1/m),\quad \text{for }m\rightarrow\infty\text{ as }n\rightarrow\infty.
	$$

	The stated result then follows  if $J_n=|\int_{\|t\|\le \gamma}{\pi}_\alpha(t\mid S_n)\d t-\Pi_\alpha(\gamma)|=o_p(1).$ To this end, recall the definition of $\theta_n$ as the value of $\theta\in\Theta$ such that $0=M_n(\theta)=\partial Q_n(\theta)/\partial\theta$. Since,  $0=\partial Q_n(\theta_n)/\partial\theta$, for  $Q^\alpha_n(\theta)=-\frac{\alpha}{2}\|\Sigma_{n}(\theta)^{-1/2}\{b_n(\theta)-S_n\}\|^2$, and for any fixed $\alpha\in[0,1]$, it follows that $0=\partial Q^\alpha_n(\theta_n)/\partial\theta$. Hence, we can rewrite $\pi_\alpha(\theta\mid S_n)$ as in the proof of Theorem \ref{thm:one}, namely, 
	\begin{flalign*}
		\pi_\alpha(\theta\mid S_n)&= \pi_\alpha(\theta_n\mid S_n)\left[{|n\Sigma_n^{}(\theta)|}/{|n\Sigma_n^{}(\theta_n)|}\right]^{-\alpha/2}\exp\left\{Q^\alpha_n(\theta)-Q^\alpha_n(\theta_n)\right\}\pi(\theta)/\pi(\theta_n).
	\end{flalign*} The form of the factional posterior $\pi_\alpha(\theta\mid S_n)$ is the same as in equation \eqref{eq:Clev_post},  up to the constant factor $\alpha$. So long as $\alpha\in[0,1]$ and fixed, the arguments  used to prove Theorem \ref{thm:one} remain valid and will deliver the stated result. For the sake of brevity, we do not replicate this argument.
\end{proof}

\begin{proof}[Proof of Corollary \ref{corr:absl}]
	Take $\theta\sim \widehat{\pi}(\theta\mid S_n)$, with $\widehat{\pi}(\theta\mid S_n)$ the naive  poster based on $\Sigma_n(\theta)=I_{d}/n$. Under this choice of $\Sigma_n(\theta)$, Theorem \ref{thm:two} implies that 
	$
	\|\Pi(t\in\cdot\mid S_n)-N\{0,\Delta\}\|_{\mathrm{TV}}=o_p(1),
	$ where $\|\cdot\|_{\mathrm{TV}}$ denotes the total variation distance, and $\Delta:=\{-H(\theta_{\bullet})\}^{-1}$, with $\theta_{\bullet}:=\arg\min_{\theta\in\Theta}\|b(\theta)-b_{\bullet}\|$. Recall $\tilde\theta=\bar\theta_n=\widehat\Delta_n+\widehat{W}_{\sqrt{n}}\widehat\Delta_n^{-1/2}(\theta-\bar\theta_n)$ and write
	\begin{flalign*}
		\vartheta = \sqrt{n}(\tilde\theta-\theta_n)&=\sqrt{n}(\bar\theta_n-\theta_n)+\widehat\Delta_n\widehat{W}^{1/2}_n\widehat\Delta_n^{-1/2}\sqrt{n}(\theta-\theta_n)-\widehat\Delta_n\widehat{W}^{1/2}_n\widehat\Delta_n^{-1/2}(\bar\theta_n-\theta_n)
		\\&=\sqrt{n}(\bar\theta_n-\theta_n)+\Delta W_\star \Delta^{-1/2}\sqrt{n}(\theta-\theta_n)-\widehat\Delta_n\widehat{W}^{1/2}_n\widehat\Delta_n^{-1/2}\sqrt{n}(\bar\theta_n-\theta_n)\\&-\left(\Delta W_\star \Delta^{-1/2}-\widehat\Delta_n\widehat{W}^{1/2}_n\widehat\Delta_n^{-1/2}\right)\sqrt{n}(\theta-\theta_n).
	\end{flalign*}
	From Corollary \ref{corr:bslm}, $\sqrt{n}(\bar\theta_n-\theta_n)=o_p(1)$ as $m,n\rightarrow\infty$. Further, since $\sqrt{n}(\theta-\theta_n)\mid S_n\Rightarrow N(0,\Delta)$ (Theorem \ref{thm:two}), for $\widehat{W}_n$ and $\widehat\Delta_n$ both consistent, by Slutsky's Theorem we have that, for $\theta\sim \widehat{\pi}(\theta\mid S_n)$, $$\left(\Delta W_\star \Delta^{-1/2}-\widehat\Delta_n\widehat{W}^{1/2}_n\widehat\Delta_n^{-1/2}\right)\sqrt{n}(\theta-\theta_n)=o_p(1).$$  Define
	$$
	x_n=\sqrt{n}(\bar\theta_n-\theta_n)-\widehat\Delta_n\widehat{W}^{1/2}_n\widehat\Delta_n^{-1/2}\sqrt{n}(\bar\theta_n-\theta_n) -\left(\Delta W_\star \Delta^{-1/2}-\widehat\Delta_n\widehat{W}^{1/2}_n\widehat\Delta_n^{-1/2}\right)\sqrt{n}(\theta-\theta_n),
	$$ so that, for $t=\sqrt{n}(\theta-\theta_n)$, with $\theta\sim \widehat{\pi}(\theta\mid S_n)$, 
	$$
	\vartheta = x_n+\Delta W_\star \Delta^{-1/2}\sqrt{n}(\theta-\theta_n) = x_n+\Delta W_\star \Delta^{-1/2}t.
	$$
	Since the total variation distance is invariant under a shift of location and a change of scale, we have the following string of equalities:
	\begin{flalign*}
		o_p(1)&=\|\Pi(t\in\cdot\mid S_n)-N\{0,\Delta\}\|_{\mathrm{TV}}
		\\&=\|\Pi(x_n+\Delta W_\star \Delta^{-1/2}t\in\cdot\mid S_n)-N\{x_n,\Delta W_\star^{1/2} \Delta^{-1/2}\Delta\Delta^{-1/2}W^{1/2}_\star\Delta\}\|_{\mathrm{TV}}\\&=\|\Pi(\vartheta\in\cdot\mid S_n)-N\{0,\Delta W_\star\Delta\}\|_{\mathrm{TV}}.
	\end{flalign*}The first equality is a restatement of Theorem \ref{thm:two}, the second equality follows from the invariance of $\|\cdot\|_{\mathrm{TV}}$ under location and scale changes, and the third follows from a change of variables.  
	
	To prove the second part of the result,  by Corollary \ref{corr:bslm} we again have  that $\sqrt{n}(\bar\theta_n-\theta_n)=o_p(1)$ for $m,n\rightarrow\infty$. Hence,  for consistent $\widehat{W}_n$ and $\widehat\Delta_n$, we have
	\begin{flalign*}	
		\sqrt{n}(\widetilde\theta_n-\theta_{\bullet})&=\sqrt{n}(\bar\theta_n-\theta_{\bullet})+\widehat\Delta_n\widehat{W}_{\sqrt{n}}\widehat\Delta_n^{-1/2}\sqrt{n}(\bar\theta_n-\bar\theta_n)\Rightarrow N(0,\Delta W_\star\Delta^\top),
	\end{flalign*}where the last line follows from the stated result in Corollary \ref{corr:bslm}.
	
\end{proof}

\section{Key Lemmas}\label{app:lemmas}The following lemmas are used to prove our main results.
\begin{lemma}\label{bound} Under Assumptions \ref{ass:three}-\ref{ass:five}, with probability converging to one,
	$$
	0\le {\pi(\theta_n\mid S_n)}/{{n^{{d_\theta}/2}}}\le {1}/\{{{(2\pi)^{d_\theta}|\Delta|^{}}}\}^{1/2}.
	$$
\end{lemma}
\begin{proof}
	The proof proceeds via a similar series of arguments to Theorem \ref{thm:one}, and shares arguments to those used to prove Lemma 2.1 in \cite{Chen1985}. Let $\gamma_n=o(1)$ with $\gamma_n{{\sqrt{n}}}\rightarrow\infty$, and let $\mathcal{N}_\gamma=\{\theta\in\Theta:\|\theta-\theta_n\|\le\gamma_n\}$. 
	Repeating arguments used in the proof of Theorem \ref{thm:one} over $\mathcal{N}_\gamma$, yields
	\begin{flalign*}
		\int_{\mathcal{N}_\gamma}\pi(\theta\mid S_n)\d\theta&=\pi(\theta_n\mid S_n)\{1+o_p(1)\}\int_{\mathcal{N}_\gamma}\exp\left\{Q_n(\theta)-Q_n(\theta_n)\right\}\d\theta\\&=\frac{\pi(\theta_n\mid S_n)}{{n^{d_{\theta}/2}}}\{1+o_p(1)\}\int_{\|t\|\le\gamma_n{{\sqrt{n}}}}\exp\{Q_n(\theta_n+t/{{\sqrt{n}}})-Q_n(\theta_n)\}\dt t,
	\end{flalign*}
	and, for $A_n$ as defined in that proof, up to an $o(1)$ term,
	\begin{flalign*}
		-t^{\top}\Delta^{-1}t/2-t^{\top}A_n\Delta^{-1}t/2\le Q_n(\theta_n+t/{{\sqrt{n}}})-Q_n(\theta_n)&\leq -t^{\top}\Delta^{-1}t/2+t^{\top}A_n\Delta^{-1}t/2.
	\end{flalign*}
	
	Let $M^{\pm}_n=I\pm\Delta^{1/2}A_n\Delta^{-1/2}$, and rewrite the above as
	\begin{flalign*}
		-t^{\top}\Delta^{-1/2}M^+_n\Delta^{-1/2}t/2\le Q_n(\theta_n+t/{{\sqrt{n}}})-Q_n(\theta_n)&\leq -t^{\top}\Delta^{-1/2}M^-_n\Delta^{-1/2}t/2.
	\end{flalign*}
	For $\|\theta_n-\theta_\star\|$ small enough, i.e., $n$ large enough, $\Delta^{-1/2}M^\pm_n\Delta^{-1/2}$ is positive-definite (with probability converging to one). Thus, for $n$ large we bound the posterior probability over $\mathcal{N}_\gamma$ as
	\begin{flalign*}
		\int_{\mathcal{N}_\gamma}\pi(\theta\mid S_n)\d\theta
		&\leq\frac{\pi(\theta_n\mid S_n)}{{n^{{d_\theta}/2}}}\left|\Delta^{-1/2}M_n^-\Delta^{-1/2}\right|^{-1/2} \int_{T^-_n}\exp(-x^{\top}x/2)\d x \\
		&\geq\frac{\pi(\theta_n\mid S_n)}{{n^{{d_\theta}/2}}}\left|\Delta^{-1/2}M_n^+\Delta^{-1/2}\right|^{-1/2}\int_{T^+_n}\exp(-x^{\top}x/2)\d x
	\end{flalign*}where $$T^{-}_n=\left\{x:\|x\|\leq \frac{\gamma_n{{\sqrt{n}}}}{\lambda_{\text{min}}(\Delta^{-1/2}M_n^-\Delta^{-1/2})^{1/2}}\right\},\quad T^{+}_n=\left\{x:\|x\|\leq \frac{\gamma_n{{\sqrt{n}}}}{\lambda_{\text{max}}(\Delta^{-1/2}M_n^+\Delta^{-1/2})^{1/2}}\right\},$$	and
	where we have used the fact that for any positive semi-definite matrix $M$ and $\gamma>0$
	$$
	\{x:\|x\|\le\gamma/{\lambda_{\text{max}}(M)}^{1/2}\}\subseteq\{x:x^\intercal Mx\le\gamma\}\subseteq	\{x:\|x\|\le\gamma/{\lambda_{\text{min}}(M)}^{1/2}\}.
	$$
	Under the restriction $\gamma_n {{\sqrt{n}}}\rightarrow\infty$, $T_n^{+}$ and $T_n^{-}$ both converge to $\mathbb{R}^{d_\theta}$ and $\int_{T^{\pm}_n}\exp(-x^\intercal x/2)\d x\rightarrow {(2\pi)}^{d_\theta/2}$ as $n\rightarrow\infty$. Hence, with probability converging to one,
	\begin{flalign*}
		|M_n^+|^{1/2}\int_{\mathcal{N}_\gamma}\pi_n(\theta\mid S_n)\d\theta\le	\frac{\pi(\theta_n\mid S_n)}{{n^{{d_\theta}/2}}}(2\pi)^{d_\theta/2}|\Delta|^{1/2}&\le |M_n^-|^{1/2}\int_{\mathcal{N}_\gamma}\pi(\theta\mid S_n)\d\theta.
	\end{flalign*}Since $|M_n^\pm|\rightarrow1$, $|\Delta|>0$ and $0\le\int_{\mathcal{N}_\gamma}\pi(\theta\mid S_n)\d\theta\le1$, with probability converging to one, $
	0\le\frac{\pi(\theta_n\mid S_n)}{{n^{{d_\theta}/2}}}\le 1/\{(2\pi)^{d_\theta}|\Delta|^{}\}^{1/2}.
	$
\end{proof}

\begin{lemma}\label{lem:orders}If the assumptions of Theorem \ref{thm:one} are satisfied, then for $m\rightarrow\infty$ as $n\rightarrow\infty$
	$$
	\int_{\|t\|\le\gamma} |\widehat{\pi}(t\mid S_n)-\pi(t\mid S_n)|\dt t=O_p(1/m).
	$$
\end{lemma}
\begin{proof}[Proof of Lemma \ref{lem:orders}]
	The proof is broken up into two parts. Firstly, we show that, for each $\theta\in\Theta$,  
	\begin{flalign}
		|\widehat{g}_n(S_n\mid\theta)-g_n(S_n\mid\theta)|\le g_n(S_n\mid\theta)\{1+O(1/m)\}+O(1/m)\label{eq:new1}.
	\end{flalign}We then use equation \eqref{eq:new1} to demonstrate the 
	stated result. 	
	
	Define $\widehat{ Q}_{n}(\theta)=-\frac{1}{2}\|\widehat{\Sigma}_n(\theta)^{-1/2}\left\{\widehat{b}_{n}(\theta)-S_n\right\}\|^2$ and $\overline{Q}_n(\theta)=-\frac{1}{2}\|{\Sigma}_n(\theta)^{-1/2}\left\{\widehat{b}_{n}(\theta)-S_n\right\}\|^2$. Assumption \ref{ass:two} and the triangle inequality imply that $\sup_{\theta \in \Theta}\|n\widehat{\Sigma}_n(\theta)-n{\Sigma}_n(\theta)\|=o_p(1)$, so that we have
	$
	\widehat{Q}_n(\theta)=\overline{Q}_n(\theta)\{1+o_p(1)\}, 
	$ as $m\rightarrow\infty$, for each $\theta\in\Theta$. 
	
	Let $\E_z$ denote expectation with respect to the simulated data $z^1,\dots,z^m$ at a fixed $\theta$ and $S_n$, where the dependence of $\E_z$ on $\theta$ and $S_n$ is suppressed for notational simplicity. 
	From standard properties of quadratic forms, 
	$$
	\E_z\{\overline{Q}_n(\theta)\}=Q_n(\theta)+O(1/m)
	$$while, under the moment hypothesis in Assumption \ref{ass:five}, Lemma \ref{lem:simple} implies
	\begin{equation}\label{eq:varbound}
		\mathrm{var}_z\{\overline{Q}_n(\theta)\}=\E_z\{\overline{Q}_n(\theta)^2\}-\E_z\{\overline{Q}_n(\theta)\}^2\le O(1/m).
	\end{equation}
	
	For $x\ge0$, expand $\exp(-x)$, around $\{x-\E(x)\}$, and let $\zeta$ be some Lagrange remainder lying between $x$ and $\E(x)$: 
	\begin{flalign*}
		\exp(-x)=&\exp\{-\E(x)\}-\exp\{-\E(x)\}\{x-\E(x)\}+\frac{1}{2}\exp\{-\zeta\}\{x-\E(x)\}^2\\&\le \exp\{-\E(x)\}[1-\{x-\E(x)\}]+\{x-\E(x)\}^2,
	\end{flalign*}
	where the inequality follows since $\exp(-x)$ is bounded by unity on $x\ge0$, and since $\zeta\ge0$ by construction. Now, taking $x=-\overline{Q}_n(\theta)$ in the above expansion yields
	\begin{flalign*}
		\exp\{\overline{Q}_n(\theta)\}&\le \exp\left\{{Q}_n(\theta)\right\}\{1+O(m^{-1})\}-\exp\left[\E_z\left\{\overline{Q}_n(\theta)\right\}\right][\overline{Q}_n(\theta)-\E_z\left\{\overline{Q}_n(\theta)\right\}]\\&+[\overline{Q}_n(\theta)-\E_z\left\{\overline{Q}_n(\theta)\right\}]^2.
	\end{flalign*}Taking expectations of both sides, and applying equation \eqref{eq:varbound} yields
	\begin{flalign*}
		\mathbb{E}_z\left[\exp \left\{\overline Q_{n}(\theta)\right\}\right]&\le  \exp\left\{Q_n(\theta)\right\}\left\{1+O\left({1}/{m}\right)\right\}+O(1/m).
	\end{flalign*}
	Recalling the definitions $\widehat{g}_n(S_n\mid\theta)$ and $g_n(S_n\mid\theta)$, apply the above equation to obtain
	\begin{flalign*}
		|\widehat{g}_n(S_n\mid\theta)-g_n(S_n\mid\theta)|\leq g_n(S_n\mid\theta)\left\{O\left({1}/{m}\right)\right\}+O(1/m).
	\end{flalign*}
	
	We now use \eqref{eq:new1} to obtain the stated result. Rewrite the difference of the posteriors over $\mathcal{N}_\gamma=\{t:\|t\|\le\gamma\}$ as follows, 
	\begin{flalign*}
		\int_{\mathcal{N}_\gamma} |\widehat{\pi}(\theta\mid S_n)-\pi(\theta\mid S_n)|\d\theta&=\int_{\mathcal{N}_\gamma} \left|\frac{\widehat{g}_n(S_n\mid\theta)\pi(\theta)}{\int_\Theta \widehat{g}_n(S_n\mid\theta)\pi(\theta)}-\frac{{g}_n(S_n\mid\theta)\pi(\theta)}{\int_\Theta {g}_n(S_n\mid\theta)\pi(\theta)}\right|\dt\theta	\\&=\int_{\mathcal{N}_\gamma} \bigg{|}\frac{\{\widehat{g}_n(S_n\mid\theta)-{g}_n(S_n\mid\theta)\}\pi(\theta)}{\int_\Theta \widehat{g}_n(S_n\mid\theta)\pi(\theta)}\frac{\int_\Theta {g}_n(S_n\mid\theta)\pi(\theta)\d\theta}{{\int_\Theta {g}_n(S_n\mid\theta)\pi(\theta)\d\theta}}\\&-{{g}_n(S_n\mid\theta)\pi(\theta)}\left(\frac{1}{\int_\Theta {g}_n(S_n\mid\theta)\pi(\theta)\d\theta}-\frac{1}{\int_\Theta \widehat{g}_n(S_n\mid\theta)\pi(\theta)\d\theta}\right)\bigg{|}\dt\theta
	\end{flalign*}
	and apply the triangle inequality twice to obtain, 
	\begin{flalign*}
		\int_{\mathcal{N}_\gamma} |\widehat{\pi}(\theta\mid S_n)-\pi(\theta\mid S_n)|\d\theta&\leq \int_{\mathcal{N}_\gamma}\frac{|\widehat{g}_n(S_n\mid\theta)-{g}_n(S_n\mid\theta)|\pi(\theta)}{\int_\Theta \widehat{g}_n(S_n\mid\theta)\pi(\theta)}\dt\theta\\&+\int_{\mathcal{N}_\gamma} \frac{{g}_n(S_n\mid\theta)\pi(\theta)}{\int_\Theta {g}_n(S_n\mid\theta)\pi(\theta)\d\theta}\left(1-\frac{\int_\Theta {g}_n(S_n\mid\theta)\pi(\theta)\d\theta}{\int_\Theta \widehat{g}_n(S_n\mid\theta)\pi(\theta)\d\theta}\right)\dt\theta\\& \leq \int_{\mathcal{N}_\gamma}\frac{|\widehat{g}_n(S_n\mid\theta)-{g}_n(S_n\mid\theta)|\pi(\theta)}{\int_\Theta \widehat{g}_n(S_n\mid\theta)\pi(\theta)}\dt\theta+\left(1-\frac{\int_\Theta {g}_n(S_n\mid\theta)\pi(\theta)\d\theta}{\int_\Theta \widehat{g}_n(S_n\mid\theta)\pi(\theta)\d\theta}\right)
	\end{flalign*}where the second inequality uses the fact that $0\le\int_{\mathcal{N}_\gamma}\pi(\theta\mid S_n)\le1$. 
	Apply equation \eqref{eq:new1} twice to obtain
	\begin{flalign*}
		&\int_{\mathcal{N}_\gamma}\frac{|\widehat{g}_n(S_n\mid\theta)-{g}_n(S_n\mid\theta)|\pi(\theta)}{\int_\Theta \widehat{g}_n(S_n\mid\theta)\pi(\theta)}+\left(1-\frac{\int_\Theta {g}_n(S_n\mid\theta)\pi(\theta)\d\theta}{\int_\Theta \widehat{g}_n(S_n\mid\theta)\pi(\theta)\d\theta}\right)\\&= \int_{\mathcal{N}_\gamma}\frac{|\widehat{g}_n(S_n\mid\theta)-{g}_n(S_n\mid\theta)|\pi(\theta)}{\int{g}_n(S_n\mid\theta)\pi(\theta)\d\theta\left\{1+O(1/m)\right\}}+\left(1-\frac{\int_\Theta {g}_n(S_n\mid\theta)\pi(\theta)\d\theta}{\int_\Theta {g}_n(S_n\mid\theta)\pi(\theta)\d\theta\left\{1+O(1/m)\right\}}\right)+O(1/m)\\&=\frac{1}{m}\int_{\mathcal{N}_\gamma}\pi(\theta\mid S_n)\dt\theta+O(1/m).
	\end{flalign*}Since  $0\le\int_{\mathcal{N}_\gamma}\pi(\theta\mid S_n)\le1$, the result follows. 
\end{proof}

The following result is used in Lemma \ref{lem:orders} to upper bound the variance of the log synthetic likelihood, which can be expressed as a specific type of quadratic form. As we are unaware of a general reference for the variance of this type of quadratic form, we state (and prove) the following general result.
\begin{lemma}\label{lem:simple}
	Let $Z\in\mathbb{R}^d$ be a random variable with mean $\mu=\E[Z]$, variance matrix $\Sigma$, and finite fourth moments. Let $Z_1,\dots,Z_m$ denote independent and identically distributed copies of $Z$, and let $\overline{Z}=\frac{1}{m}\sum_{i=1}^{m}Z_i$. For $M$ a $(d\times d)$-dimensional positive semi-definite matrix, as $m\rightarrow\infty$,
	$$
	\mathrm{var}(\overline{Z}^\intercal M\overline{Z})\le O(1/m)+O(1/m^{3/2})+O(1/m^2).
	$$
\end{lemma}

\begin{proof}
	Writing $M=A^\intercal A$, for some square-root matrix $A$,
	\begin{flalign*}
		\|A\overline{Z}\|^2=\overline{Z}^{\top}M\overline{Z}&=(\overline{Z}-\mu)^{\top}M(\overline{Z}-\mu)-\mu^{\top}M\mu+2\mu^{\top}M\overline{Z}=\|A(\overline{Z}-\mu)\|^2-\|A\mu\|^2+2\mu^{\top}M\overline{Z},	
	\end{flalign*}
	and
	\begin{flalign*}
		\mathrm{var}\{\|A\overline{Z}\|^2\}&=\mathrm{var}\{\|A(\overline{Z}-\mu)\|^2\}+4\mathrm{var}(\mu^{\top}M\overline{Z})+2\text{cov}\{2\mu^{\top}M\overline{Z},\|A(\overline{Z}-\mu)\|^2\}.
	\end{flalign*}From $\mathrm{var}(\mu^{\top}M\overline{Z})=\mu^\intercal M\Sigma M\mu/m$, and Cauchy-Schwartz,
	\begin{flalign}
		\mathrm{var}(\|A\overline{Z}\|^2)&\le\mathrm{var}\{\|A(\overline{Z}-\mu)\|^2\}+4\{\mu^{\top}M\Sigma M\mu\}/m\nonumber\\&+2\{4(\mu^{\top}M\Sigma M\mu)/m\}^{1/2}[{\mathrm{var}\{\|A(\overline{Z}-\mu)\|^2\}}]^{1/2}.\label{eq:newvar}
	\end{flalign}
	
	Fist, let us bound $(\mu^{\top}M\Sigma M\mu)/m$. Recall that, for any positive semi-definite and symmetric matrices $A,B$, of compatible size, 
	\begin{equation}\label{eq:trace}
		\text{Tr}(AB)^2\le \text{Tr}(A^2)\text{Tr}(B^2)\le[\text{Tr}(A)]^2[\text{Tr}(B)]^2. 	
	\end{equation}Using positive semi-definiteness and symmetry of $M$ and $\Sigma$, and applying (twice) equation \eqref{eq:trace}, we have
	\begin{flalign*}
		\mu^{\top}M(\Sigma/m)M\mu=\text{Tr}(\mu^{\top}M\Sigma M\mu)/m&=\text{Tr}(M\Sigma M\mu\mu^{\top})/m\\&\le \text{Tr}(M\Sigma M)\text{Tr}(\mu\mu^{\top})/m\\&\le  \|\mu\|^2\text{Tr}(M^2)\text{Tr}(\Sigma)/m,
	\end{flalign*}where the last line follows since $\text{Tr}(\mu\mu^\intercal)=\|\mu\|^2$. From Assumption \ref{ass:five}, $\|\mu\|^2\lesssim d$, and $\tr(\Sigma)\lesssim d$, and we obtain
	\begin{flalign}\label{eq:simp2} 
		\mu^{\top}M(\Sigma/m)M\mu& \le  \|\mu\|^2\text{Tr}(M^2)\text{Tr}(\Sigma)/m\le d^2C \text{Tr}(M^2)/m = O\{\|M\|^2/m\},
	\end{flalign} where the last inequality follows by \eqref{eq:trace} and $\text{Tr}(M^2)\le \|M\|^2$.

	Now, we upper bound $\mathrm{var}\{\|A(\overline{Z}-\mu)\|\}$ using
	$$
	\mathrm{var}\{\|A(\overline{Z}-\mu)\|^2\}\le \E\{\|A(\overline{Z}-\mu)\|^4\}+[\E\{\|A(\overline{Z}-\mu)\|^2\}]^2.
	$$
	and the result follows by bounding both terms.  Firstly, for $\overline{z}_j$ (resp., $\mu_j$) denoting the $j$-th marginal component of $Z$ (resp., $\mu$),
	\begin{flalign*}
		\E\{\|A(\overline{Z}-\mu)\|^4\}&=\E \text{Tr}\{(\overline{Z}-\mu)^\intercal A^\intercal A(\overline{Z}-\mu)\}^2=\E\text{Tr}\{A^\intercal A(\overline{Z}-\mu)(\overline{Z}-\mu)^\intercal \}^2 \\&\le \E\text{Tr}\{(\overline{Z}-\mu)(\overline{Z}-\mu)^\intercal\}^2\text{Tr}(M^2)\\&=\text{Tr}\E\left[\left(\sum_{j=1}^{d}(\overline{z}_j-\mu_j)^2\right)^2\right]\text{Tr}(M^2).
	\end{flalign*}Using H\"{o}lders-inequality to bound the cross-terms we have, for some $C>0$,
	\begin{flalign*}
		\E\left[\left(\sum_{j=1}^{d}(\overline{z}_j-\mu_j)^2\right)^2\right]&\le \sum_{j=1}^{d}\E(\overline{z}_j-\mu_j)^4+\sum_{i\ne j}[\E(\overline{z}_i-\mu_i)^4]^{1/2}[\E(\overline{z}_j-\mu_j)^4]^{1/2}
	\end{flalign*}Applying, in-turn, the previous display equations we have
	\begin{flalign*}
		\mathrm{var}\{\|A(\overline{Z}-\mu)\|\}&
		\lesssim \max_{j\in\{1,\dots,d\}}\E\left[(\bar{z}_j-\mu_j)^4\right].
	\end{flalign*}From, e.g., Corollary 4.3 in  \cite{severini2005elements}, for $z_j$ (and $\mu_j$) the $j$-th element of $Z$ (and $\mu$),
	$$
	\E[(\bar{z}_j-\mu_j)^4]=\frac{3}{m^{2}} \E\{(z_j-\mu_j)^2\}^2+\frac{1}{m^{3}}\left[\E\{(z_j-\mu_j)^4\}-3 \E\{(z_j-\mu_j)^2\}^{2}\right].
	$$Hence, for $m$ large,
	$\mathrm{var}\{\|A(\overline{Z}-\mu)\|\}\le O(1/m^2)+O(1/m^3).$
	Applying this bound into equation \eqref{eq:newvar} and re-arranging orders we obtain
	\begin{flalign*}
		\mathrm{var}\{\|A(\overline{Z}-\mu)\|\}\le O(1/m)+O(1/m^{3/2})+O(1/m^2).
	\end{flalign*}
	
\end{proof}

\section{Additional results and discussion}\label{app:supp}
\subsection{Additional Results}
In this section, we collect several frequentist results that are helpful for proving the main results in the paper. Before presenting the results, we first recall some key definitions. The exact log synthetic likelihood criterion is  $\log g_n(S_n\mid\theta)$, which, for some $C$ that does not depend on $\theta$, can be written as 
\begin{flalign*}
	\log\{g_n(S_n\mid\theta)\}&=-\log\{|\Sigma_n(\theta)|\}-n\{b_n(\theta)-S_n\}^{\top}\left\{n\Sigma^{}_n(\theta)\right\}^{-1}\{b_n(\theta)-S_n\}/2+C\\
	&=-\log\{|\Sigma_n(\theta)|\}+Q_n(\theta)+C.\end{flalign*}
For $M_n(\theta)=n^{-1}\partial \log g_n(S_n\mid\theta)/\partial\theta$, using the fact that $n\Sigma_n(\theta)=\Sigma(\theta)+o_p(1)$ and $b_n(\theta)=b(\theta)+o_p(1)$, uniformly over $\Theta$ (Assumption \ref{ass:two}), $M_n(\theta)$ can be stated as (up to $o_p(1)$ terms)
\begin{flalign}
	M_n(\theta)=&-\left(\text{tr}\left\{\Sigma^{-1}(\theta)\Lambda_1(\theta)\right\},\dots,\text{tr}\left\{\Sigma^{-1}(\theta)\Lambda_{d_\theta}(\theta)\right\}\right)^{\top}/n\nonumber\\&+\{(\partial/\partial\theta^{\top}) b(\theta)\}^{\top}\Sigma^{-1}(\theta)\{b(\theta)-S_n\}\nonumber\\&-\left\{(\partial/\partial\theta^{\top}) \text{Vec}\left[\Sigma(\theta)\right]\right\}^{\top}\left[\Sigma^{-1}(\theta)\otimes\Sigma^{-1}(\theta)\right]\times\text{Vec}\left[\left\{S_n-b(\theta)\right\}\left\{S_n-b(\theta)\right\}^{\top}\right]\label{eq:bslfoc},
\end{flalign}where $\Lambda_j(\theta)=\partial \Sigma(\theta)/\partial\theta_j$, and where we recall that $b(\theta)= b_n(\theta)+o_p(1)$, for all $\theta\in\Theta$. In the case of scalar $\theta$, $M_n(\theta)$ has the more analytically useful representation 
\begin{flalign*}
	M_n(\theta)=-\frac{\text{tr}\left\{\Sigma^{-1}(\theta)\Lambda(\theta)\right\}}{n}&-\left\{(\partial/\partial\theta^{\top}) b(\theta)\right\}^{\top}\Sigma^{-1}(\theta)\{b(\theta)-S_n\}\\&+\left\{S_n-b(\theta)\right\}^{\top}\Sigma^{-1}(\theta)\Lambda(\theta)\Sigma^{-1}(\theta)\{S_n-b(\theta)\}.
\end{flalign*}
where $\Lambda(\theta)=\d \Sigma(\theta)/\d\theta$. 

Using the scalar representation of $M_n(\theta)$, a (relatively simple) analytical expression for the Hessian matrix $H_n(\theta)=\partial M_n(\theta)/\partial\theta'$ can be constructed by concatenating the scalar partial derivatives. Even in this simple case, the matrix $H(\theta_\star)$ has a complicated structure that depends on the second derivatives of $\Sigma(\theta),\;b(\theta)$, as well as the level of model misspecification $\{b_{\bullet}-b(\theta_\star)\}$. Therefore, unless $\{b_{\bullet}-b(\theta_\star)\}=0$ there is no way in general for $H(\theta_{\star})$ to satisfy a generalized information matrix equality; i.e., it will not be the case that $W_\star=-H(\theta_{\star})$ when $\|b_{\bullet}-b(\theta_\star)\|>0$. 
	
	\begin{lemma}\label{lemma:freq}
		Under Assumptions \ref{ass:three}-\ref{ass:five}, the following are satisfied.

		\noindent1. For some $\theta_\star\in\Theta_\star$, the estimator $\theta_n$ exists and satisfies $\|\theta_n-\theta_\star\|=o_p(1)$.
		
		\noindent2. If in addition to Assumptions \ref{ass:three}-\ref{ass:five}, Assumption \ref{ass:norm} is satisfied, then $\|\theta_n-\theta_\star\|=O_p(n^{-1/2})$.  
		
		\noindent3. For any $\delta_n=o(1)$, $T_n=\{\theta\in\Theta,\theta_\star\in\Theta_\star:\|\theta-\theta_\star\|\leq \delta_n\}$, and $t={{\sqrt{n}}}(\theta-\theta_\star),\; \theta\in T_n$,
		$$
		Q_n(\theta_\star+t/{{\sqrt{n}}})-Q_n(\theta_\star)=t^{\top}{{\sqrt{n}}}M_n(\theta_\star)+\frac{1}{2}t^{\top}H(\theta_\star)t\{1+O_p(\delta_n)+O_p(\|Z_n/{{\sqrt{n}}}\|)\}.
		$$
	\end{lemma}
	\begin{proof}[Proof of 1]
		The result follows from verifying the sufficient conditions in Theorem 2 of \cite{yuan1998asymptotics}. Firstly, from the definition of $M_n(\theta)$ given in equation \eqref{eq:bslfoc}, and our assumptions on $\Theta_\star$, we can show that $M_n(\theta_\star)=o_p(1)$. Secondly, from Assumptions \ref{ass:three}, $M_n(\theta)$ is continuously differentiable for all $\|\theta-\theta_\star\|\le\delta$, and some $\delta>0$. Moreover, from the definition of $H_n(\theta)$ and Assumption \ref{ass:one}, we conclude that 
		$
		\|H_n(\theta) - H(\theta)\|\leq O_p(\|S_n-b_{\bullet}\|)=o_p(1),
		$ for all $\|\theta-\theta_\star\|\le\delta$. Moreover, $H(\theta_\star)$ is non-singular by Assumption \ref{ass:three}. This verifies the sufficient conditions in \cite{yuan1998asymptotics} and we can conclude that: 1) $\theta_n$ exists for $n$ large enough; 2) $\theta_n$ satisfies, $\|\theta_n-\theta_\star\|=o_p(1)$. 
	\end{proof}
	\begin{proof}[Proof of 2]	
		To simplify the proof we only consider the case of scalar $\theta$, which allows us to make use of the specific form of $M_n(\theta)$ in the scalar case. The result can be extended by applying a similar argument dimension-by-dimension.

		Firstly, from consistency of $\theta_n$ there exists some positive $\delta_n=o(1)$	such that $$P_{\bullet}^{(n)}\left\{\|\theta_n-\theta_\star\|\ge\delta_n\right\}=o(1).$$ With $P^{(n)}_{\bullet}$ - probability converging to one for this sequence, we first show that  
		\begin{flalign}\label{eq:res1}
			\sup_{\|\theta-\theta_\star\|\le\delta_n}\|M_n(\theta)-M(\theta)-M_n(\theta_\star)\|=o_p(n^{-1/2}).
		\end{flalign}
		From the definition of $M_n(\theta)$ and $M(\theta)$ in the univariate case, for $G(\theta)=\dt b(\theta)/\dt \theta$, 
		\begin{flalign*}
			M_n(\theta)&=M(\theta)+\left[G(\theta)^{\top}\Sigma(\theta)^{-1}+2\{b(\theta)-b_{\bullet}\}^{\top}\Sigma(\theta)^{-1}\Lambda(\theta)\Sigma(\theta)^{-1}\right]\{b_{\bullet}-S_n\}\\&-\{b_{\bullet}-S_n\}^{\top}\Sigma(\theta)^{-1}\Lambda(\theta)\Sigma^{-1}(\theta)\{b_{\bullet}-S_n\}\\
			&=M(\theta)+O_p(1/{{\sqrt{n}}})+O_p(1/n),
		\end{flalign*}where the second equality follows from Assumption \ref{ass:one}. Using $M(\theta_\star)=0$,
		\begin{flalign*}
			M_n(\theta_\star)&=\left[G(\theta_\star)^{\top}\Sigma(\theta_\star)^{-1}+2\{b(\theta_\star)-b_{\bullet}\}^{\top}\Sigma(\theta_\star)^{-1}\Lambda(\theta_\star)\Sigma(\theta_\star)^{-1}\right]\{b_{\bullet}-S_n\}\\&-\{b_{\bullet}-S_n\}^{\top}\Sigma(\theta_\star)^{-1}\Lambda(\theta_\star)\Sigma(\theta_\star)^{-1}\{b_{\bullet}-S_n\}\\&=\left[G(\theta_\star)^{\top}\Sigma(\theta_\star)^{-1}+2\{b(\theta_\star)-b_{\bullet}\}^{\top}\Sigma(\theta_\star)^{-1}\Lambda(\theta_\star)\Sigma(\theta_\star)^{-1}\right]\{b_{\bullet}-S_n\}+O_p(1/n).
		\end{flalign*}
		For $B(\theta)=\Sigma^{-1}(\theta)\Lambda(\theta)\Sigma^{-1}(\theta)$, $e(\theta)=b(\theta)-b_{\bullet}$, $X(\theta)=G(\theta)'\Sigma^{-1}(\theta)$, and $Z_n={{\sqrt{n}}}\{S_n-b_{\bullet}\}$
		\begin{flalign*}
			\|M_n(\theta)-M(\theta)-M_n(\theta_\star)\|&\leq \|Z_n/{{\sqrt{n}}}\|^2\|B(\theta)-B(\theta_\star)\|\\&+\|Z_n/{{\sqrt{n}}}\|\left\|\left[X(\theta)+2e(\theta)^{\top}B(\theta)\right]-\left[X(\theta_\star)+2e(\theta_\star)^{\top}B(\theta_\star)\right]\right\|.
		\end{flalign*}By Assumption \ref{ass:hess}, $B(\theta)$, $e(\theta)$, and $X(\theta)$ are Lipschitz in a neighborhood of $\theta_\star$. Therefore, for $\delta_n$ as above, all $\|\theta-\theta_\star\|\leq\delta_n$, and some $C>0$,  
		\begin{flalign*}
			\|M_n(\theta)-M(\theta)-M_n(\theta_\star)\|&\leq C\|Z_n/{{\sqrt{n}}}\|\|\theta-\theta_\star\|\{1+\|Z_n/{{\sqrt{n}}}\|\}.
		\end{flalign*}Applying the fact that $Z_n/{{\sqrt{n}}}=O_p(1/{{\sqrt{n}}})$, this proves \eqref{eq:res1}. 
		
		With $P^{(n)}_{\bullet}$ - probability converging to one for the sequence $\delta_n$, we then have 
		\begin{flalign*}
			\|M_n(\theta_n)-M(\theta_n)-M_n(\theta_\star)\|&\leq o_p(n^{-1/2})\\&\geq \|M(\theta_n)\|-\|M_n(\theta_n)\|-\|M_n(\theta_\star)\|.
		\end{flalign*}Rearranging terms, and applying Assumption \ref{ass:norm},
		\begin{flalign*}
			\|M(\theta_n)\|&\leq o_p(n^{-1/2})+\|M_n(\theta_\star)\|\{1+o_p(1)\}=O_p(n^{-1/2}).
		\end{flalign*}From the differentiability of $M(\theta)$, and the positive-definiteness of $\{-H(\theta_\star)\}$, there exists $C$ such that 
		$
		C\|\theta_n-\theta_\star\|\leq	\|M(\theta_n)\|\leq O_p(n^{-1/2}).
		$
	\end{proof}
	
	\begin{proof}[Proof of 3]On the set $T_n$, the result follows from a Taylor expansion of $Q_n(\theta)=-\frac{1}{2}\{b_n(\theta)-S_n\}^{\top}\Sigma_{n}^{-1}(\theta)\{b_n(\theta)-S_n\}$ around $\theta_\star$, which for $M_n(\theta)$ and $H_n(\theta)$ as defined in Section \ref{sec:theory},  gives 
		\begin{flalign*}
			Q_n(\theta)=Q_n(\theta_\star)+{{\sqrt{n}}}(\theta-\theta_\star)'{{\sqrt{n}}}M_n(\theta_\star)+{{\sqrt{n}}}(\theta-\theta_\star)'H_n(\tilde\theta){{\sqrt{n}}}(\theta-\theta_\star)/2,
		\end{flalign*}for $\tilde\theta$ a term-by-term intermediate value such that $\|\tilde\theta-\theta_\star\|\leq \|\theta-\theta_\star\|$, and where $\theta\in T_n$. 
		
		Rewrite the second term as 
		\begin{flalign*}
			\frac{1}{2}t^\intercal H_n(\tilde\theta)t&=	\frac{1}{2}t^\intercal H(\theta_\star)\left\{1+H(\theta_\star)^{-1}[H_n(\tilde\theta)-H(\tilde\theta)]+H(\theta_\star)^{-1}[H(\tilde\theta)-H(\theta_\star )]\right\}t.
		\end{flalign*}
		From the definition of $H(\theta)$, and the twice continuous differentiability hypothesis on $Q_n(\theta)$, $H_n(\theta)$ is Lipschitz in this neighbourhood and we have that:\footnote{By Assumption \ref{ass:two}, the map $H_n(\theta)=\partial M_n(\theta)/\partial\theta'$ is continuously differentiable in a neighbourhood of $\theta_\star$, for some $\delta_n$, and any $\theta_\star\in\Theta_\star$. Therefore, for each $\theta_\star\in\Theta$, $H_n(\theta)$ is Lipschitz, with (possibly) differing Lipschitz constant, in this neighbourhood.} for $Z_n={{\sqrt{n}}}\{S_n-b_{\bullet}\}$
		\begin{flalign*}
			\|H_n(\tilde\theta)-H_n(\theta_\star)\|+\|H_n(\theta_\star)-H(\theta_\star)\|&\leq O_p(\|\theta-\theta_\star\|)+O_p(\|Z_n/{{\sqrt{n}}}\|)\\&= O_p(\delta_n+\|Z_n/{{\sqrt{n}}}\|).
		\end{flalign*}From the definition $t={{\sqrt{n}}}(\theta-\theta_\star)$, rearranging terms yields 
		\begin{flalign*}
			Q_n(\theta_\star+t/{{\sqrt{n}}})=Q_n(\theta_\star)+t^{\top}{{\sqrt{n}}}M_n(\theta_\star)+\frac{1}{2}t^{\top}H(\theta_\star)t\{1+O_p(\delta_n+\|Z_n/{{\sqrt{n}}}\|)\}.
		\end{flalign*}

	\end{proof}
	
	\section{Additional details and discussion}\label{app:discuss}
	\subsection{Comparison with results for correctly specified models}\label{rem:compare}
	In this section, we compare our results for BSL in misspecified models  against existing results for correctly specified models obtained in \cite{frazier2019bayesian}. As mentioned in the main text, Assumptions \ref{ass:one}, \ref{ass:two} and \ref{ass:four}, are similar to the assumptions employed by \cite{frazier2019bayesian} to deduce a Bernstein-von Mises result for $\widehat{\pi}(\theta\mid S_n)$ in \textit{correctly specified models}. However, Assumptions \ref{ass:three} and \ref{ass:hess} are stronger versions of the corresponding smoothness assumptions used in \cite{frazier2019bayesian}. In particular, the only differentiability condition required by \cite{frazier2019bayesian} was differentiability of the map $\theta\mapsto b(\theta)$, where we recall that $b(\theta)=\lim_n b_n(\theta)=\lim_n\E_z\{S_n(z)\mid\theta\}$, while our results require that $b(\theta)$ and $\Sigma_n(\theta)$ be twice continuously differentiable with well behaved derivatives (see Assumptions \ref{ass:three} and \ref{ass:hess} in the main text). 
	
	In misspecified models, the behavior of the log synthetic likelihood, $\log g_n(S_n\mid\theta)$, differs markedly from the case of correct specification analysed in \cite{frazier2019bayesian}, and the stronger smoothness conditions maintained in the main text are necessary to deal with the misspecified setting. When the model is correctly specified,  \cite{frazier2019bayesian} show that the BSL posterior can be controlled using a certain global approximation to the quadratic term in the log synthetic likelihood:
	$$
	Q_n(\theta)=\frac{n}{2}\{b(\theta)-S_n\}^{\top}\left\{n\Sigma^{}_n(\theta)\right\}^{-1}\{b(\theta)-S_n\}. 
	$$ More particularly, \cite{frazier2019bayesian} show that in correctly specified models the posterior $\pi(\theta\mid S_n)$ can be controlled through the following quadratic approximation of $Q_n(\theta)$: for some $\theta_{\bullet}\in\Theta$ such that $b(\theta_{\bullet})=b_{\bullet}$, 
	\begin{flalign}\label{eq:fraz_exp}
		Q_n(\theta)-Q_n(\theta_{\bullet})=&	n\{b(\theta_{\bullet})-S_n\}^\top \Sigma_n(\theta)^{-1}\{\partial b(\theta_{\bullet})/\partial\theta^\top\}(\theta-\theta_{\bullet})\\&-\frac{n}{2}(\theta-\theta_{\bullet})^\top \{\partial b(\theta_{\bullet})/\partial\theta^\top\}^\top \Sigma_n(\theta)^{-1} \{\partial b(\theta_{\bullet})/\partial\theta^\top\}(\theta-\theta_{\bullet})+R_n(\theta)\nonumber,
	\end{flalign}for some remainder term $R_n(\theta)$ that can be appropriately controlled.   
	
	As such, the large sample results of \cite{frazier2019bayesian}, and the proof arguments used to obtain them, ultimately depend on the validity and accuracy of the approximation in equation \eqref{eq:fraz_exp}. Critically, this global approximation is invalid if there does not exist a $\theta_{\bullet}\in\Theta$ such that $b(\theta_{\bullet})=b_{\bullet}$. To see this note that when there is no $\theta_{\bullet}$ such that $b(\theta_{\bullet})=b_{\bullet}$, the term $\sqrt{n}\{b(\theta_{\bullet})-S_n\}$ in equation \eqref{eq:fraz_exp} diverges: for any $\theta\in\Theta$, 
	$$
	\sqrt{n}\{b(\theta)-S_n\}=\sqrt{n}\{b(\theta)-b_{\bullet}\}+\sqrt{n}\{b_{\bullet}-S_n\},
	$$and when the model is misspecified $$\inf_{\theta\in\Theta}\|b(\theta)-b_{\bullet}\|>0,$$ so that the sequence $\|\sqrt{n}\{b(\theta)-S_n\}\|\rightarrow+\infty$. 
	Since $\sqrt{n}\{b(\theta)-S_n\}$  is not asymptotically bounded for any $\theta\in\Theta$, the difference $Q_n(\theta)-Q_n(\theta_{\bullet})$ cannot be controlled, and the expansion in equation \eqref{eq:fraz_exp}, \textit{which underpins each result in} \cite{frazier2019bayesian},  cannot be used  to control the posterior $\pi(\theta\mid S_n)$ in misspecified models. 
	
	Consequently, in both the multi-modal and uni-modal cases considered in this current work, the results obtained in \cite{frazier2019bayesian} for correctly specified models cannot be used to establish the behavior of the posterior in misspecified models. 
	
	In contrast to the expansion in \eqref{eq:fraz_exp}, control of the BSL posterior in the misspecified case requires control on $M_n(\theta)$ and $H_n(\theta)$, which depend on the first and second derivatives of $b(\theta)$ and $\Sigma_n(\theta)$; see Section \ref{app:supp} for details. As such, controlling the behavior of the posterior in misspecified models requires controlling both the first and second derivatives of $\log g_n(S_n\mid\theta)$, which differs markedly from the correctly specified case where we only need to control the first derivative of $b(\theta)$ (i.e., equation \eqref{eq:fraz_exp}). We speculate that without appropriate control of the second-order derivative terms, it is unlikely that large samples results similar to those in Theorem \ref{thm:one} and Theorem \ref{thm:two} can be obtained in misspecified models.

	\subsection{Comparison: adjusted posteriors}\label{app:adjust}
	
	In this subsection, we compare the adjustment procedure proposed in Section \ref{sec:rob_adjust} in the main text, with the adjustment approach for correctly specified models proposed in Section 4 of \cite{frazier2019bayesian}. First, we recall the two-steps used to produce samples from the adjusted posterior proposed in Section \ref{sec:rob_adjust} of the main text. 
	
	\medskip 
	
	\noindent1.  Set $\widehat{\Sigma}_n(\theta)=I_d/n$, for all $\theta\in\Theta$, as the variance matrix in the synthetic likelihood and obtain the corresponding naive posterior mean, $\bar{\theta}_n$, and variance, $\widehat{\Delta}_n$. 
	
	\medskip 
	
	\noindent2. For $\theta^j$, $j=1,\dots,N$, a sample from the  naive posterior, adjust $\theta^j$ according to 
	$
	\widetilde{\theta}^{j}=\bar{\theta}_n+\widehat{\Delta}_n\widehat{W}_{\sqrt{n}}\widehat{\Delta}_n^{-1/2}(\theta^j-\bar{\theta}_n),
	$ where $\widehat{W}_n$ is a consistent estimator of the asymptotic covariance of $W_\star:= \mathrm{var}\{{\sqrt{n}}M_n(\theta_\star)\}$.
	
	\medskip 
	
	The adjustment posterior obtained from Steps 1. and 2. is related to, but different from, the covariance adjustment procedure discussed in Section 4 of \cite{frazier2019bayesian}. 	In correctly specified models,  \cite{frazier2019bayesian} consider adjusting the covariance of the BSL posterior in situations where a computationally simpler, and $\theta$-dependent, misspecified covariance matrix, say ${\Psi}_n(\theta)$, is used in place of the usual covariance matrix $\widehat{\Sigma}_n(\theta)$ within the synthetic likelihood. After obtaining posterior draws for $\theta$ using the simpler covariance matrix ${{\Psi}}_n(\theta)$, the resulting draws can be adjusted using a similar approach to that given in Step 2., but where, due to the dependence on the matrix $\Psi_n(\theta)$ on $\theta$, the variance of the score equations is estimated using Gaussian processes (see Step 2 in Section 4.1 of \citealp{frazier2019bayesian}). 
	
	We advise against carrying out the original approach of \cite{frazier2019bayesian} in genuinely misspecified models. The choice of ${\Psi}_n(\theta)$ ties the adjusted posterior in \cite{frazier2019bayesian} to a pseudo-true value defined by the the misspecified covariance matrix, ${\Psi}_n(\theta)$, and since this choice is driven by computational concerns it is unlikely to produce a meaningful pseudo-true value onto which the resulting posterior will concentrate.
	
	Furthermore,  due to the presence of the $\theta$-dependent covariance matrix, in order to estimate the variance of the synthetic likelihood score equations (i.e., to estimate $W_\star$) \cite{frazier2019bayesian} must approximate the derivatives of the log synthetic likelihood using a Gaussian process, and take the variance of the Gaussian process at the posterior mean as their estimate of the variance. In contrast, since our choice of weight matrix is fixed at the identity matrix, the synthetic likelihood  score equations have a simple form that only depends on $\{\widehat{b}_n(\theta)-S_n\}$ and $\partial\widehat{b}_n(\theta)/\partial\theta^\top$, where the later term can be easily estimated using a mix of simulation from the model and finite-difference. 
	
	Consequently, unlike the approach proposed in Section 4 of \cite{frazier2019bayesian}, our approach bypasses the need to approximate derivatives of the synthetic likelihood using Gaussian processes, and, more importantly, delivers a posterior that concentrates onto a meaningful pseudo-true value.

	\section{Additional details for the MA(1) model}
	\subsection{Mean and variance}\label{sec:appmaexam}
	The researcher believes $y$ is generated according to an moving average model of order one, see equation \eqref{MA2}, and our prior beliefs are uniform over $[-1,1]$. The summary statistics are
	$\gamma_j(y)=\frac{1}{n}%
	\sum_{t=1+j}^{n}y_{t}y_{t-j}$, for $j\in\{0,1\}$, and $S_n\left( y\right) =(\gamma_0(y), \gamma_1(y))^\intercal$. In this example, the mean and variance of the summaries can be calculated exactly, with these quantities then used to construct the exact Bayesian synthetic likelihood posterior. The mean of the summaries is
	$
	b(\theta)=\mathbb{E}\{S_n(z_{1:n})\mid \theta\}=\left(1+\theta^2,\theta\right)^{\top}. 
	$ 
	
	The variance of the summaries also has a closed-form, and can be derived using the results of \cite{de1981investigation} on the variance and variance of sample autocorrelations in autoregressive integrated moving average  models.
	Partitioning $\Sigma_n(\theta)$ as $$
	\Sigma_n(\theta)=\begin{pmatrix}
		\Sigma_{11,n}(\theta)&\Sigma_{12,n}(\theta)\\\Sigma_{12,n}(\theta)&\Sigma_{22,n}(\theta)
	\end{pmatrix},
	$$ the leading terms in the components of $\Sigma_n(\theta)$ are as follows:
	\begin{flalign*}
		\Sigma_{11,n}(\theta)&=(2/n^4)\left[n^3 \cdot (1+\theta^2)^2+2 \cdot n^2 \cdot (n-1) \cdot \theta^2\right]+O(n^{-2})\\
		\Sigma_{22,n}(\theta)&=(1/n^2)\left[(n-1) \cdot ((1+\theta^2)^2+\theta^2)+2 \cdot (n-2) \cdot \theta^2\right]+O(n^{-2})\\
		\Sigma_{12,n}(\theta)&=(2/n^4)\left[n^2 \cdot ( (n-1) \cdot (2 \cdot (1+\theta^2) \cdot \theta))\right]+O(n^{-2}).
	\end{flalign*}From this representation, it is clear that each term has a dominant $O(n^{-1})$ term, and that $n\Sigma_n(\theta)$ is positive-definite for all $\theta\in[-1,1]$. 
	
	\subsection{Verification of Assumptions}\label{app:verify}
	We now verify Assumptions \ref{ass:three}-\ref{ass:norm} in the main text for the running example MA(1) model based on observed summaries $S_n(y)=(\gamma_0(y),\gamma_1(y))^\top$. 
	
	\medskip 
	
	\noindent \textbf{Assumption 1.}  The parameter space is $\Theta=[-1,1]$ and is compact. From the definitions of $b(\theta)$ and $\Sigma_n(\theta)$ in Section \ref{sec:appmaexam}, it follows directly that $\theta\mapsto\log g_n(\cdot\mid\theta)$ is twice-differentiable on $\mathrm{Int}(\Theta)$. 
	
	\medskip 
	
	\noindent \textbf{Assumption 2.} Under the true DGP in equation \eqref{MA2}, it is simple to show that $\E[S_n(y)]=b_{\bullet}=(b_{\bullet,0},b_{0,1})^\top$. To obtain the variance of $S_n(y)$, we must compute $\mathrm{var}\left(T^{-1}\sum_{t=1}^{T}y_t^2\right)$, $\mathrm{var}\left(T^{-1}\sum_{t=2}^{T}y_ty_{t-1}\right)$, and $\text{cov}\left(T^{-1}\sum_{t=1}^{T}y_t^2,T^{-1}\sum_{t=2}^{T}y_ty_{t-1}\right)$, and demonstrate that the above terms are finite. If this is the case, we can then apply Corollary 3.1 of \cite{genon2000stochastic}, with $f(z_i,\cdots,z_{d+i})=(z_{i+1}^2,z_{i+1}z_{i})^\top$, to conclude that 
	$
	\sqrt{n}\{S_n(y)-b_{\bullet}\}\Rightarrow N(0,V).
	$
	
	To obtain the result in question, we make use of the following moments given in \cite{andersen1996gmm} for the stochastic volatility model defined in equation \eqref{trueDGP} of the main text:  for $\sigma_t:=\exp\{h_t/2\}$, $j=1,2,3,\cdots,t-1$, and $r,s$ positive integers, $\mu=\omega/(1-\rho)$, and $\sigma^2=\sigma_v^2/(1-\rho^2)$, we have that, 
	\begin{flalign*}
		\E[y_t^4]&=3\E(\sigma_t^4),\\\E[y_t^2y_{t-j}^2]&=\E(\sigma^2_t\sigma^2_{t-j}),\\ \E(\sigma^r_t\sigma^s_{t-j})&=\E(\sigma_t^r)\E(\sigma_t^s)\exp\{rs\rho^j\sigma^2/4\},\\	\E(\sigma_t^r)&=\exp\{r\mu/2+r^2\sigma^2/8\}.
	\end{flalign*}
	First consider  $\mathrm{var}\left(T^{-1}\sum_{t=1}^{T}y_t^2\right)$, and note that
	$$
	\E\left[\left(T^{-1}\sum_{t=1}^{T}y_t^2\right)^2\right]=\frac{1}{T^2}\sum_{t=1}^{T}\E[y_t^4]+\frac{1}{T^2}\sum_{t\ne k}^{T}\E[y_{t}^4y_{k}^4].
	$$ Rewriting $y_t=\epsilon_t\sigma_t$, and noting that $\epsilon_t\stackrel{iid}{\sim} N(0,1)$, we see that 
	\begin{flalign*}
		\E[y_{t}^4y_{k}^4]&=\E(\epsilon_k^4\epsilon_t^4\sigma_t^4\sigma^4_k)=\E(\epsilon_k^4)\E(\epsilon_t^4)\E(\sigma^4_t\sigma^4_k)=9\E(\sigma^4_t\sigma^4_k)	.
	\end{flalign*}
	Since the above expectations are positive and finite for all $t,k$, we have that $
	\E\left[\left(T^{-1}\sum_{t=1}^{T}y_t^2\right)^2\right]=O(1/T)
	$, and $\mathrm{var}\left(T^{-1}\sum_{t=1}^{T}y_t^2\right)=O(1/T)$.
	
	Next,  consider $\mathrm{var}\left(T^{-1}\sum_{t=2}^{T}y_ty_{t-1}\right)$, and note that  
	$$
	\E\left(T^{-1}\sum_{t=2}^{T}y_ty_{t-1}\right)^2=\frac{1}{T^2}\sum_{t=2}^{T}\E[y_t^2y_{t-1}^2]+\frac{2}{T^2}\sum_{t=2}^{T}\sum_{t<k}^{}\E[y_{t}y_{t-1}y_{k}y_{k-1}].
	$$For the first term,  $0<\E(y_t^2y_{t-1}^2)=\E(\sigma^2_t\sigma^2_{t-j})<\infty$ for all $t$,  and the first term is $O(1/T)$. For the second term, the independence between $\epsilon_t,\epsilon_k$ ($t\ne k$), and the fact that $\E(\epsilon_t)=0$, can be used to show that $\E[y_{t}y_{t-1}y_{k}y_{k-1}]=0$. Hence, $\mathrm{var}\left(T^{-1}\sum_{t=2}^{T}y_ty_{t-1}\right)=O(1/T)$. 
	
	To deal with the covariance term, we must calculate terms of the form $\E(y_t^2 y_{t}y_{t-1})$. Since, $\epsilon_t\stackrel{iid}{\sim} N(0,1)$,  we have $\epsilon_t\perp \sigma_t$, $\epsilon_t\perp\sigma_k$, for all $t,k$, 
	$$
	\E(y_t^2 y_{t}y_{t-1})=\E(y_t^3y_{t-1})=\E[\sigma_t^3\epsilon_t^3\epsilon_{t-1}\sigma_{t-1}]=\E(\sigma_t^3\epsilon_t^3\E[\epsilon_{t-1}\mid\mathcal{F}_{t-1}]\sigma_{t-1})=0.
	$$Hence, the covariance term is $o(1/T)$. The asymptotic variance of $\sqrt{n}\{S_n(y)-b_{\bullet}\}$  is therefore positive-definite, and Corollary 3.1 of \cite{genon2000stochastic} yields the stated convergence in \textbf{Assumption 2}. 
	
	\medskip 
	
	\noindent \textbf{Assumption 3.} From the statement of $\Sigma_n(\theta)$ in \ref{sec:appmaexam}, note that 
	\begin{equation}\label{eq:matrix}
		\Sigma(\theta):=\lim_{n\rightarrow+\infty}n\Sigma_n(\theta)=
		\begin{pmatrix}
			(1+\theta^2)^2+2\theta^2 & 2(1+\theta^2)\theta\\2(1+\theta^2)\theta&(1+\theta^2)^2+3\theta^2
		\end{pmatrix}.
	\end{equation}
	From the structure of $\log g_n(S_n\mid\theta)$, we see that for $Q(\theta):=\plim_{n\rightarrow\infty}n^{-1}\log g_n(S_n\mid\theta)$ is given by $Q(\theta):=(-1/2)\{b(\theta)-b_{\bullet}\}^\top \Sigma(\theta)^{-1}\{b(\theta)-b_{\bullet}\}$. While $Q(\theta)$ has an analytical form, it is complicated and lengthy to state, but is available in the code accompanying the paper. 
	
	We first verify that the set $\Theta_\star$ is non-empty and finite. This can be done numerically by plotting the function $Q(\theta)$ over $\Theta$ across each value of  $b_{{\bullet},0}\in\{0.01,0.10,0.25,0.50,0.75,0.95\}$ used in the experiments. The results are presented in Figure \ref{fig:crit}, and demonstrate that, when $b_{{\bullet},0}\in\{0.75,0.99\}$, the set $\Theta_\star$ contains a single element, and for $b_{{\bullet},0}\in\{0.01,0.10,0.25\}$, the set $\Theta_\star$ contains two elements. In the case where $b_{{\bullet},0}=0.50$, it is unclear whether $\Theta_\star$ contains a continuum of elements or simply two elements that are close together, see Section \ref{sec:flat} for further discussion. As discussed in the main text, our results are not applicable to the case where $\Theta_\star$ contains a dense set of points. 
	\begin{figure}[H]
		\centering{\includegraphics[width=175mm, height=110mm]{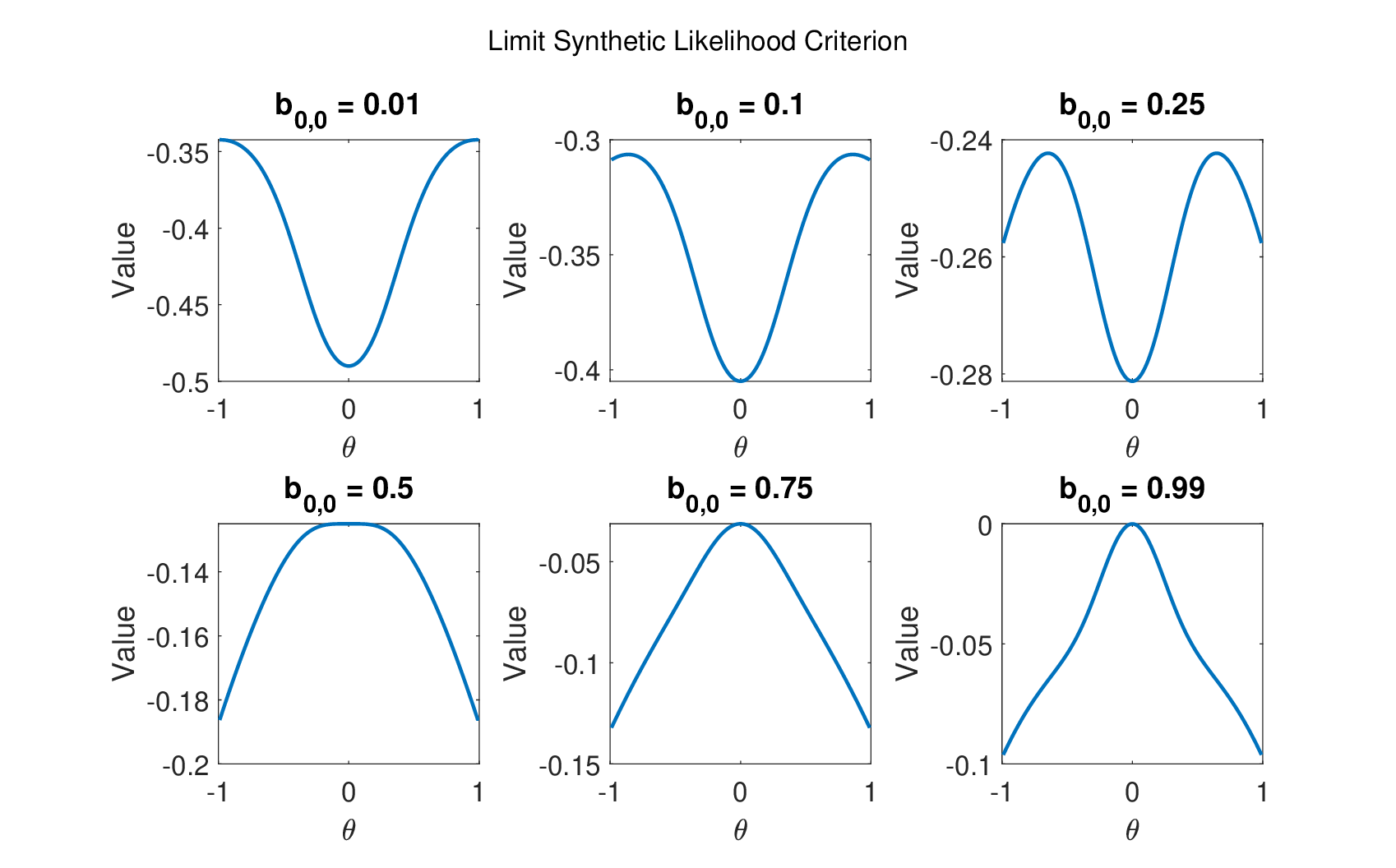}}
		\caption{Behavior of the limit synthetic likelihood criterion as the value of $b_{{\bullet},0}$ increases. }
		\label{fig:crit}
	\end{figure}
	
	To verify the second part of the assumption, we note that the limit Hessian is given by $H(\theta)=\partial^2 Q(\theta)/\partial\theta^2$. While it is feasible to state $H(\theta)$ analytically, it is complicated and does not have a simple structure. Instead, we verify this assumption numerically by showing that $-H(\theta)$ is positive at each value in $\Theta_\star$, and for each $b_{{\bullet},0}\in\{0.01,0.1,0.25,0.75,0.99\}$ used in the experiments. Analysing the result in Figure \ref{fig:hess} we see that $0<-H(\theta_\star)<\infty$ for each $\theta_\star\in\Theta_\star$, and across each situation where $b_{{\bullet},0}\in\{0.01,0.10,0.25,0.75,0.99\}$. 
	
	\begin{figure}[H]
		\centering{\includegraphics[width=175mm, height=110mm]{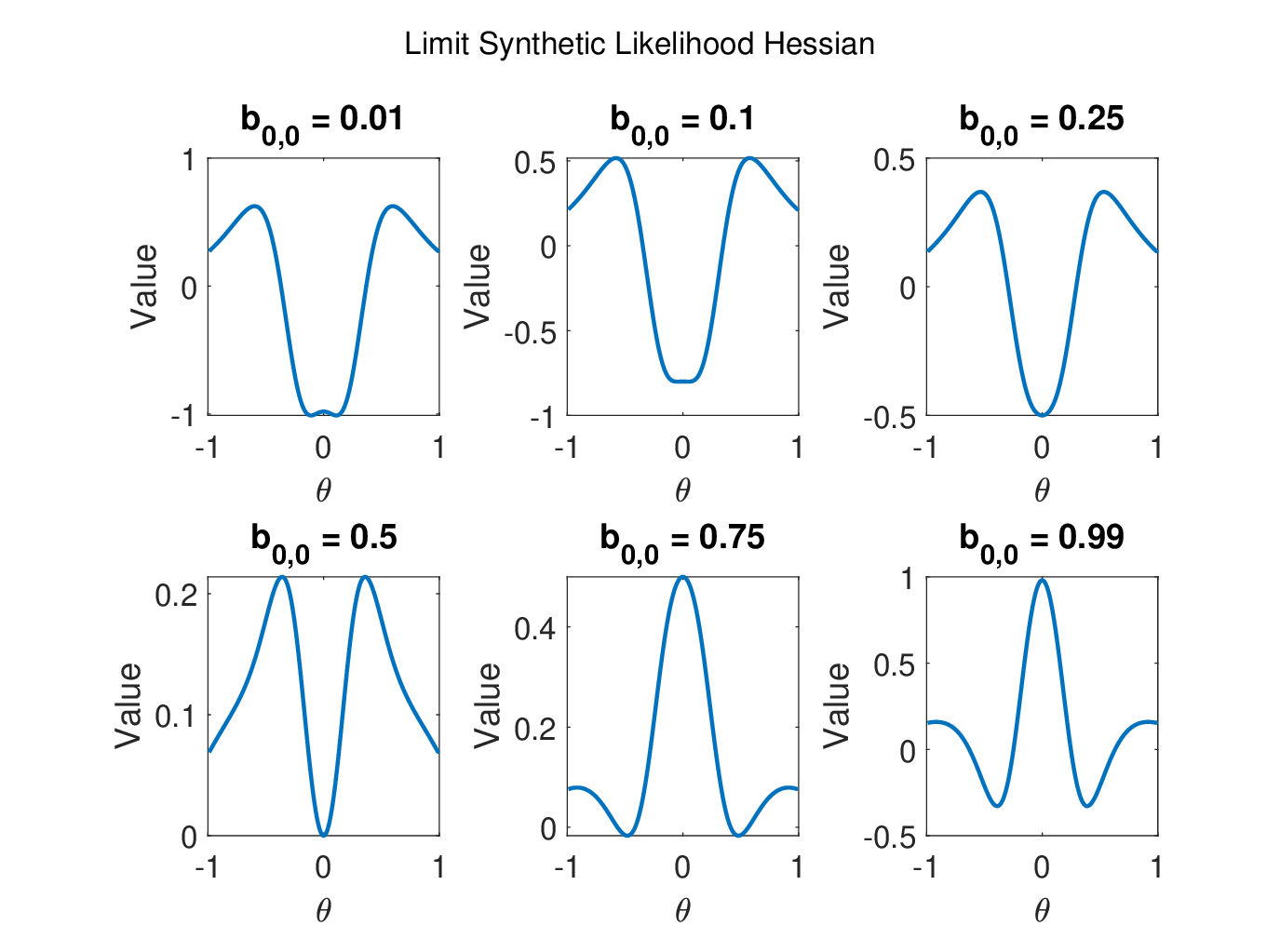}}
		\caption{Behavior of the limit (negative) synthetic likelihood Hessian as the value $b_{{\bullet},0}$ changes. }
		\label{fig:hess}
	\end{figure}

	In the case where there appears to be a flat region in the criterion, $b_{{\bullet},0}=0.50$, the Hessian is nearly-singular in a neighbourhood of zero, indicating that Assumption 3 will not be satisfied in this case. However, as stated in the main text, our results are not applicable in the case where the synthetic likelihood criterion admits a continuum of modes.

	\medskip
	
	\noindent \textbf{Assumption 4.} We remind the reader that Figure \ref{fig:hess} plots the Hessian as a function of $\theta$. Hence, Assumption 4 can be verified for each value of $b_{{\bullet},0}\in\{0.01,0.10,0.25,0.75,0.99\}$ used in the experiments by inspecting the corresponding plot in Figure \ref{fig:hess}, and noting that the resulting curve is continuously differentiable. 
	
	\begin{remark}
		We note that verification of Assumptions 1, 3, and 4 remains entirely feasible even if the mean and variance of the summaries are not available in closed form. In such cases, rather than constructing $Q(\theta)$ analytically, or numerically, as was done above, we can approximate this term via Monte Carlo by simulating datasets of large size from the assumed model over a grid of $\Theta$. That is, we can approximate $b(\theta)$ and $\Sigma(\theta)$ directly by drawing $z_1,\cdots,z_n\mid\theta$, with $n$ very large, and then estimating $b(\theta)$ and $\Sigma(\theta)$ using their corresponding sample averages. Given these estimators, the resulting numerical analysis can then proceed directly as described above, and with the unknown $b_{\bullet}$ replaced by the observed summary $S_n(y)$.   
	\end{remark}

	\medskip 
	
	\noindent \textbf{Assumption 5.} From the definition of $\Sigma_n(\theta)$ in Section \ref{sec:appmaexam}, it is direct that, for $n$ large enough, $n\Sigma_n(\theta)$ is positive-definite uniformly across $\Theta$. Hence, Assumption 5(i) is satisfied. Assumption 5(ii) follows from inspecting the definition of $\Sigma(\theta)$ given in equation \eqref{eq:matrix}. Assumption 5(iii) follows by noting that, since $n\Sigma_n(\theta)$ is continuous, and $\Theta$ compact, pointwise convergence of $n\Sigma_n(\theta)$ to $\Sigma(\theta)$ holds uniformly over $\Theta$. 
	
	\medskip 
	
	\noindent \textbf{Assumption 6.} We have that $\pi(\theta)=1/2$ for all $\theta\in[-1,1]$, which is positive and continuous for any $\theta\in\Theta$. 
	
	\medskip 
	
	\noindent \textbf{Assumption 7.} While it may be possible to verify this condition analytically, it is simpler to very this condition numerically. To this end, we generate a large number of independent replicated datasets, $M=1000000$, from the MA(1) model with sample size $n=100,500,1000$, and across an evenly-spaced grid of values for $\theta\in[-1,1]$. For each value of $\theta$, we simulate $z=(z_1,\dots,z_n)$ from the DGP and calculate $\sum_{i=1}^{M}|\gamma_0(z^i)|^4/M$ and $\sum_{i=1}^{M}|\gamma_1(z^i)|^4/M$. The resulting Monte Carlo means are plotted over $\Theta$ in Figure \ref{fig:ass7}. The figure shows that uniformly over $\Theta$ the above expectations are bounded, which verifies the stated assumption.

	\begin{figure}[H]
		\centering{\includegraphics[width=175mm, height=75mm]{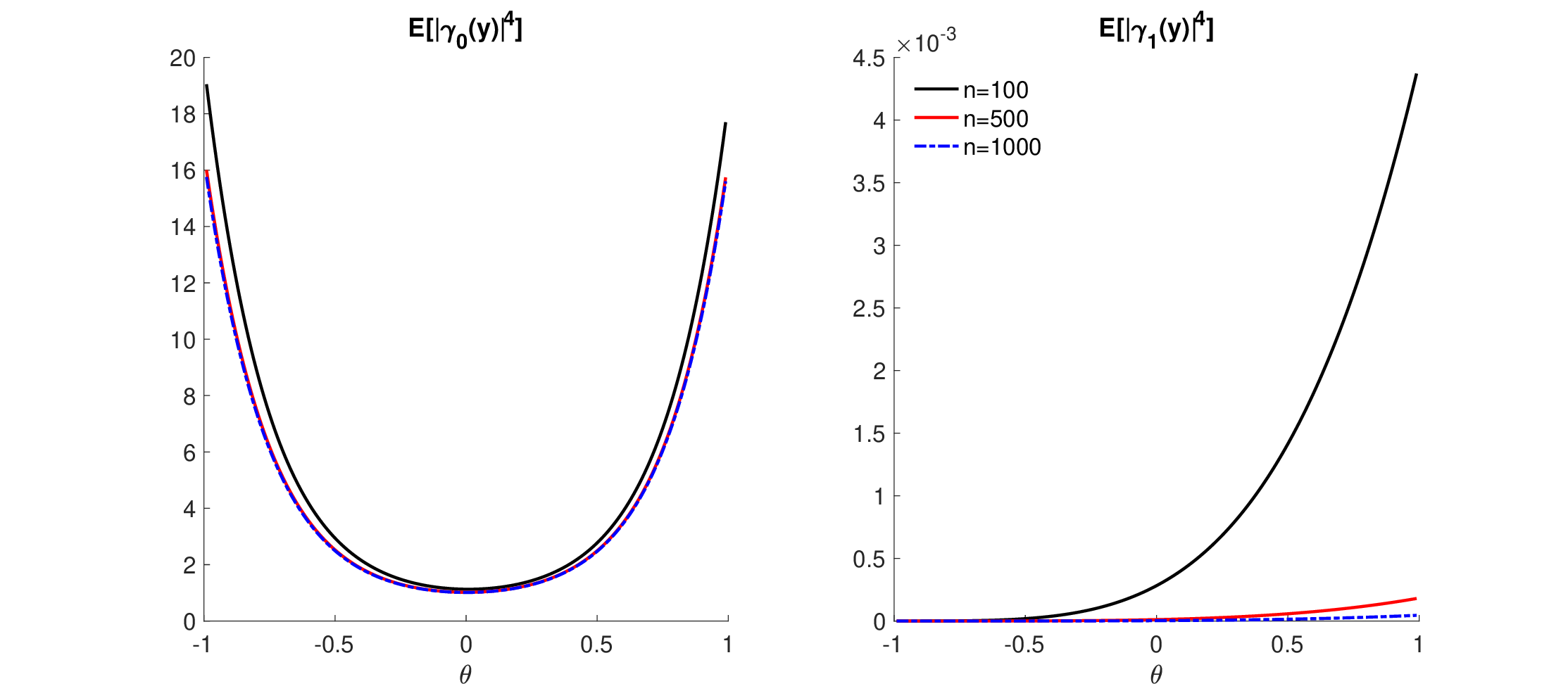}}
		\caption{Monte Carlo estimates of the fourth-moments for the summaries used in the running example.}	\label{fig:ass7}
	\end{figure}
	
	\medskip 
	
	\noindent \textbf{Assumption 9.} We  demonstrate numerically that the score equations are asymptotically Gaussian in the case where $b_{{\bullet},0}\approx0.90$. We remark that it is feasible to carry out such a procedure in this case as we know with certainty that the pseudo-true value is $\theta_\star=0$, and that without this information it would not be feasible to verify such an assumption. We verify the assumption by generating $10000$ replications from the true DGP in equation \eqref{trueDGP} in the main text across sample sizes $n=500,1000,5000$. For each of these datasets we calculate $\sqrt{n}M_n(\theta_\star)$. For each sample size, in Figure \ref{fig:ass9} we then plot the resulting kernel density  of $\sqrt{n}M_n(\theta_\star)$ across the replicated realizations, where we standardize the results to aid visual interpretability. For $n=5000,$ and $n=10000$, the resulting densities are close to the standard Gaussian density. Hence, the results in Figure \ref{fig:ass9} demonstrate that the synthetic likelihood score equations are indeed converging towards a Gaussian distribution. 
	
	\begin{figure}[H]
		\centering{\includegraphics[width=150mm, height=100mm]{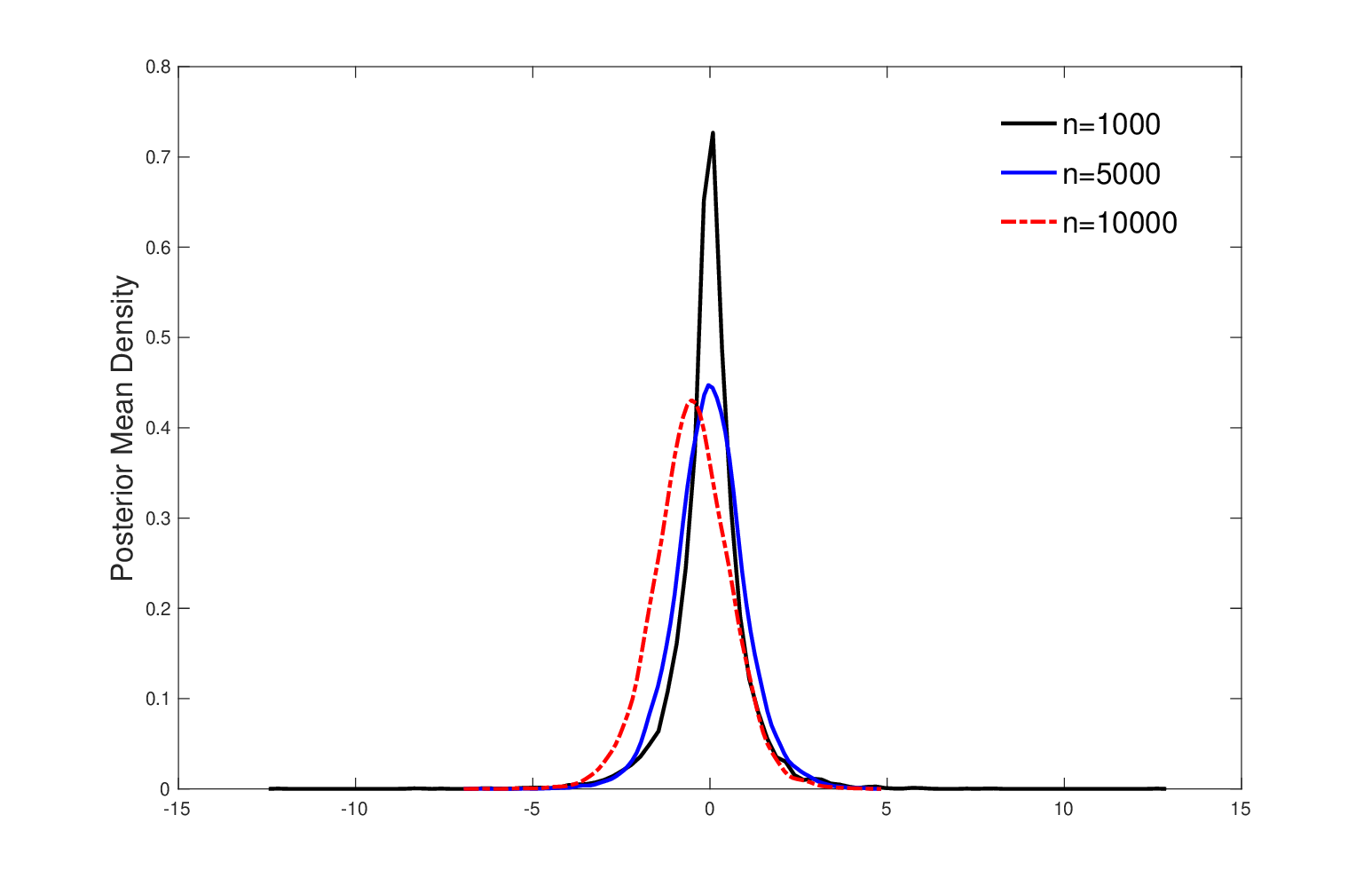}}
		\caption{Behavior of the (normalized) synthetic likelihood score equation for $b_{{\bullet},0}\approx0.90$. }
		\label{fig:ass9}
	\end{figure}
	
	\begin{remark}
		As the above discussion indicates, with the exception of Assumptions 2 and 9 in the main text, it is possible to verify each of the stated assumptions used to derive our results using simulation from the assumed model; Assumptions 2 and 9 can only be verified in cases where the true DGP and pseudo-true value are known, and cannot be verified otherwise. Critically,  little of the techniques described in this section are particular to the MA(1) model, and verifying these conditions for other models can proceed using the same procedures discussed in this section for any model where it is cheap to simulate artificial datasets. 
	\end{remark}

	\subsection{Discussion on Posterior flatness}\label{sec:flat}
	In this section, we explore the mechanism behind the posterior flatness observed in Figure \ref{fig2} in the main text when $b_{{\bullet},0}=0.50$. In particular, it is uncertain if the posterior is flat due to the existence of a genuine region across which the synthetic likelihood function is flat, or if the posterior flatness is the result of two modes that are close to one another. If the latter was indeed the case, then for reasonable sample sizes, the posterior would appear flat even though, technically, the synthetic likelihood would exhibit some curvature between the two modes. 
	
	To determine if the flatness observed in Figure \ref{fig2} is genuine,  we explore the behavior of the limiting synthetic likelihood criterion as we change the value of $b_{{\bullet},0}$ (see Section \ref{sec:maexam} for details). We plot the limit synthetic likelihood as we vary $b_{{\bullet},0}$ in a small neighbourhood between $0.490$ and $0.501$, and where the criterion is calculated over a fine  grid of values for $\theta$. The result of this exercise is presented in Figure \ref{fig:ma1_flat}.

	The behavior in Figure \ref{fig:ma1_flat} constitutes convincing evidence that the posterior flatness observed when $b_{{\bullet},0}=0.50$ may be caused by the ``merging'' of two modes. That is, this figure indicates that the flatness is created by two modes that are so close as to be indistinguishable at reasonable resolutions. At a value of $b_{{\bullet},0}=0.490$ the synthetic likelihood exhibits two well-separated modes. However, as the value of $b_{{\bullet},0}$ approaches $b_{{\bullet},0}=0.50$ the two modes become closer and closer. At $b_{{\bullet},0}=0.499$, it is possible to discern a small gap between the two modes, but at $b_{{\bullet},0}=0.500$ this gap is no longer present. 
	
	From a practical standpoint, if the criterion has two close, but well-separated modes, the posterior will behave as if the criterion function exhibited a genuine region of flatness between the two modes. That is, even if asymptotically the modes are well-separated, so long as the two are close enough then for any reasonable sample size we will observe a region where the posterior appears ``flat''. We thank an anonymous referee for asking us to reconsider this behavior. 
	
	\begin{figure}[H]
		\centering{\includegraphics[width=175mm, height=110mm]{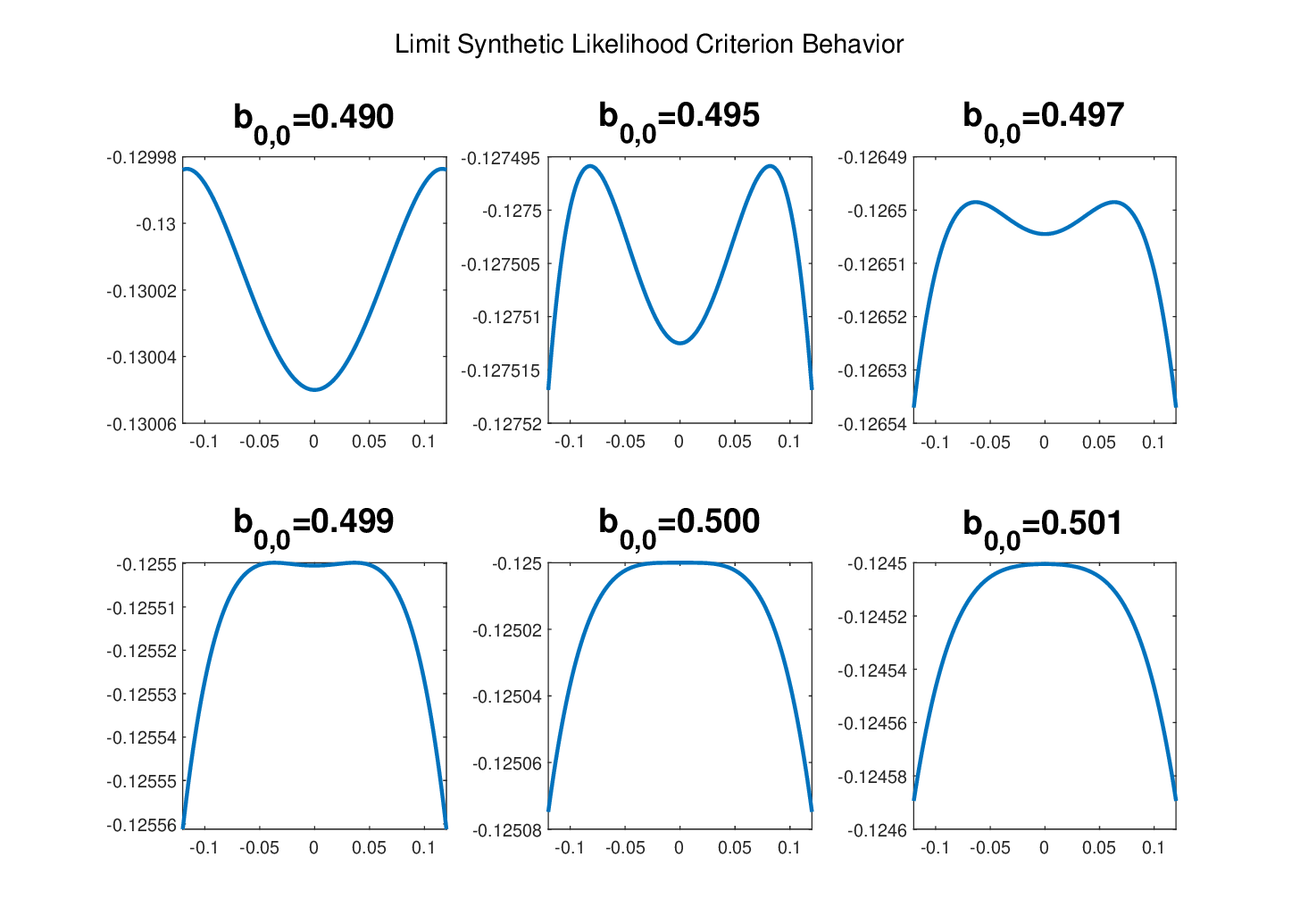}}
		\caption{Behavior of the limit synthetic likelihood criterion as the value of $b_{{\bullet},0}$ changes. }
		\label{fig:ma1_flat}
	\end{figure}

	\subsection{Asymptotic Gaussianity: Posterior Means}
	In this section, we give Monte Carlo evidence that demonstrates the conclusion of Corollary \ref{corr:bslm} in the main paper; asymptotic normality of the BSL posterior mean in the single mode case. In this experiment, we generate data from the true DGP in equation \eqref{trueDGP} for sample sizes of $n=100,500,1000$. To demonstrate these results, we restrict our analysis to the case where the limit synthetic likelihood admits an asymptotically unique  mode. To this end, we consider the true DGP in equation \eqref{trueDGP} with $(\omega,\rho,\sigma_v)=(-0.05,0.90,0.40)^\top$, which yields a value of $b_{{\bullet},0}\approx0 .90$. 
	
	For this true DGP, and the sample sizes stated earlier, we generate 1000 replications  and apply BSL with the usual covariance matrix $\widehat{\Sigma}_n(\theta)$. We use RWMH-MCMC sampling to obtain posterior draws, and we fix the number of simulated datasets generated at each iteration to be $m=\lceil n^{0.60}\rceil$. 
	
	We represent the results visually across the three sample sizes in Figure \ref{fig:meancomp}. As the sample size increases, we see that the posterior is concentrating around the point $\theta=0$, and that the concentration is roughly Gaussian. While there appear to be some outliers at small sample sizes, these dissipate as the same size increases. 
	\begin{figure}[H]
		\centerline{\includegraphics[width=130mm, height=70mm]{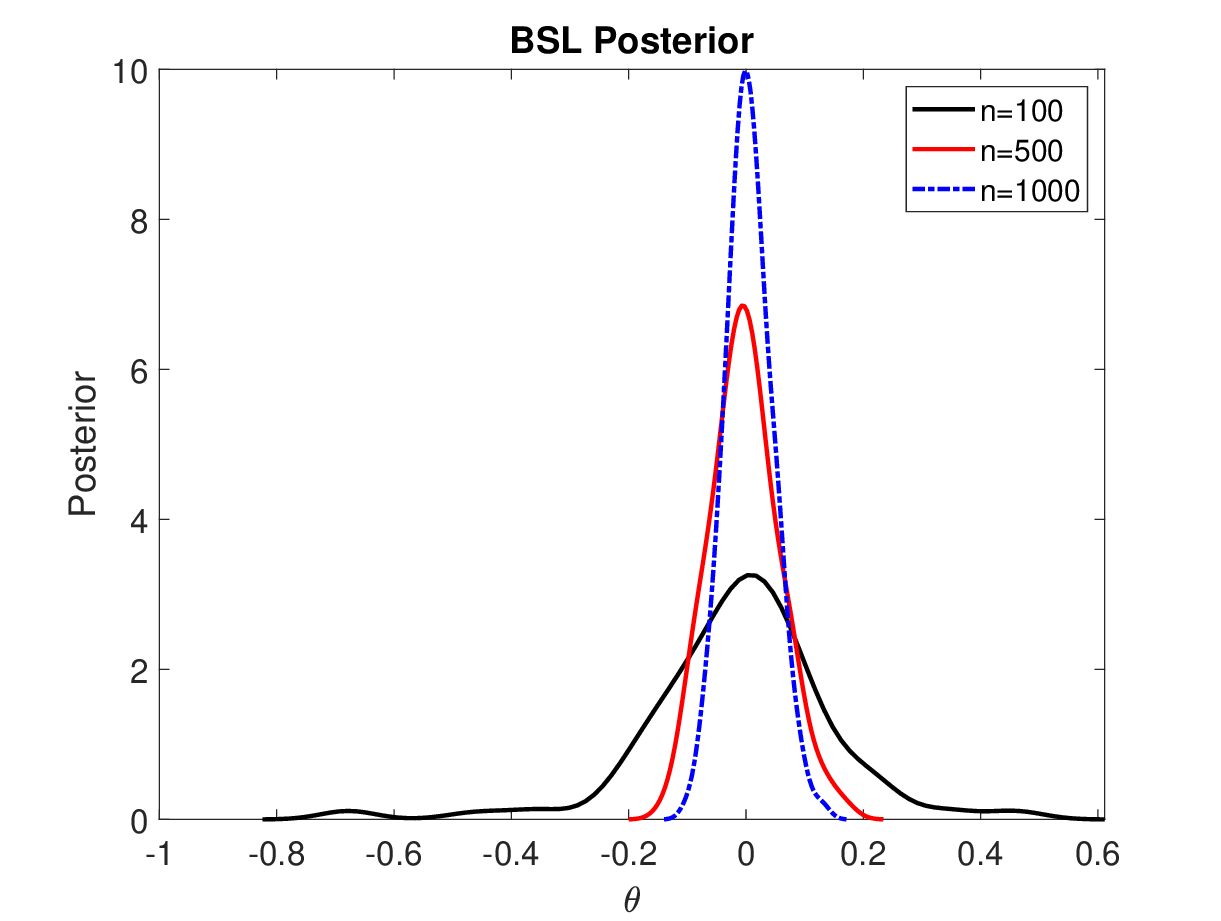}}
		\caption{Kernel density plots of the standard BSL posterior mean across 1000 replications from the DGP in equation \eqref{trueDGP} in the main text.}
		\label{fig:meancomp}
	\end{figure}

	\subsection{Multi-modality in correctly specified models}\label{sec:ma2}
	In this section, we present an illustrative example which demonstrates that the conclusions in Theorem 3.1 remain applicable in correctly specified models under identification failure. 
	
	The observed data $y=(y_1,\dots,y_n)^\intercal$ is generated according to a moving average model of order two
	\begin{equation}
		y_{t}=e_{t}+\theta_{1}e_{t-1}+\theta _{2}e_{t-2}\quad  (t=1,\dots,n), \label{MA_new}
	\end{equation}with $e_t$ independent and identically distributed $N(0,1)$, and where $(\theta_1,\theta_2)'\in(-1,1)^2$ is unknown with uniform prior beliefs over this region. We take as summary statistics the
	first two sample autocovariances $\gamma _{j}(y)=\frac{1}{n}%
	\sum_{t=1+j}^{n}y_{t}y_{t-j}$, for $j\in\{0,1\}$, and let $S_n\left( y\right) =(\gamma_0(y), \gamma_1(y))^\intercal$. 
	
	Let $\theta_{\bullet}$ denote the value of $\theta$ that has generated the data. Under this model, the simulated summaries converge in probability to
	$$
	b(\theta)=[(1+\theta_1^2+\theta_2^2),\theta_1(1+\theta_2)]^\intercal
	$$and $b_{\bullet}=b(\theta_{\bullet})$ by definition. 
	
	Now, consider a value of $\theta_{\bullet}=(.8,-.2)^\intercal$. Solving
	$$
	0=b(\theta)-b(\theta_{\bullet})=\begin{pmatrix}
		1+\theta_1^2+\theta_2^2\\\theta_1(1+\theta_2)
	\end{pmatrix}-\begin{pmatrix}
		1.68\\.64
	\end{pmatrix}
	$$
	delivers two solutions: $$\theta_{1,\star}=(.8,-.2)^\intercal\text{ and }\theta_{2,\star}=(.3683,.7378)^\intercal.$$Consequently, the set $\Theta_\star$ defined in Assumption 3.3 contains these two points.

	We now demonstrate empirically that the Bayesian synthetic likelihood posterior is bimodal with modes at $\theta_{1,\star}$ and $\theta_{2,\star}$ using a similar approach to the running example in the main text. Namely, we consider sample sizes of $n=100, 1000, 10000$ and conduct inference on $\theta_{\bullet}$ using the above summary statistics.  Often the Bayesian synthetic likelihood posterior is sampled with Markov chain Monte Carlo, however, to better capture the posterior multimodality in this example we use sequential Monte Carlo.  Specifically, we use the independent sequential Monte Carlo algorithm of \citet{South2019}, which uses an independent proposal distribution in the Markov chain Monte Carlo mutation part of the algorithm.  The independent proposal is formulated as a mixture model, whose hyperparameters are calibrated using the population of particles available in sequential Monte Carlo.  Here we simply use a two-component Gaussian mixture as the proposal.  sequential Monte Carlo requires a sequence of distributions to sample, and we use (synthetic) likelihood tempering to connect the prior and the Bayesian synthetic likelihood posterior.  We use 5000 particles in sequential Monte Carlo and $m=20$ model simulations for estimating each synthetic likelihood. 
	
	
	We plot the results graphically in Figure \ref{fig:ma2}. The results demonstrate the existence of two persistent modes around $\theta_{1,\star}$ and $\theta_{2,\star}$, with the mass around these models increasing as the sample size increases, resulting in modes that are further and further apart. 
	
	\begin{figure}[H]
		\centerline{\includegraphics[scale=0.7]{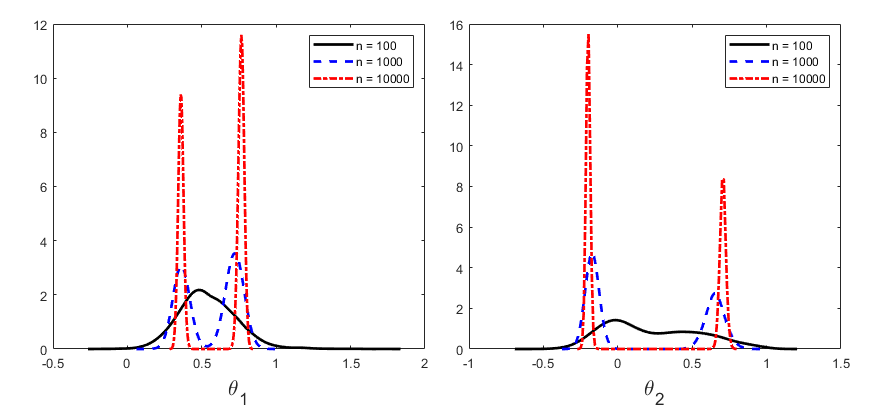}}
		\caption{Posteriors for the moving average model of order two example with identification failure. Recall $\theta_{1,\star}=(.8,-.2)^\intercal$ and  $\theta_{2,\star}=(.3683,.7378)^\intercal.$}
		\label{fig:ma2}
	\end{figure}

	\section{Adjustment Example: $g$-and-$k$ model}\label{ex:gandk}
	
	The $g$-and-$k$ family of univariate distributions \citep{haynes+mm97} 
	has four parameters $\theta=(A,B,g,k)$, where the parameters
	control location, scale, skewness and kurtosis respectively.  A $g$-and-$k$ distribution is defined
	through a closed-form quantile function
	$$Q(p;A,B,g,k)=A+B\{1+c \tanh [gz(p)/2]\}z(p)[1+z(p)^2]^k,\;\;\;p\in (0,1),$$
	where $z(p)$ is the quantile function of the standard normal distribution, 
	and the constant $c$ is conventionally fixed at $0.8$, which results in the constraint $k>-0.5$.  
	Bayesian inference for the $g$-and-$k$ model was considered in \cite{allingham+km09}, where it was 
	noted that the closed form quantile function allows easy simulation from the model using the inversion method,  
	making likelihood-free inference methods attractive.
	
	In this example a dataset modelled using the $g$-and-$k$ distribution by 
	\cite{prangle20} is considered.  The data are available 
	in the R package \texttt{Ecdat} \citep{croissant+g20} and following \cite{prangle20} we consider
	the daily log returns for exchange rates of the US versus Canadian dollar.  There are 1867 observations over
	the period 1980 to 1987.
	We fix $k=0$ in the $g$-and-$k$ model to induce clear misspecification, with the lack of kurtosis resulting in 
	the inability to capture the heavy-tailed behaviour of the real returns data.  
	
	Our initial focus is to explore a suggestion made by \cite{muller13} in the context of sandwich-type 
	variance adjustments for Bayesian inference under misspecification.  \cite{muller13} considers adjustments 
	in which, under correct model specification, a sandwich-type variance estimate and the 
	asymptotic posterior variance estimate should be approximately equal.  Section 4.4 of M\"{u}ller suggests using 
	some summary of the difference
	between estimates as a diagnostic for misspecification.  We do something similar.
	Consider the adjusted Bayesian synthetic likelihood method of Section 4.2.2, but fixing the summary statistic variance matrix to be the value for the estimated posterior mean for $\theta$ under standard Bayesian synthetic likelihood.  
	Under correct specification, and from the results of Corollary 3.1, using this fixed variance matrix
	estimate, the adjusted Bayesian synthetic likelihood should result
	in the same inferences asymptotically as the ordinary Bayesian synthetic likelihood.  So a large adjustment 
	could be considered evidence for misspecification.  
	
	Let $Q_1, Q_2, Q_3$ be the three quartiles of the data $y=(y_1,\dots, y_n)$.  We define summary statistics
	$S_1=Q_2$, $S_2=(Q_3-Q_1)$ and $S_3=(Q_3-2Q_2+Q_1)/S_2$.  These are three of four robust summary
	statistics considered in \cite{drovandi2011likelihood} for the $g$-and-$k$ model.  We also define a fourth summary statistic
	$S_4$ as the lower $1$ percentage quantile of the data, which captures the extreme negative returns.  
	Since we fix $k=0$, we have a three parameter family
	of distributions which has flexible behaviour in terms of 
	location, scale and skewness.  We use priors which are independent and uniform
	over ranges $[-1,1]$, $[0,1]$ and $[-5,5]$ for $A$, $B$ and $g$ respectively.  Using the summary statistics $S^{(1)}=(S_1,S_2,S_3)$, 
	there is no incompatibility, but with summary statistics $S^{(2)}=(S_1,S_2,S_3,S_4)$ there is.  The reason is that
	capturing the location, scale and skewness evident in the first three summary statistics while simultaneously
	matching the lower tail behaviour specified through $S_4$ is not possible when $k=0$. 
	
	Figure \ref{canadian-3ss} shows the Bayesian synthetic likelihood posterior estimates based on summary statistics $S^{(1)}$, together with the adjusted Bayesian synthetic likelihood
	posterior estimates.  The top row shows univariate posterior marginals, and the bottom row shows bivariate posterior
	marginals.  The values denoted KLDN above the plots of the univariate marginals are the Kullback-Leibler divergence between normal
	approximations to the unadjusted and adjusted posterior densities, where the normal approximations
	are based on posterior means and standard deviations for each method.  Precisely, 
	$$\text{KLDN}=\log \frac{\sigma_A}{\sigma_S}+\frac{\sigma_S^2+(\mu_S-\mu_A)^2}{2\sigma_A^2}-\frac{1}{2},$$
	where $\mu_A,\sigma_A$ and $\mu_S,\sigma_S$ are the estimated mean and standard deviation for the adjusted Bayesian synthetic likelihood and
	standard Bayesian synthetic likelihood respectively. The KLDN allows us to measure the overall change of
	the posterior marginals after adjustment.  
	
	The Bayesian synthetic likelihood and adjusted Bayesian synthetic likelihood estimates
	are based on 80,000 iterations of a random walk Metropolis algorithm with 10,000 burn-in and $m=60$ simulations per 
	likelihood estimate, with 1,000 samples retained after thinning.  Even though incompatibility is not an issue
	for the summary statistics $S^{(1)}$, there is a substantial adjustment to the posterior marginal distributions.
	\begin{figure}[htbp]
		\centerline{\includegraphics[width=120mm]{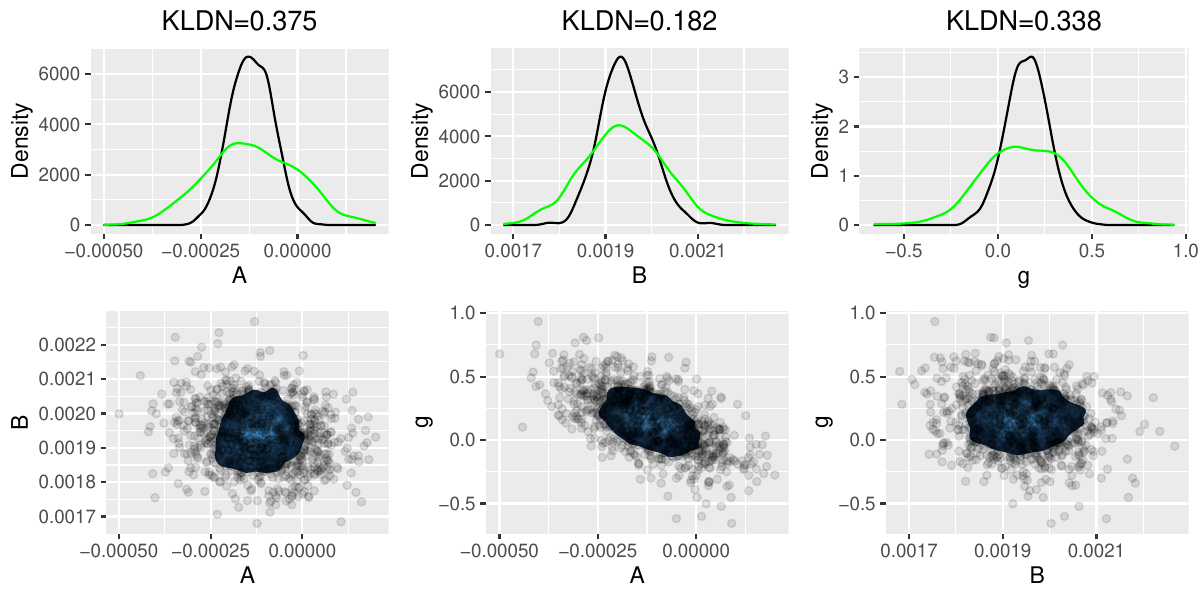}}
		\caption{\label{canadian-3ss} Estimated Bayesian synthetic likelihood posterior densities for the summary statistic vector $S^{(1)}$ using standard Bayesian synthetic likelihood and adjusted Bayesian synthetic likelihood for the US-Canadian exchange rate data and the $g$-and-$k$ model with $k=0$. The top row shows univariate marginals (black=standard Bayesian synthetic likelihood, green=adjusted Bayesian synthetic likelihood).    The KLDN values are described in the text and summarize how much the posterior changes after adjustment.  The bottom row shows bivariate marginals.  The contours show the standard Bayesian synthetic likelihood and the points are adjusted Bayesian synthetic likelihood sample values based on $1,000$ samples.}
	\end{figure}  
	This seems to be due to the misspecification of the synthetic likelihood variance.

	{To confirm this finding, we estimate the variance of the summaries for the observed data using the bootstrap and compare this estimate against the bootstrap estimate of the variance that results from $1,000$ posterior predictive replicates of the data based on the standard Bayesian synthetic likelihood posterior. The bootstrap estimates of variance for the summary statistics for the observed data are much smaller than the bootstrap variance estimates for the posterior predictive replicates, particularly for $S_3$. This indicates that the variance used in the Bayesian synthetic likelihood posterior can not reasonably accommodate the actual variance of the summaries. }

	{Figure \ref{canadian-4ss} shows the Bayesian synthetic likelihood posterior estimates based on the summary statistics $S^{(2)}$, which
		is the case of incompatibility where the summary statistic $S_4$ capturing the tail behaviour is added to $S^{(1)}$.  
		There is a large change in the estimated posterior for $g$, and larger adjustments
		are being made for both $A$ and $g$ compared to the previous case, which can be seen both graphically and from the KLDN 
		values.  The large change again suggests possible misspecification.  
	\begin{figure}[htbp]
		\centerline{\includegraphics[width=120mm]{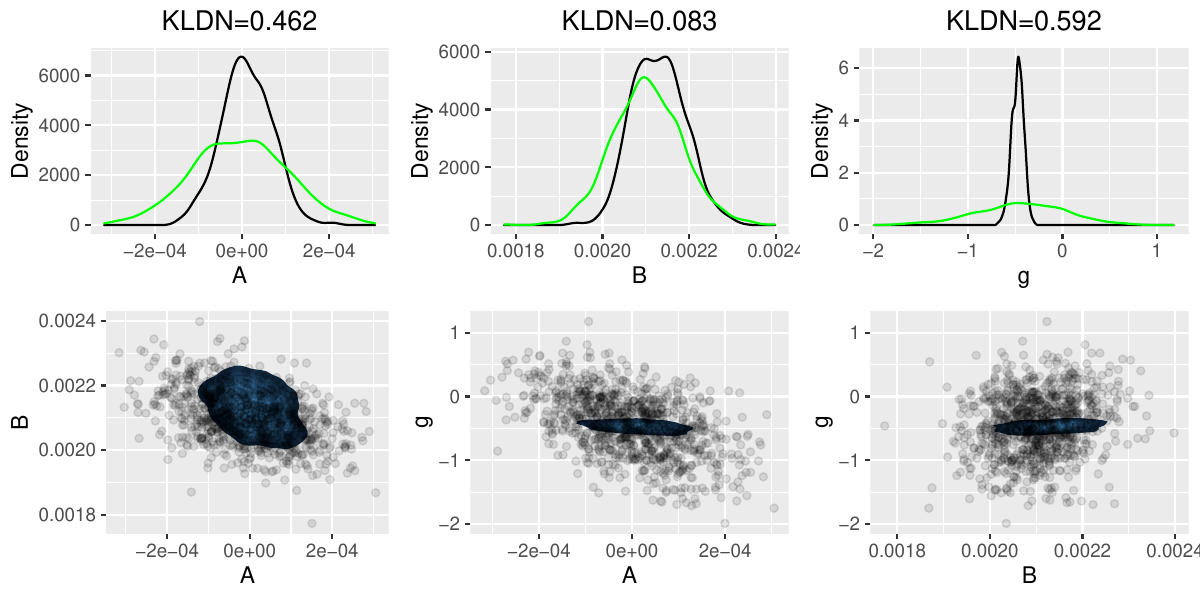}}
		\caption{\label{canadian-4ss} Estimated Bayesian synthetic likelihood and adjusted Bayesian synthetic likelihood posterior densities for the summary statistic vector $S^{(2)}$. Please see Figure \ref{canadian-3ss} for further details. }
	\end{figure}

	{The above analysis demonstrates that if the model is not correct, large adjustments can 
		arise from either summary statistic incompatibility (i.e., misspecification), or from misspecification of the 
		summary statistic variance under the assumed model. 
		However, as is true for standard Bayesian synthetic likelihood, even if the model is correctly specified, the adjustment to the synthetic likelihood
		could also be large if the summary statistics are non-Gaussian. In addition, since the adjustment is based on asymptotic arguments, large differences could also be observed if the adjustments are unreliable in finite samples.  We conclude that while the existence of a large adjustment
		is suggestive of misspecification, other diagnostics may be needed to diagnose specific features of the model that may be misspecified.}
	
	{To this end, we demonstrate that  robust Bayesian synthetic likelihood can be used as a diagnostic to pinpoint which features of the observed data cannot be matched by the assumed model.} For the prior on $\Gamma$, we use an exponential distribution with a mean of 0.5 on each component and assume the components are independent, as suggested by \cite{frazier2019robust}.  We apply standard Bayesian synthetic likelihood with $S^{(1)}$ and $S^{(2)}$, and  robust Bayesian synthetic likelihood with $S^{(2)}$.  As above, we use the Bayesian synthetic likelihood R package of \cite{an2019bsl} for running the Bayesian synthetic likelihood methods.  We use 100,000 iterations of Markov chain Monte Carlo for each run of Bayesian synthetic likelihood, and use a starting value with good support under each approximate posterior to avoid the need for a burn-in.
	
	Figure \ref{figsub:summary-predictive-standard-3ss} shows the posterior predictive distribution for each component of $S^{(1)}$ when fitted with standard Bayesian synthetic likelihood. As suggested earlier, the model is compatible with these three statistics.  The corresponding plot for standard Bayesian synthetic likelihood with $S^{(2)}$ is shown in Figure \ref{figsub:summary-predictive-standard-4ss}. It is evident that the model is not able to recover the four statistics.  By trying to match the four statistics simultaneously, the model is unable to recover any of the statistics with high accuracy, particularly $S_3$.

	\begin{figure}[htbp]
		\centering
		\includegraphics[scale=0.7]{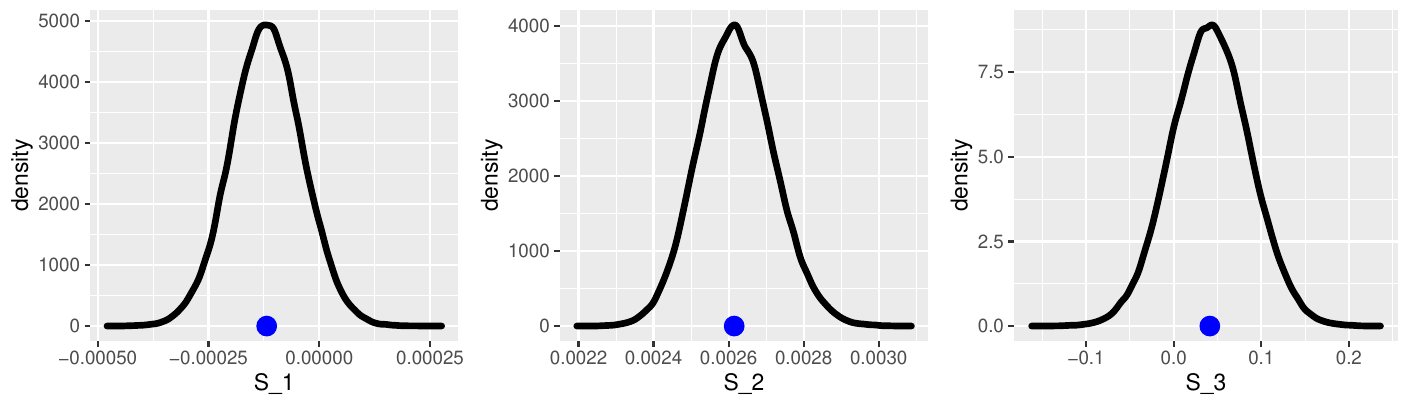}
		\caption{Posterior predictive distribution of the summary statistics when applying standard Bayesian synthetic likelihood with $S^{(1)}$ to the US-Canadian exchange rate data.  The dots show the observed values of the summary statistics. }
		\label{figsub:summary-predictive-standard-3ss}
	\end{figure}

	\begin{figure}[htbp]
		\centering
		\includegraphics[scale=0.7]{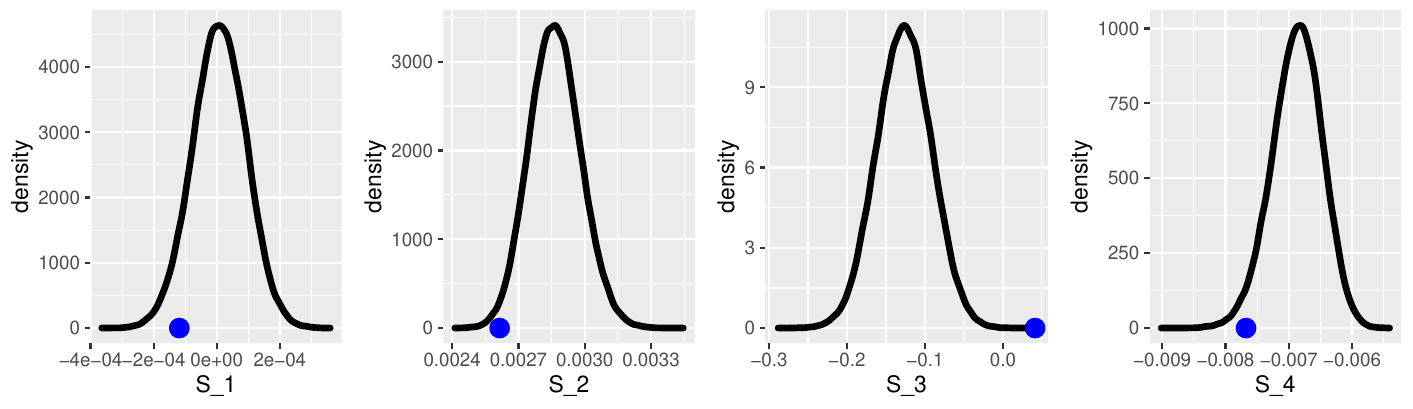}
		\caption{Posterior predictive distribution of the summary statistics when applying standard Bayesian synthetic likelihood with $S^{(2)}$ to the US-Canadian exchange rate data.  The dots show the observed values of the summary statistics. }
		
		\label{figsub:summary-predictive-standard-4ss}
	\end{figure}
	
	\begin{figure}[htbp]
		\centering
		\includegraphics[scale=0.7]{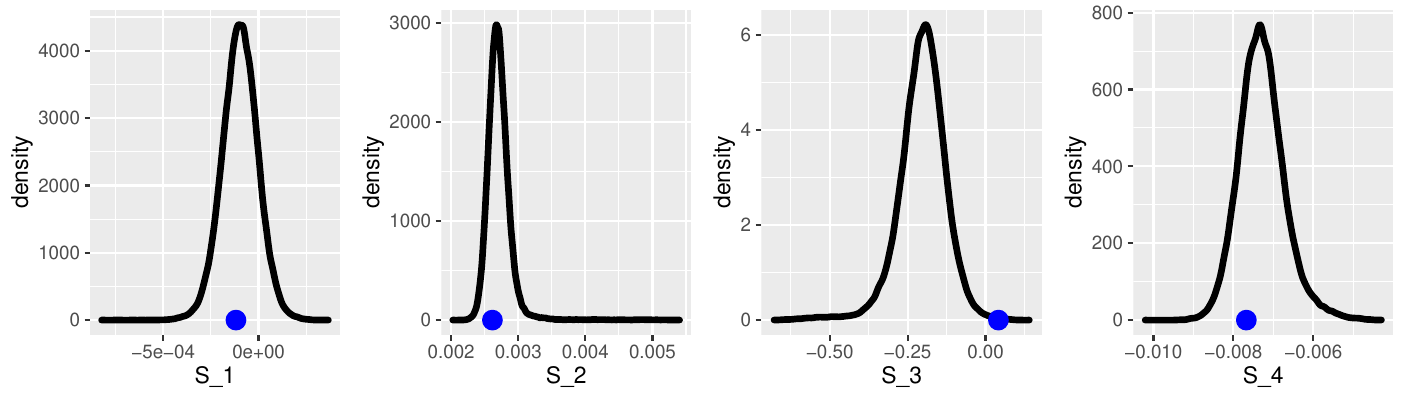}
		\caption{Posterior predictive distribution of the summary statistics when applying {robust Bayesian synthetic likelihood with $S^{(2)}$}	to the US-Canadian exchange rate data.  The dots show the observed values of the summary statistics. }
		
		\label{figsub:summary-predictive-robust-4ss}
	\end{figure}

	The posterior distribution of $\Gamma$ when using  robust Bayesian synthetic likelihood is shown in Figure \ref{fig:posterior-gamma-robust}. The  robust Bayesian synthetic likelihood method suggests that the misspecification/incompatibility is due to the models inability to recover $S_3$.  This is evident by the large departure in the posterior distribution of $\gamma_3$ compared to its prior.  The posterior predictive distribution of the summary statistics obtained with  robust Bayesian synthetic likelihood is shown in Figure \ref{figsub:summary-predictive-robust-4ss}.  By allowing for incompatibility,  robust Bayesian synthetic likelihood produces a posterior distribution that is able to recover $S_1$, $S_2$ and $S_4$ accurately, whilst placing little emphasis on $S_3$.

	\begin{figure}[htbp]
		\centerline{\includegraphics[scale=0.7]{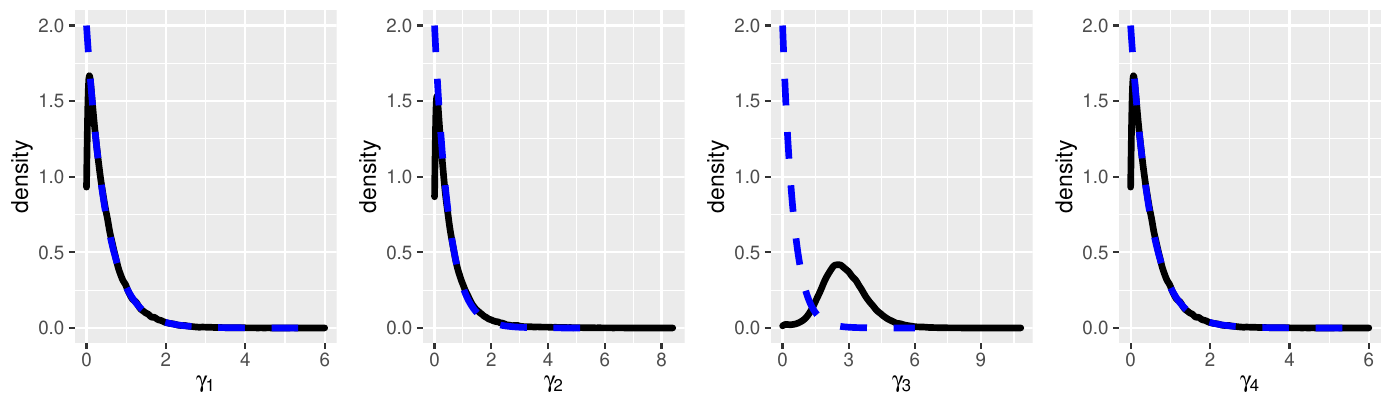}}
		\caption{\label{fig:posterior-gamma-robust} Estimated posterior distributions for the components of $\Gamma$ (solid) when applying  robust Bayesian synthetic likelihood to the US-Canadian exchange rate data based on $S^{(2)}$. Dashed lines show the priors of the components of $\Gamma$.}
	\end{figure}

	Furthermore, there is a computational benefit of  robust Bayesian synthetic likelihood.  Using only $m=30$,  robust Bayesian synthetic likelihood produces an Markov chain Monte Carlo acceptance rate of 21\%.  In contrast, standard Bayesian synthetic likelihood applied to $S^{(2)}$ using $m=300$ gives an acceptance rate of only $14\%$ with a carefully tuned random walk covariance matrix.  This is due to the fact that the observed statistic always lies in tail of the model summary statistic distribution regardless of the value of $\theta$.  In contrast,  robust Bayesian synthetic likelihood introduces variance inflation to adjust the model to be compatible even when it is not.
	
	The univariate posterior distributions of $\theta$ produced from Bayesian synthetic likelihood and  robust Bayesian synthetic likelihood based on $S^{(2)}$ are shown in Figure \ref{fig:posterior-theta-4ss}.  There is a substantial difference between the posterior distributions. The  robust Bayesian synthetic likelihood method produces a fit to the data where the model is compatible with $S_1$, $S_2$ and $S_4$, whilst largely ignoring $S_3$.  The posterior variance of $g$ is substantially larger with  robust Bayesian synthetic likelihood, and is consistent with the Bayesian synthetic likelihood adjustment results.  However, unlike Bayesian synthetic likelihood with adjustment,  robust Bayesian synthetic likelihood can shift the location of the posteriors.

	\begin{figure}[htbp]
		\centerline{\includegraphics[scale=0.7]{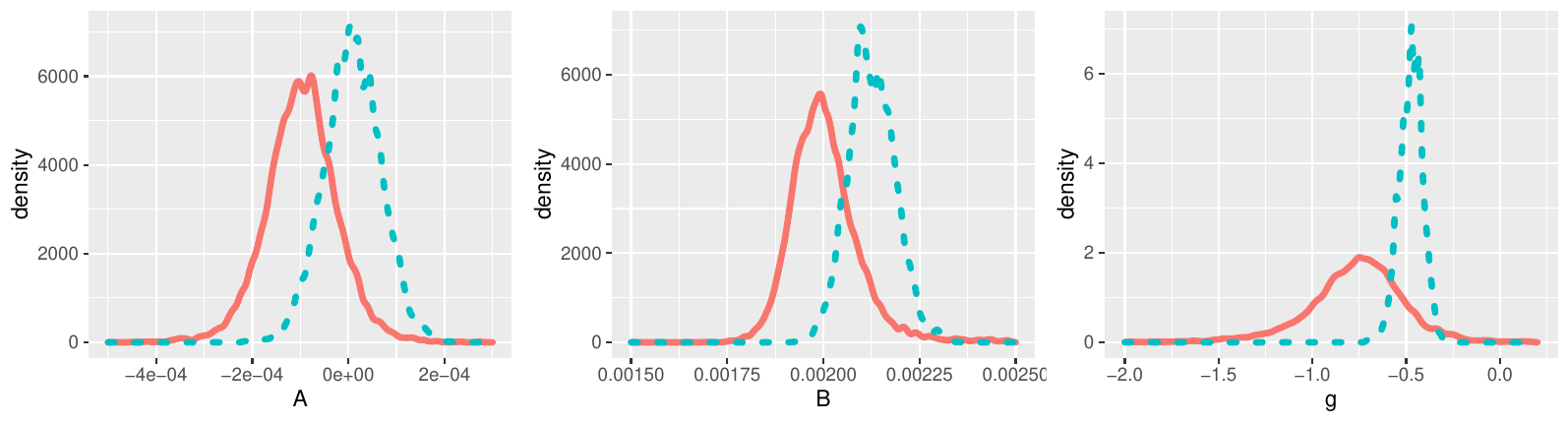}}
		\caption{\label{fig:posterior-theta-4ss} Estimated posterior distributions for the components of $\theta$ when applying Bayesian synthetic likelihood (dash) and  robust Bayesian synthetic likelihood (solid) to the US-Canadian exchange rate data based on $S^{(2)}$.}
	\end{figure}
	
	\newpage
	
	\section{Tabular Reference of Various Synthetic Likelihood Methods}
	This section gives an easy to read tabular reference that briefly compares and contrasts this current paper with several recent BSL works: \cite{price2018bayesian}, \cite{frazier2019robust},  and \cite{frazier2019bayesian}. We briefly state the goal of that paper, its empirical and theoretical contributions, and the extent to which that paper deals with model misspecification - if at all. 
	
	\newpage

	\includepdf{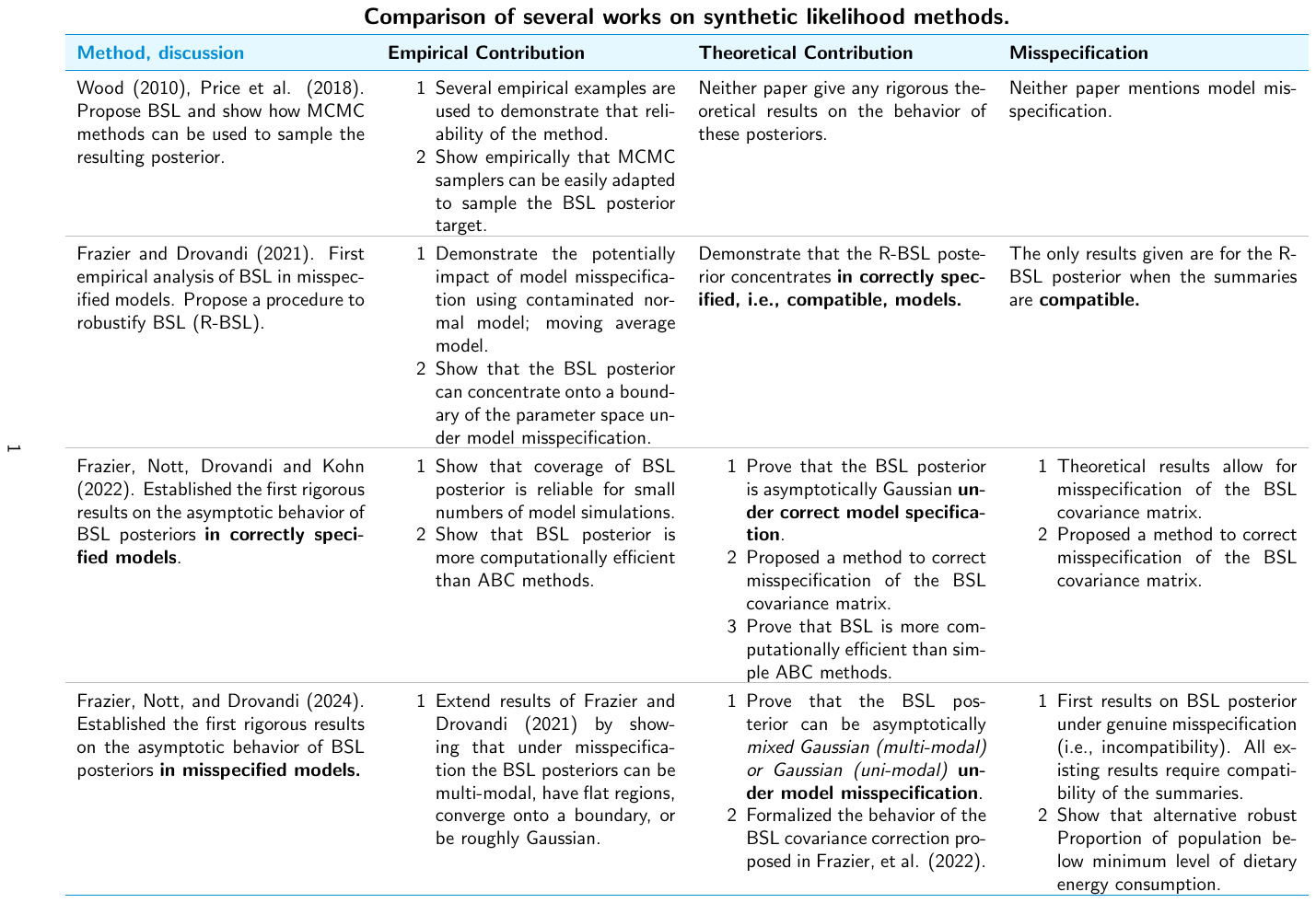}

\end{document}